%% file: matflow.tex
\documentclass[a4paper,11pt]{article}

\usepackage{a4wide}
\usepackage{graphicx}
\usepackage{amssymb,amsmath,amsthm,mathrsfs,xfrac,mathtools}
\usepackage{enumerate, enumitem}
\usepackage[title]{appendix}
\usepackage{amsfonts}
\usepackage{hyperref}
\usepackage[utf8]{inputenc}
\usepackage{color}
\usepackage{pgf,tikz}
\usetikzlibrary{arrows}
\usepackage{booktabs}
\usepackage{float}

\usepackage{tikz}
\newlength\fwidth
\usetikzlibrary{arrows}
\usetikzlibrary{decorations.pathreplacing}
\usetikzlibrary{plotmarks}
\usepackage{pgfplots}
\usepackage{pgfgantt}
\usepackage{pdflscape}
\pgfplotsset{compat=newest}
\pgfplotsset{plot coordinates/math parser=false}
\usepackage{varwidth}
\usepgfplotslibrary{fillbetween}
\pgfplotsset{compat = 1.3}
\usepgfplotslibrary{patchplots}
\usetikzlibrary{quotes,angles}

\usepackage[margin=10pt,font=normal,labelfont=bf,
labelsep=endash]{caption}

\newtheorem{proposition}{Proposition}[section]
\newtheorem{theorem}[proposition]{Theorem}
\newtheorem{lemma}[proposition]{Lemma}
\newtheorem{definition}[proposition]{Definition}
\newtheorem{corollary}[proposition]{Corollary}
\newtheorem{remark}[proposition]{Remark}
\newtheorem{lgrthm}[proposition]{Algorithm}

\numberwithin{equation}{section}
\allowdisplaybreaks[4]

\renewenvironment{proof}{\smallskip\noindent\emph{\textbf{Proof.}}%
  \hspace{1pt}}{\hspace{-5pt}{\nobreak\quad\nobreak\hfill\nobreak%
    $\square$\vspace{2pt}\par}\smallskip\goodbreak}

\newcommand{\C}[1]{\mathbf{C^{#1}}}
\newcommand{\Cc}[1]{\mathbf{C_c^{#1}}}
\renewcommand{\L}[1]{\mathbf{L^#1}}

\newcommand{\BV}{\mathbf{BV}}
\newcommand{\W}[2]{{\mathbf{W}^{#1,#2}}}
\newcommand{\modulo}[1]{{\left|#1\right|}}
\newcommand{\norma}[1]{{\left\|#1\right\|}}

\newcommand{\reali}{{\mathbb{R}}}
\newcommand{\R}{\mathbb R}

\newcommand{\interi}{{\mathbb{Z}}}

\renewcommand{\epsilon}{\varepsilon}
\renewcommand{\phi}{\varphi}

\renewcommand{\theta}{\vartheta}

\newcommand{\tv}{\mathinner{\rm TV}}

\renewcommand{\d}[1]{\mathinner{\mathrm{d}{#1}}}

\newcommand{\dt}{{\Delta t}}
\newcommand{\dx}{{\Delta x}}
\newcommand{\dy}{{\Delta y}}
\newcommand{\rh}[1]{\rho^{n}_{#1}}
\newcommand{\x}[1]{x_{#1}}

\newcommand{\vs}{v^{stat}}
\newcommand{\vd}{v^{dyn}}
\newcommand{\rv}{\boldsymbol{R}}
\newcommand{\sv}{\boldsymbol{S}}
\newcommand{\xv}{\boldsymbol{x}}
\newcommand{\yv}{\boldsymbol{y}}

\DeclareMathOperator{\sgn}{sgn}

\newcommand{\del}{\partial}

\renewcommand{\div}{\mathinner{\mathop{{\rm div}}}}

\newcommand{\be}{\begin{equation}}
\newcommand{\ee}{\end{equation}}

\definecolor{ffqqqq}{rgb}{1.,0.,0.}
\definecolor{uuuuuu}{rgb}{0.26666666666666666,0.26666666666666666,0.26666666666666666}

\makeatletter
\let\@fnsymbol\@arabic
\makeatother

\makeatletter
\newcommand\appendix@section[1]{%
\refstepcounter{section}%
\orig@section*{Appendix \@Alph\c@section: #1}%
\addcontentsline{toc}{section}{Appendix \@Alph\c@section: #1}%
}
\let\orig@section\section
\g@addto@macro\appendix{\let\section\appendix@section}
\makeatother

\title{Well--posedness of a non-local model for material flow \\on conveyor belts}

\author{Elena  Rossi\footnotemark[1] \and Jennifer K\"otz\footnotemark[2] \and Paola Goatin\footnotemark[1] \and  Simone G\"ottlich\footnotemark[2]}

\date{ }

\begin{document}
\maketitle
\footnotetext[1]{Inria Sophia
  Antipolis - M\'editerran\'ee, Universit\'e C\^ote d'Azur, Inria,
  CNRS, LJAD, 2004 route des Lucioles - BP 93, 06902 Sophia Antipolis
  Cedex, France. E-mail: \texttt{\{elena.rossi, paola.goatin\}@inria.fr}}
\footnotetext[2]{University of Mannheim, Department of Mathematics,
  68131 Mannheim, Germany. Email: \texttt{jkoetz@mail.uni-mannheim.de, goettlich@uni-mannheim.de
    }}

\begin{abstract}
In this paper, we focus on finite volume approximation schemes to solve a non-local material flow model in two space dimensions. Based on the numerical discretisation with dimensional splitting, we prove the convergence of the approximate solutions, where the main difficulty arises in the treatment of the discontinuity occurring in the flux function. In particular, we compare a Roe-type scheme to the well-established Lax-Friedrichs method and provide a numerical study highlighting the benefits of the Roe discretisation. Besides, we also prove the $\L1$-Lipschitz continuous dependence on the initial datum, ensuring the uniqueness of the solution.
\end{abstract}
  \noindent

  \medskip


  \noindent\textit{2010~Mathematics Subject Classification:} 35L65, 65M12

  \medskip

\section{Introduction}
\label{sec:intro}

In this paper, we consider the Cauchy problem for a non-local scalar conservation law in two space dimensions, namely
\begin{equation}
  \label{eq:1}
  \left\{
    \begin{array}{l@{\qquad}r@{\,}c@{\,}l}
      \del_t \rho + \nabla \cdot (\rho \, \boldsymbol{\vs} (x,y) + \rho \, \boldsymbol{\vd} (\rho)) =0,
      & (t,x,y)
      &\in
      & [0,T] \times \reali^2,
      \\
      \rho(0,x,y) = \rho_o(x,y),
      & (x,y)
      &\in
      & \reali^2,
    \end{array}
  \right.
\end{equation}
for any $T>0$, where for the weighting factor $\epsilon>0$
\begin{equation}
  \label{eq:vdyn}
  \boldsymbol{\vd}(\rho) = H (\rho - \rho_{max}) \, \boldsymbol{I}(\rho)
  \qquad
    \mbox{with }
    \boldsymbol{I}(\rho)
    =
    - \epsilon \dfrac{\nabla(\eta * \rho)}{\sqrt{1+ \norma{\nabla(\eta * \rho)}^2}}.
  \end{equation}
  Here, $H$ denotes the Heaviside function which becomes active whenever the maximal density $\rho_{max}>0$ is exceeded.

The above model was introduced in~\cite{original}, to describe the flow of objects on a conveyor belt. In particular, the unknown function $\rho=\rho(t,x,y)$ is the density of transported parts, $\boldsymbol{\vs}$ is the velocity vector field induced by the conveyor belt, which is constant in time.
So, below the maximal density, parts are transported with the velocity of the conveyor belt. On the other hand, the dynamic velocity vector field $\boldsymbol{\vd}$ is active only at high densities and accounts for colliding parts through the operator $\boldsymbol{I}(\rho)$.
The negative gradient of the convolution $\eta * \rho$ pushes the mass towards lower density regions. The particular choice of the collision operator $\boldsymbol{I}(\rho)$ was introduced in~\cite{ColomboGaravelloMercier2012} to describe crowd dynamics.

Conservation laws with non-local flux function have been recently introduced in the literature to describe transport phenomena accounting for non-local interaction effects among agents, such as road traffic flow~\cite{BlandinGoatin2016} or pedestrian dynamics~\cite{ColomboGaravelloMercier2012}.
General well-posedness results have been provided by~\cite{ACT2015} in the scalar one-dimensional case, while~\cite{ACG2015} deals with systems of non-local conservation laws in multi-space dimensions. Even if the latter result applies to our problem~\eqref{eq:1}, estimates in~\cite{ACG2015}  where obtained for general flux functions using finite volume approximate solutions constructed via Lax-Friedrichs scheme. The aim of the present paper is instead to derive sharp estimates for the Roe scheme, which is known to give less diffusive solutions and is therefore computationally more convenient, especially in the case of non-local problems in multi-D. The same remark holds for the $\L1$-stability estimates, which have been derived from scratch even if more general results are present in the literature~\cite{MercierDaSola,MercierV2}.

We finally remark that, even if the original model proposes the use of the discontinuous Heaviside function, the stability of the numerical schemes requires a smooth approximation of it. Indeed, it can be noticed that $\L\infty$-norm of the derivative $H'$ appears in the estimates (for instance in the CFL condition~\eqref{eq:CFLroe} that guarantees the $\BV$-estimates), which blow up with it.

The paper is organised as follows: for a better overview, we introduce our main results in Section~\ref{sec:MR} and start then to prove convergence of the approximate solution constructed by the Roe scheme in Section~\ref{sec:exist}. We also add a proof of the Lipschitz continuous dependence on the initial data in Section~\ref{sec:lipdep}. Sections~\ref{sec:lf} and~\ref{sec:results} are devoted to the comparison of results obtained by the Lax-Friedrichs method. The numerical results emphasise the good performance of the Roe scheme.

\section{Main results}
\label{sec:MR}

Throughout the paper, we will denote
$\mathcal{I} (r,s) := [\min\left\{r,s\right\}, \max\left\{r,s
\right\}]$, for any $r, s \in \R$.  We require the following
assumptions to hold:
\begin{enumerate}[label={$\boldsymbol{(v)}$}]
\item \label{vs} $\boldsymbol{\vs} \in \C2(\reali^2; \reali^2)$.
\end{enumerate}
\begin{enumerate}[label={$\boldsymbol{(H)}$}]
\item \label{H} The function $H$ is a smooth approximation of the
  Heaviside function such that, setting
  \begin{displaymath}
    f(r) = r \, H(r - \rho_{max}),
  \end{displaymath}
  the function $f$ has bounded derivative. In particular, we denote by
  $L_f$ the Lipschitz constant of the function $f$:
  \begin{equation}
    \label{eq:L}
    L_f  = \norma{f'}_{\L\infty(\reali)}.
  \end{equation}
\end{enumerate}
\begin{enumerate}[label={$\boldsymbol{(\eta)}$}]
\item \label{eta}
  $\eta \in (\C3 \cap \W3\infty )(\reali^2; \reali^+)$.
\end{enumerate}

Recall the definition of solution to the Cauchy problem~\eqref{eq:1},
see also~\cite{ACG2015, ACT2015, ColomboGaravelloMercier2012}. 

\begin{definition}
  \label{def:sol}
  Let $\rho_o \in \L\infty (\reali^2; \reali^+)$. A map
  $\rho: [0,T]\to \L\infty (\reali^2; \reali)$ is a solution
  to~\eqref{eq:1} if it is a Kru\v zkov solution to
  \begin{equation}
    \label{eq:Kr}
    \left\{
      \begin{array}{l@{\qquad}r@{\,}c@{\,}l}
        \del_t \rho + \nabla \cdot g(t, x, y,\rho) = 0,
        & (t,x,y)
        &\in
        & [0,T] \times \reali^2,
        \\
        \rho(0,x,y) = \rho_o(x,y),
        & (x,y)
        &\in
        & \reali^2,
      \end{array}
    \right.
  \end{equation}
  with
  $g(t,x, y,\rho) = \rho \, \boldsymbol{\vs}(x,y) - \epsilon \, \rho
  \, H(\rho - \rho_{\max}) \dfrac{\nabla(\eta * \rho)}{\sqrt{1+
      \norma{\nabla(\eta * \rho)}^2}}$.
\end{definition}
\noindent Above, for the definition of Kru\v zkov solution we refer
to~\cite[Definition~1]{Kruzkov}.

\begin{theorem}
  \label{thm:main}
  Let $\rho_o \in (\L\infty \cap \BV) (\reali^2; \reali^+)$. Let
  assumptions~\ref{vs}, \ref{H} and~\ref{eta} hold. Then, for all
  $T>0$, there exists a unique weak entropy solution
  $\rho \in (\L\infty \cap \BV) ([0,T] \times \reali^2; \reali^+)$ to
  problem~\eqref{eq:1}. Moreover, the following estimates hold: for
  all $t \in [0,T]$
  \begin{align*}
    \norma{\rho(t)}_{\L1} =  \
    & \norma{\rho_o}_{\L1},
    \\
    \norma{\rho(t)}_{\L\infty} \leq  \
    & \norma{\rho_o}_{\L\infty} \, e^{\mathcal{C}_{\infty} \, t},
    \\
    \tv(\rho(t)) \leq \
    & \tv(\rho_o) \, e^{2 \, t \, \mathcal{K}_1}
      + \frac{2\, \mathcal{K}_2}{\mathcal{K}_1} \left(e^{2 \, t \, \mathcal{K}_1} - 1\right),
    \\
    \norma{\rho(t) - \rho(t - \tau)}_{\L1} \leq \
    &  2 \, \mathcal{C}_t (t) \, \tau, \quad \mbox { for } \tau >0,
  \end{align*}
  where $\mathcal{C}_\infty$ is defined in~\eqref{eq:Cinf},
  $\mathcal{K}_1$ is defined in~\eqref{eq:K1defroe}, $\mathcal{K}_2$
  is defined in~\eqref{eq:K2defroe} and $\mathcal{C}_t(t)$ is as
  in~\eqref{eq:Ct}.
\end{theorem}
The proof of Theorem~\ref{thm:main} is standard,
see~\cite[Theorem~2.3]{ACG2015}, which refers to~\cite{sanders}. The
uniqueness is ensured by Proposition~\ref{prop:lipdep}, which provides
the Lipschitz continuous dependence estimate of solutions
to~\eqref{eq:1} on the initial data.

The bounds presented in Theorem~\ref{thm:main} are obtained by passing
to the limit in the corresponding discrete bounds.

\section{Existence}
\label{sec:exist}

Introduce the uniform mesh of width $\dx$ along the $x$-axis and $\dy$
along the $y$-axis, and a time step $\dt$ subject to a CFL condition,
specified later on. For $k \in \interi$ set
\begin{align*}
  x_k = \ & (k-1/2) \dx,
  &
    y_k = \ & (k-1/2) \dy,
  \\
  x_{k+1/2} = \ & k \dx,
  &
    y_{k+1/2} = \ & k \dy,
\end{align*}
where $(x_{i+\frac12},y_j)$ and $(x_{i}, y_{j+1/2})$ denote the cells
interfaces and $(x_i, y_j)$ are the cells centres. Set
$N_T = \lfloor T/\dt \rfloor$ and let $t^n = n \, \dt$ ,
$n=0, \ldots, N_T$, be the time mesh. Set $\lambda_x = \dt/\dx$ and
$\lambda_y = \dt / \dy$, and let $\alpha, \beta \geq 1$ be the
viscosity coefficients.

For the sake of shortness, sometimes we will use also the notation
$x_{i,j} = (x_i, y_j)$.

We approximate the initial datum as follows: for $i, \, j \in \interi$
\begin{displaymath}
  \rho_{i,j}^0 = \frac{1}{\dx \, \dy} \int_{x_{i-\sfrac12}}^{x_{i+\sfrac12}}
  \int_{y_{i-\sfrac12}}^{y_{i+\sfrac12}} \rho_o(x,y) \d{x} \d{y},
\end{displaymath}
and we define a piece-wise constant solution to~\eqref{eq:1} as
\begin{equation}
  \label{eq:rhodelta}
  \rho_{\Delta} (t,x,y ) =  \rh{i,j}
  \quad \mbox{ for } \quad
  \left\{
    \begin{array}{r@{\;}c@{\;}l}
      t & \in &[t^n, t^{n+1}[ \,,
      \\
      x & \in & [x_{i-1/2}, x_{i+1/2}[ \,,
      \\
      y & \in & [y_{j-1/2}, y_{j+1/2}[ \,,
    \end{array}
  \right.
  \quad \mbox{ where } \quad
  \begin{array}{r@{\;}c@{\;}l}
    n & = & 0, \ldots, N_T-1,
    \\
    i & \in & \interi,
    \\
    j & \in & \interi.
  \end{array}
\end{equation}
through a modified Roe-type scheme with dimensional splitting, exactly
as in~\cite{original}:
\begin{lgrthm}
  \label{alg:2}
  \begin{align}
    \label{eq:fx1}
    & V_1 (x,y,u,w) = \vs_1(x,y) \, u + \min\{0, \, \vs_1(x,y)\} (w-u)
    \\
    \label{eq:fx2}
    & F (u,w, J(t,x,y)) =
      J (t, x,y) \, f(u) + \min\{0, J(t,x,y)\} \left(f(w) -f(u)\right)
    \\ \label{eq:gx1}
    & V_2 (x,y,u,w) = \vs_2(x,y) \, u + \min\{0, \, \vs_2(x,y)\} (w-u)
    \\
    &\texttt{for } n=0,\ldots N_T-1 \nonumber
    \\
    \label{eq:scheme1roe}
    & \qquad
      \rho_{i,j}^{n+1/2} = \rh{i,j} - \lambda_x \left[
      V_1(\x{i+\sfrac12,j},\rh{i,j}, \rh{i+i,j})
      -
      V_1(\x{i-\sfrac12,j},\rh{i-1,j}, \rh{i,j})
      \right.
    \\ \nonumber
    &
      \qquad\quad\qquad
      \left.
      + F (\rh{i,j}, \rh{i+1,j}, J^n_1(\x{i+\sfrac12,j})
      - F (\rh{i-1,j}, \rh{i,j}, J^n_1 ( \x{i-\sfrac12,j}))\right]
    \\ \label{eq:scheme2roe}
    & \qquad \rho_{i,j}^{n+1} =
      \rho_{i,j}^{n+1/2} - \lambda_y \left[
      V_2(\x{i, j+\sfrac12},\rh{i,j}, \rh{i,j+1})
      -
      V_2(\x{i, j-\sfrac12},\rh{i,j-1}, \rh{i,j})
      \right.
    \\ \nonumber
    &
      \qquad\quad\qquad
      \left.
      + F (\rh{i,j}, \rh{i,j+1}, J^n_2(\x{i, j+\sfrac12})
      - F (\rh{i,j-1}, \rh{i,j}, J^n_2 ( \x{i,j-\sfrac12}))\right]
    \\
    & \texttt{end} \nonumber
  \end{align}
\end{lgrthm}

Above, we set
$\boldsymbol{I} (\rho^n) (x,y) = \left(J_1^n (x,y), \, J_2^n
  (x,y)\right)$ and the convolution products are computed through the
following quadrature formula, for $k=1,2$,
\begin{equation} \label{eq:conv}
  (\partial_i \eta * \rho ) (x_i,y_j) =
  \dx \, \dy \sum_{k, \ell \in \interi} \rho_{k,\ell} \, \partial_i \eta(x_{i-k}, y_{j-\ell}),
\end{equation}
where $\partial_1 \eta=\partial_x\eta$ and
$\partial_2\ eta=\partial_y\eta$.
Remark that the choice of evaluating the numerical flux at $t^n$ for both fractional steps allows to save computational time, because the convolution products~\eqref{eq:conv} are computed only once per time step.

Introduce the following notation, which is of use below:
\begin{equation}
  \label{eq:notvJ}
  v_{i+1/2} =  \vs_1 (x_{i+1/2}).
\end{equation}

\input{Roe}

\subsection{Discrete entropy inequalities}
\label{sec:dei}
Following~\cite{ACG2015}, see also~\cite{CMfs, CMmonotone}, introduce
the following notation: for $i,j \in \interi$, $n=0, \ldots, N_T-1$
and $\kappa \in \reali$,
\begin{align*}
  \Phi_{i+1/2,j}^n (u,v) = \
  &  V_1(x_{i+1/2,j},u \vee \kappa, v \vee \kappa)
    +  F (u \vee \kappa, v \vee \kappa, J^n_1 (x_{i+1/2,j}))
  \\
  & - V_1(x_{i+1/2,j}, u \wedge \kappa, v \wedge \kappa)
    - F(u \wedge \kappa, v \wedge \kappa, J^n_1 (x_{i+1/2,j})),
  \\
  \Gamma_{i,j+1/2}^n (u,v) = \
  &  V_2 (x_{i,j+1/2},u \vee \kappa, v \vee \kappa)
    + F (u \vee \kappa, v \vee \kappa, J^n_2 (x_{i,j+1/2}))
  \\
  &
    - V_2(x_{i,j+1/2}, u \wedge \kappa, v \wedge \kappa)
    - F(u \wedge \kappa, v \wedge \kappa, J^n_2 (x_{i,j+1/2})),
\end{align*}
with $V_1$, $V_2$ and $F$ defined as in~\eqref{eq:fx1}, \eqref{eq:gx1}
and~\eqref{eq:fx2} respectively.

\begin{lemma} {\bf (Discrete entropy condition)}
  \label{lem:die}
  Fix $\rho_o \in (\L\infty \cap \BV) (\reali^2;
  \reali^+)$. Let~\ref{vs}, \ref{H}, \ref{eta}, \eqref{eq:CFLroe-v2}
  hold. Then, the solution $\rho_\Delta$ in~\eqref{eq:rhodelta}
  constructed through Algorithm~\ref{alg:2} satisfies the following
  discrete entropy inequality: for $i,j \in \interi$, for
  $n=0, \ldots, N_T-1$ and $\kappa \in \reali$,
  \begin{align*}
    \modulo{\rho^{n+1}_{i,j} - \kappa} - \modulo{\rh{i,j} - \kappa}
    + \lambda_x \, \left(
    \Phi^n_{i+1/2,j}(\rh{i,j},\rh{i+1,j}) - \Phi^n_{i-1/2,j}(\rh{i-1,j},\rh{i,j})
    \right)
    &
    \\
    + \lambda_x \, \sgn(\rho^{n+1/2}_{i,j} - \kappa)
    \left(\vs_1(x_{i+1/2,j}) - \vs_1(x_{i-1/2,j}) \right) \kappa
    &
    \\
    + \lambda_x \, \sgn(\rho^{n+1/2}_{i,j} - \kappa)
    \left( J_1^n(x_{i+1/2,j})  - J_1^n(x_{i-1/2,j}) \right) f(\kappa)
    &
    \\
    + \lambda_y \, \left(
    \Gamma^n_{i,j+1/2}(\rho^{n+1/2}_{i,j},\rho^{n+1/2}_{i,j+1}) - \Gamma^n_{i,j-1/2}(\rho^{n+1/2}_{i,j-1},\rho^{n+1/2}_{i,j})
    \right)
    &
    \\
    + \lambda_y \, \sgn(\rho^{n+1}_{i,j} - \kappa)
    \left(\vs_2(x_{i,j+1/2}) - \vs_2(x_{i,j-1/2}) \right) \kappa
    &
    \\
    + \lambda_y \, \sgn(\rho^{n+1}_{i,j} - \kappa)
    \left( J_2^n(x_{i,j+1/2})  - J_2^n(x_{i,j-1/2}) \right) f(\kappa)
    & \leq 0.
  \end{align*}
\end{lemma}
\noindent The proof is omitted, being entirely analogous to that of~\cite[Proposition~2.8]{ACT2015}, see also~\cite[Lemma~2.8]{ACG2015}.

\section{Lipschitz continuous dependence on initial data}
\label{sec:lipdep}

\begin{proposition}
  \label{prop:lipdep}
  Fix $T>0$. Let~\ref{vs}, \ref{H} and \ref{eta} hold. Let
  $\rho_o, \, \sigma_o \in (\L\infty \cap \BV) (\reali^2;
  \reali^+)$. Call $\rho$ and $\sigma$ the corresponding solutions
  to~\eqref{eq:1}. Then the following estimate holds:
  \begin{displaymath}
     \norma{\rho (t) - \sigma (t)}_{\L1 (\reali^2)} \leq
     \norma{\rho_o - \sigma_o}_{\L1 (\reali^2)} \, e^{t \, A (t)}.
  \end{displaymath}
  with $A(t)$ defined in~\eqref{eq:As}.
\end{proposition}

\begin{proof}
  In the rest of the proof, to avoid heavy notation, we will denote
  pairs in $\reali^2$ by $\xv$ or $\yv$. Introduce the following
  notation:
  \begin{align}
    \label{eq:RS}
    \rv(t,\xv) = \
    & \left(\boldsymbol{I}(\rho(t))\right)\!(\xv),
    &
      \sv(t,\xv) = \
    & \left(\boldsymbol{I}(\sigma(t))\right)\!(\xv).
  \end{align}
  The idea is to apply the \emph{doubling of variables method}
  introduced by Kru\v zkov in~\cite{Kruzkov}, exploiting in particular
  the proof of~\cite[Lemma~4]{NARWA2019}. There, a flux of the form
  $f(t,x, \rho) \, V (t,x)$ is taken into account, with $x\in \reali$,
  the proof being valid also in the multidimensional case,
  i.e.~$x\in\reali^n$. Therefore we are going to use this result for
  what concerns the part of the flux of type $f(\rho) \,
  \rv(t,\xv)$.

  For the sake of completeness, we recall that a flux function of type
  $l (x)\, g (\rho)$ is considered in~\cite{KarlsenRisebro2003}, with
  $x\in \reali^n$, and the proof of~\cite[Lemma~4]{NARWA2019} follows
  the lines of that of~\cite[Theorem~1.3]{KarlsenRisebro2003}. Thus,
  here we are adding the dependence on time to the function $l(x)$
  considered in~\cite{KarlsenRisebro2003}.

  Let $\phi \in \Cc1 (]0,T[ \times \reali^2;\reali^+)$ be a test
  function as in the definition of solution by Kru\v zkov. Let
  $Y \in \Cc\infty (\reali; \reali^+)$ be such that
  \begin{align*}
    Y (z)= \
    & Y (-z),
    &
      Y (z) =\
    & 0 \mbox{ for } \modulo{z}\geq1,
    & \int_{\reali}Y (z)\d z=1.
  \end{align*}
  Define, for $h>0$, $Y_h (z)= \frac1h Y (\frac{z}{h})$. Clearly,
  $Y_h \in \Cc\infty (\reali; \reali^+)$, $Y_h (z)=Y_h (-z)$,
  $Y_h (z)=0$ for $\modulo{z}\geq h$, $\int_\reali Y_h (z) \d z=1$ and
  $Y_h \to \delta_0$ as $h\to 0$, $\delta_0$ being the Dirac delta in
  $0$. Define moreover
  \begin{displaymath}
    \psi_h (t,\xv, s, \yv) =
    \phi\left(\frac{t+s}2, \frac{\xv + \yv}2\right) \, Y_h (t-s) \, \prod_{i=1}^2 Y_h (x_i - y_i).
  \end{displaymath}
  Introduce the space $\Pi_T = \ ]0,T[ \times \reali^2$ and, from the
  definition of solution, derive the following entropy inequalities
  for $\rho=\rho (t,\xv)$ and $\sigma=\sigma (s, \yv)$:
  \begin{align*}
    \iiiint\limits_{\Pi_T \times \Pi_T} \left\{
    \modulo{\rho-\sigma} \, \partial_t \psi_h
    + \modulo{\rho-\sigma}\, \nabla_{\xv} \boldsymbol{\vs} (\xv) \cdot \nabla_{\xv}\psi_h
    + \sgn (\rho-\sigma)\left(f (\rho)- f (\sigma)\right)  \rv (t, \xv)  \cdot  \nabla_{\xv}\psi_h
    \right.
    \\
    \left.
    - \sgn (\rho-\sigma) \, \sigma\, \div_{\xv} \boldsymbol{\vs} (\xv) \, \psi_h
    - \sgn (\rho-\sigma) \, f (\sigma) \,\div_{\xv} \rv (t, \xv) \, \psi_h
    \right.\} \d\xv\d{t} \d\yv \d{s} \geq 0,
    \\[6pt]
    \iiiint\limits_{\Pi_T \times \Pi_T} \left\{
    \modulo{\sigma-\rho} \, \partial_s \psi_h
    + \modulo{\sigma-\rho}\, \nabla_{\yv} \boldsymbol{\vs} (\yv) \cdot \nabla_{\yv}\psi_h
    + \sgn (\sigma-\rho)\left(f (\sigma)- f (\rho)\right)  \sv (s, \yv)  \cdot  \nabla_{\yv}\psi_h
    \right.
    \\
    \left.
    - \sgn (\sigma-\rho) \, \rho\, \div_{\yv} \boldsymbol{\vs} (\yv) \, \psi_h
    - \sgn (\sigma-\rho) \, f (\rho) \,\div_{\yv} \sv (s, \yv) \, \psi_h
    \right.\} \d\xv\d{t} \d\yv \d{s} \geq 0.
  \end{align*}
  Sum the two inequalities above and rearrange the terms therein,
  following the proof of~\cite[Theorem~1]{Kruzkov} for what concerns
  the linear part of the flux and the proof
  of~\cite[Lemma~4]{NARWA2019}, see
  also~\cite[Theorem~1.3]{KarlsenRisebro2003}, for the other part:
  \begin{align*}
    \iiiint\limits_{\Pi_T\times\Pi_T}\left\{
    \modulo{\rho-\sigma} (\partial_t \psi_h + \partial_s \psi_h)
    + \sgn (\rho-\sigma) (\rho \, \boldsymbol{\vs} (\xv) - \sigma \boldsymbol{\vs} (\yv))
    \cdot (\nabla_{\xv} \psi_h + \nabla_{\yv}\psi_h)
    \right.&
    \\[-4pt]
    +\sgn (\rho-\sigma)\, \sigma \left[
    (\boldsymbol{\vs} (\yv) - \boldsymbol{\vs} (\xv)) \cdot \nabla_{\xv}\psi_h
    - \div_{\xv} \boldsymbol{\vs} (\xv) \, \psi_h
    \right]&
    \\
    +\sgn (\rho-\sigma) \,\rho \left[
    (\boldsymbol{\vs} (\yv) - \boldsymbol{\vs} (\xv)) \cdot \nabla_{\yv}\psi_h
    + \div_{\yv} \boldsymbol{\vs} (\yv) \, \psi_h
    \right]&
    \\
    + \sgn (\rho-\sigma) \left(f (\rho) \, \rv (t, \xv) - f (\sigma) \, \sv (s, \yv)\right)
    \cdot (\nabla_{\xv} \psi_h + \nabla_{\yv}\psi_h)&
    \\
    + \sgn (\rho-\sigma) \, f (\sigma) \left[
    \left(\sv (s, \yv) - \rv (t,\xv)\right) \cdot \nabla_{\xv} \psi - \div_{\xv} \rv (t, \xv) \, \psi_h
    \right]&
    \\ \left.
    + \sgn (\rho-\sigma) \, f (\rho) \left[
    \left(\sv (s, \yv) - \rv (t,\xv)\right) \cdot \nabla_{\yv} \psi + \div_{\yv} \sv (s, \yv) \, \psi_h
    \right] \right\}\d\xv \d t \d\yv \d s
    & \geq 0.
  \end{align*}
  Let $h \to 0$, which gives
  \begin{align*}
    \iint\limits_{\Pi_T}\left\{
    \modulo{\rho-\sigma} \partial_t \phi_h
    + \sgn (\rho-\sigma) \left[
    (\rho - \sigma)  \, \boldsymbol{\vs} (\xv) + \sv (t, \xv) \left(f (\rho) - f (\sigma) \right)
    \right]\cdot \nabla_{\xv} \phi
    \right.&
    \\[-4pt]
    + \sgn (\rho-\sigma) \, f (\rho)
    \div_{\xv} \left(\sv (t, \xv) - \rv (t, \xv)\right) \, \phi
    \\
    \left.
    + \sgn (\rho-\sigma) \, f' (\rho)
    \left(\sv (t, \xv) - \rv (t, \xv)\right) \partial_x\rho (t,\xv) \, \phi
    \right\}\d\xv \d t \d\yv \d s
    & \geq 0.
  \end{align*}
  Choosing a suitable test function $\phi$ leads to
  \begin{align*}
    \int_{\reali^2}\modulo{\rho (t,\xv) - \sigma (t,\xv)} \d\xv
    - \int_{\reali^2}\modulo{\rho (\tau,\xv) - \sigma (\tau,\xv)} \d\xv&
    \\
    +\int_\tau^t \int_{\reali^2} \modulo{\div (\sv (s,\xv) - \rv (s,\xv))} \, f (\rho (s,\xv)) \d\xv \d s&
    \\
    + \int_\tau^t L_f \, \norma{\sv (s) - \rv (s)}_{\L\infty (\reali^2)} \tv (\rho (s)) \d s & \geq 0.
  \end{align*}
  Observe that, following~\cite[Lemma~4.1]{parahyp}, the following bounds hold
  \begin{align*}
    \norma{\sv(s) - \rv (s)}_{\L\infty(\reali^2)} \leq \
    & 2 \, \epsilon \norma{\nabla \eta}_{\L\infty}\norma{\rho(s) - \sigma(s)}_{\L1(\reali^2)},
    \\
    \norma{\div(\sv (s)- \rv (s))}_{\L\infty(\reali^2)} \leq \
    & \epsilon \norma{\rho(s)- \sigma(s)}_{\L1(\reali^2)}
      \norma{\Delta \eta}_{\L\infty}
      \left( 1 + \norma{\sigma(s)}_{\L1(\reali^2)} \norma{\nabla \eta}_{\L1}\right).
  \end{align*}
  Thus, letting $\tau \to 0$ and exploiting the bounds on $\rho$ and $\sigma$ given by
  Theorem~\ref{thm:main}, as well as $f(r) \leq r$, we get
  \begin{align*}
    & \int_{\reali^2} \modulo{\rho (t,\xv) - \sigma (t,\xv)} \d\xv
    \\
    \leq  \
    & \int_{\reali^2}\modulo{\rho_o (\xv) - \sigma_o (\xv)} \d\xv
      + 2 \, \epsilon \norma{\nabla \eta}_{\L\infty}\norma{\rho_o}_{\L1 (\reali^2)} \int_0^t \int_{\reali^2}
      \modulo{\rho (s,\xv) - \sigma (s,\xv)} \d\xv \d s
    \\
    & + \epsilon \, L_f \norma{\Delta \eta}_{\L\infty}
      \left( 1 + \norma{\sigma_o}_{\L1(\reali^2)} \norma{\nabla \eta}_{\L1}\right)
      \int_0^t \tv (\rho (s)) \left(\int_{\reali^2}   \modulo{\rho (s,\xv) - \sigma (s,\xv)} \d\xv \right)\d s
    \\
    = \
    & \int_{\reali^2}\modulo{\rho_o (\xv) - \sigma_o (\xv)} \d\xv
      + \int_0^t A (s)  \left(\int_{\reali^2}   \modulo{\rho (s,\xv) - \sigma (s,\xv)} \d\xv\right) \d s,
  \end{align*}
  with
  \begin{equation}
    \label{eq:As}
    A (s) = 2 \, \epsilon \norma{\nabla \eta}_{\L\infty}\norma{\rho_o}_{\L1 (\reali^2)}
    + \epsilon \, L_f \norma{\Delta \eta}_{\L\infty}
    \left( 1 + \norma{\sigma_o}_{\L1(\reali^2)} \norma{\nabla \eta}_{\L1}\right) \tv (\rho (s)).
  \end{equation}
  An application of Gronwall Lemma, together with
  \begin{displaymath}
    \int_0^t A(s)\, \exp\left(\int_s^tA (r)\d r\right)\d s = -1 + \exp \left(\int_0^t A (s) \d s\right),
  \end{displaymath}
  yields the desired estimate
  \begin{displaymath}
    \norma{\rho (t) - \sigma (t)}_{\L1 (\reali^2)} \leq
     \norma{\rho_o - \sigma_o}_{\L1 (\reali^2)} \, e^{t \, A (t)}.
  \end{displaymath}
 \end{proof}

  \begin{remark}
    We can interpret $\rho$ and $\sigma$ as solutions of the following
    Cauchy problems:
    \begin{displaymath}
      \left\{
        \begin{array}{l}
          \del_t \rho + \nabla \cdot \boldsymbol{g}(t, \xv, \rho) = 0,
          \\
          \rho(0,\xv) = \rho_o(\xv),
        \end{array}
      \right.
      \qquad
      \left\{
        \begin{array}{l@{\qquad}r@{\,}c@{\,}l}
          \del_t \sigma + \nabla \cdot \boldsymbol{h}(t, \xv, \sigma) = 0,
          & (t,\xv)
          &\in
          & [0,T] \times \reali^2,
          \\
          \sigma(0,\xv) = \sigma_o(\xv),
          & \xv
          &\in
          & \reali^2,
        \end{array}
      \right.
    \end{displaymath}
    where
    \begin{align*}
      \boldsymbol{g} (t,\xv, r) = \
      & r \, \boldsymbol{\vs}(\xv) + f(r) \, \rv(t,\xv),
      & \boldsymbol{h} (t,\xv, r) = \
      & r \, \boldsymbol{\vs}(\xv) + f(r) \, \sv(t,\xv),
    \end{align*}
    so that the $\L1$ distance between the solutions at time $t>0$ can
    be estimated by~\cite[Proposition~2.10]{MercierDaSola}, see also
    the refinement in~\cite[Proposition~2.9]{MercierV2}. However,
    making use of the explicit expression of the flux in the present
    case, one may see that the bound provided by
    Proposition~\ref{prop:lipdep} is sharper than that coming
    from~\cite[Proposition~2.9]{MercierV2}.
  \end{remark}

  \input{LF}

  \input{numericalresults.tex}

\section*{Acknowledgments}
This work was financially supported by the project ``Non-local conservation laws for engineering applications'' co-funded by DAAD (Project-ID 57445223) and the PHC Procope (Project no.~42503RM) and the DFG project GO 1920/7-1. Elena Rossi was  partially supported by University of Mannheim and acknowledges its hospitality during the accomplishment of this work.

  \begin{appendices}
    \section{Technical Lemma}
    \label{sec:teclem}
    \begin{lemma}
      Let $\eta \in (\C3 \cap \W3\infty) (\reali^2; \reali)$. Then,
      for $n=0, \ldots, N_T$, for $i, j \in \interi$, the following
      estimates hold:
      \begin{align}
        \label{eq:Jinf}
        \norma{J^n_k}_{\L\infty} \leq \
        & \epsilon \quad \mbox{ for } k=1,2,
        \\
        \label{eq:Jx}
        \modulo{J^n_1 (x_{i+1/2,j}) - J_1^n (x_{i-1/2,j})}
        \leq \
        & 2 \, \epsilon \, \dx \,  \norma{\nabla^2 \eta}_{\L\infty} \norma{\rho^n}_{\L1},
        \\
        \label{eq:J1y}
        \modulo{J_1^n (x_{i+1/2,j}) - J_1^n(x_{i+1/2,j+1})}
        \leq \
        & 2 \, \epsilon \, \dy \,  \norma{\nabla^2 \eta}_{\L\infty} \norma{\rho^n}_{\L1},
        \\
        \nonumber
        \modulo{J^n_2 (x_{i,j+1/2}) - J_2^n (x_{i,j-1/2})}
        \leq \
        & 2 \, \epsilon \, \dy \,  \norma{\nabla^2 \eta}_{\L\infty} \norma{\rho^n}_{\L1},
        \\
        \nonumber
        \modulo{J^n_2 (x_{i+1,j+1/2}) - J_2^n (x_{i,j+1/2})}
        \leq \
        & 2 \, \epsilon \, \dx \, \norma{\nabla^2 \eta}_{\L\infty} \norma{\rho^n}_{\L1},
        \\
        \label{eq:Jtripla}
        \modulo{J^n_1 (x_{i+3/2,j}) -2 \,J^n_1 (x_{i+1/2,j}) - J_1^n (x_{i-1/2,j})}
        \leq \
        &  2 \, \epsilon \,  (\dx)^2 \left(
          c_1  \norma{\rho^n}_{\L1} + c_2 \norma{\rho^n}^2_{\L1}
          \right),
        \\
        \nonumber
        \modulo{J^n_2 (x_{i,j+3/2}) -2 \,J^n_2 (x_{i,j+1/2}) - J_2^n (x_{i,j-1/2})}
        \leq \
        &  2 \, \epsilon \,  (\dy)^2 \left(
          c_1  \norma{\rho^n}_{\L1} + c_2 \norma{\rho^n}^2_{\L1}
          \right),
        \\
        \label{eq:Jtriplabis}
        \left|J_1^n (x_{i+1/2,j}) \!-\! J_1^n(x_{i+1/2,j+1})\! -\! J_1^n (x_{i-1/2,j}) \!-\!
        J_1^n \right.\! &\! \left.(x_{i-1/2,j+1})\right|
                          \leq \ 2 \, \epsilon \,  \dx \, \dy  \, C ,
        \\
        \nonumber
        \left|J_2^n (x_{i,j+1/2}) \!-\! J_2^n(x_{i+1,j+1/2})\! -\! J_2^n (x_{i,j-1/2}) \!-\!
        J_2^n \right. \! &\! \left.(x_{i+1,j-1/2})\right|
                           \leq \ 2 \, \epsilon \,  \dx \, \dy  \, C ,
      \end{align}
      where we set
      \begin{align}
        \label{eq:c12}
        C = \
        &
          c_1 \norma{\rho^n}_{\L1} + c_2 \norma{\rho^n}^2_{\L1},
        &
          c_1 = \
        & 2 \, \norma{\nabla^3 \eta}_{\L\infty},
        &
          c_2 = \
        & 3 \, \norma{\nabla^2 \eta}^2_{\L\infty}.
      \end{align}
    \end{lemma}
    \begin{proof}
      The proof of~\eqref{eq:Jinf} is immediate.

      Pass now to~\eqref{eq:Jx}. For the sake of simplicity, introduce
      the following notation:
      \begin{align*}
        D_+ =\
        & \sqrt{1 + \norma{(\nabla \eta * \rho^n) (x_{i+1/2}, y_j)}^2},
        & D_- =\
        & \sqrt{1 + \norma{(\nabla \eta * \rho^n) (x_{i-1/2}, y_j)}^2}.
      \end{align*}
      Hence,
      \begin{align}
        \label{eq:2original}
        & \modulo{J^n_1 (x_{i+1/2}, y_j) - J_1^n
          (x_{i-1/2},y_j)}
        \\
        \nonumber = \
        & \epsilon \, \modulo{\frac{\dx \, \dy}{D_+} \,
          \sum_{k, \ell \in \interi} \rh{k,\ell} \, \partial_1
          \eta(x_{i+1/2-k},y_{j-\ell}) - \frac{\dx \, \dy}{D_-} \,
          \sum_{k, \ell \in \interi} \rh{k,\ell} \, \partial_1
          \eta(x_{i-1/2-k},y_{j-\ell})}
        \\ \label{eq:2} \leq \
        & \epsilon\,
          \modulo{ \frac{\dx \, \dy}{D_+} \, \sum_{k, \ell \in \interi}
          \rh{k,\ell} \, \left(
          \partial_1 \eta(x_{i+1/2-k},y_{j-\ell}) -
          \partial_1 \eta(x_{i-1/2-k},y_{j-\ell}) \right) }
        \\\label{eq:2b}
        & + \epsilon \, \dx \, \dy \,
          \modulo{\frac{1}{D_+} - \frac{1}{D_-}} \, \sum_{k, \ell \in
          \interi} \modulo{\rh{k,\ell}} \modulo{\partial_1
          \eta(x_{i-1/2-k},y_{j-\ell})}.
      \end{align}
      Consider~\eqref{eq:2}: since $D_+ \geq 1$ and
      \begin{displaymath}
        \modulo{ \partial_1 \eta(x_{i+1/2-k},y_{j-\ell})- \partial_1 \eta(x_{i-1/2-k},y_{j-\ell})}
        \leq
        \int_{x_{i-1/2-k}}^{x_{i+1/2-k}}\modulo{\partial_{11}^2 \eta (x,y_{j-\ell})} \d{x},
      \end{displaymath}
      we obtain
      \begin{equation}
        \label{eq:2ok}
        [\eqref{eq:2}]
        \leq
        \epsilon \, \dx \, \norma{\partial_{11}^2 \eta}_{\L\infty} \norma{\rho^n}_{\L1}.
      \end{equation}
      On the other hand, to estimate~\eqref{eq:2b}, compute
      \begin{displaymath}
        \modulo{\frac{1}{D_+} - \frac{1}{D_-}} = \frac{\modulo{D_+ - D_-}}{D_+ \, D_-}.
      \end{displaymath}
      Introduce $a(x) = \nabla \eta*\rho^n (x)$ and
      $b(z) = (1+\norma{z}^2)^{1/2}$, for $z \in \reali^2$. In
      particular compute
      $b' (z) = \dfrac{\norma{z}}{(1+\norma{z}^2)^{1/2}}$ and observe
      that $\modulo{b'(z)}\leq 1$. Then
      \begin{align}
        \nonumber
        \modulo{D_+ - D_-} = \
        & \modulo{b(a(x_{i+1/2})) - b(a(x_{i-1/2}))}
          = \modulo{b'(a(\tilde x_i)) \, a'(\tilde x_i) \, (x_{i+1/2} - x_{i-1/2})}
        \\
        \nonumber
        = \
        & \modulo{\frac{a (\tilde x_i)}{(1+a (\tilde x_i)^2)^{1/2}} \,
          (\partial_x \nabla \eta * \rho^n)(\tilde x_i) \, \dx}
        \\
        \label{eq:5}
        \leq \
        &
          \dx \, \norma{\rho^n}_{\L1} \norma{\nabla^2 \eta}_{\L\infty}.
      \end{align}
      Therefore,
      \begin{equation}
        \label{eq:2bok}
        \epsilon \, \dx \, \dy \,
        \modulo{\frac{1}{D_+} - \frac{1}{D_-}} \, \sum_{k, \ell \in
          \interi} \modulo{\rh{k,\ell}} \modulo{\partial_1
          \eta(x_{i-1/2-k},y_{j-\ell})} \leq
        \epsilon \, \dx \, \norma{\nabla^2 \eta}_{\L\infty} \norma{\rho^n}_{\L1}.
      \end{equation}

 \noindent Inserting~\eqref{eq:2ok} and~\eqref{eq:2bok} into the estimate
   of~\eqref{eq:2original} yields the desired result.

\smallskip
   Consider now~\eqref{eq:Jtripla}. Introduce the following notation:
   for $\mu \in \left\{-1 ;\, 1;\, 3 \right\}$ set
   \begin{displaymath}
     D_{\mu} = \sqrt{1 + \norma{(\nabla\eta * \rho^n) (x_{i+\mu/2}, y_j)}^2}.
   \end{displaymath}
   Thus
   \begin{align*}
     & J_1^n (\x{i+3/2,j}) - 2 \, J_1^n (\x{i+1/2,j}) + J_1^n (\x{i-1/2,j})
     \\
     = \
     & - \epsilon \left(\frac{(\partial_1\eta * \rho^n) (\x{i+3/2,j})}{D_3}
       - 2 \, \frac{(\partial_1\eta * \rho^n) (\x{i+1/2,j})}{D_1}
       + \frac{(\partial_1\eta * \rho^n) (\x{i-1/2,j})}{D_{-1}}\right.
     \\
     & \quad
     \left.
       \pm \frac{(\partial_1\eta * \rho^n) (x_{i+3/2,j})}{D_1}
       \pm \frac{(\partial_1\eta * \rho^n) (x_{i-1/2,j})}{D_1}\right)
     \\
     = \
     & - \epsilon \left(
       \left(\frac{1}{D_3} - \frac{1}{D_1}\right) (\partial_1\eta * \rho^n) (x_{i+3/2,j})
       + \frac{1}{D_1}
       \left((\partial_1\eta * \rho^n) (x_{i+3/2,j})
         - (\partial_1\eta * \rho^n) (x_{i+1/2,j})\right)\right.
     \\
     & \!\quad \left.
         + \frac{1}{D_1}\!
       \left((\partial_1\eta * \rho^n) (x_{i-1/2,j})
         - (\partial_1\eta * \rho^n) (x_{i+1/2,j})\right)
       \!+\!
       \left(\frac{1}{D_{-1}} - \frac{1}{D_1}\right) (\partial_1\eta * \rho^n) (x_{i-1/2,j})
     \right).
   \end{align*}
   Consider the terms separately, forgetting for a moment the
   $\epsilon$ in front of everything. Focus first on the terms with
   common denominator $D_1$:
   \begin{align}
     \nonumber
     & \frac{1}{D_1}\!
     \left((\partial_1\eta * \rho^n) (x_{i+3/2,j})\!
       - (\partial_1\eta * \rho^n) (x_{i+1/2,j})
       + (\partial_1\eta * \rho^n) (x_{i-1/2,j})
       - (\partial_1\eta * \rho^n) (x_{i+1/2,j})\right)
     \\  \nonumber
     = \
     & \frac{\dx \, \dy}{D_1} \sum_{k, \ell \in \interi} \rho^n_{k, \ell} \,
     \left(
       \partial_1 \eta (x_{i+3/2 -k}, y_{j-\ell})
       -  \partial_1 \eta (x_{i+1/2 -k}, y_{j-\ell})
     \right.
     \\  \nonumber
     & \qquad\qquad\qquad\qquad\quad
     \left.
       +  \partial_1 \eta (x_{i-1/2 -k}, y_{j-\ell})
       -  \partial_1 \eta (x_{i+1/2 -k}, y_{j-\ell})
     \right)
     \\  \nonumber
     = \
     &\frac{\dx \, \dy}{D_1} \sum_{k, \ell \in \interi} \rho^n_{k, \ell} \, \dx \,
     \left( \partial_{11}^2 \eta(\hat x_{i+1-k}, y_{j-\ell})
       -  \partial_{11}^2 \eta(\hat x_{i-k}, y_{j-\ell}) \right)
     \\  \nonumber
     = \
     & \frac{\dx \, \dy}{D_1} \sum_{k, \ell \in \interi} \rho^n_{k, \ell} \, \dx \,
     \int_{\hat x_{i-k}}^{\hat x_{i+1-k}} \partial_{111}^3 \eta (x, y_{j-\ell}) \d{x}
     \\ \label{eq:7}
     \leq \
     & 2 \, (\dx)^2 \, \norma{\partial_{111}^3 \eta}_{\L\infty} \norma{\rho^n}_{\L1},
   \end{align}
   with $\hat x_{i-k} \in \, ]x_{i-1/2-k}, x_{i+1/2-k}[$. We are left with
   \begin{equation}
     \label{eq:4}
      \left(\frac{1}{D_3} - \frac{1}{D_1}\right) (\partial_1\eta * \rho^n) (x_{i+3/2,j})
      +
      \left(\frac{1}{D_{-1}} - \frac{1}{D_1}\right) (\partial_1\eta * \rho^n) (x_{i-1/2,j}).
   \end{equation}
   Add and subtract to~\eqref{eq:4}
   \begin{displaymath}
     \left(\frac{1}{D_{-1}} - \frac{1}{D_1}\right) (\partial_1\eta * \rho^n) (x_{i+3/2,j}).
   \end{displaymath}
   Hence,
   \begin{align}
     \label{eq:4a}
    &  \left(\frac{1}{D_3} - 2 \, \frac{1}{D_1} + \frac{1}{D_{-1}} \right)
     (\partial_1\eta * \rho^n) (x_{i+3/2,j})
    \\
    \label{eq:4b}
    & \quad+
     \left(\frac{1}{D_{-1}} - \frac{1}{D_1}\right) \left(
       (\partial_1\eta * \rho^n) (x_{i-1/2,j})
       - (\partial_1\eta * \rho^n) (x_{i+3/2,j})
     \right).
   \end{align}
   Consider first~\eqref{eq:4b}: exploiting also~\eqref{eq:5}, we obtain
   \begin{align}
     \nonumber
     [\eqref{eq:4b}]= \
     & \frac{D_1 - D_{-1}}{D_1 \, D_{-1}} \, \dx \, \dy
     \sum_{k, \ell \in \interi} \rho^n_{k,\ell} \left(
       \partial_1 \eta(x_{i-1/2-k}, y_{j-\ell})
       -
       \partial_1 \eta(x_{i+3/2-k}, y_{j-\ell})
     \right)
     \\ \nonumber
     = \
     & \frac{D_1 - D_{-1}}{D_1 \, D_{-1}} \, \dx \, \dy
     \sum_{k, \ell \in \interi} \rho^n_{k,\ell}
     \int_{x_{i+3/2-k}}^{x_{i-1/2-k}} \partial_{11}^2 \eta(x,y_{j-\ell}) \d{x}
     \\ \label{eq:4bok}
     \leq \
     &
     2 \, (\dx)^2 \, \norma{\nabla^2 \eta}^2_{\L\infty} \norma{\rho^n}^2_{\L1}.
   \end{align}
   As far as~\eqref{eq:4a} is concerned, focus on the terms in the brackets:
   \begin{align}
     \nonumber
     \frac{1}{D_3} - 2 \, \frac{1}{D_1} + \frac{1}{D_{-1}}
     = \
     & \frac{D_1 \, D_{-1} - 2 \, D_3 \, D_{-1} + D_3 \, D_1}{D_3 \, D_1 \, D_{-1}}
     \\ \nonumber
     = \
     & \frac{D_{-1}(D_1 -  D_3 ) - D_3 (D_{-1}- D_1) \pm D_3(D_1 -  D_3 )}
     {D_3 \, D_1 \, D_{-1}}
     \\  \label{eq:6}
     = \
     & \frac{(D_{-1} - D_3)(D_1-D_3)}{D_3 \, D_1 \, D_{-1}}
     - \frac{D_{-1} - 2 \, D_1 + D_3}{D_1 \, D_{-1}}.
   \end{align}
   Inserting the first addend of~\eqref{eq:6} back into~\eqref{eq:4a} yields
   \begin{equation}
     \label{eq:4a1ok}
     \frac{(D_{-1} - D_3)(D_1-D_3)}{D_3 \, D_1 \, D_{-1}}
     (\partial_1\eta * \rho^n) (x_{i+3/2,j})
     \leq
     2 \, (\dx)^2 \norma{\nabla^2 \eta}_{\L\infty}^2 \norma{\rho^n}_{\L1}^2,
   \end{equation}
   where we exploit~\eqref{eq:5} twice and use the fact that $\dfrac{
     (\partial_1\eta * \rho^n) (x_{i+3/2,j})}{D_3} \leq
   1$. Concerning the second addend of~\eqref{eq:6}, focus on its
   numerator: with the notation introduced before~\eqref{eq:5},
   \begin{align*}
     D_{-1}\! - 2 \, D_1 + D_3 = \
     & b\left(a(x_{i+3/2})\right) - 2 \, b\left(a(x_{i+1/2})\right)
     + b\left(a(x_{-1/2})\right)
     \\
     = \
     & b' \left(a(\check x_{i+1})\right) \, a' (\check x_{i+1}) (x_{i+3/2} - x_{i+1/2})
     -  b' \left(a(\check x_{i})\right) \, a' (\check x_{i}) (x_{i+1/2} - x_{i-/2})
     \\
     = \
     & \dx \left(
       b' \left(a(\check x_{i+1})\right) (\partial_1 \nabla \eta * \rho^n) (\check x_{i+1})
       -  b' \left(a(\check x_{i})\right)  (\partial_1 \nabla \eta * \rho^n) (\check x_{i})
     \right)
     \\
     & \pm \dx \,\, b'\!\left(a (\check x_i)\right)
     (\partial_1 \nabla \eta * \rho^n) (\check x_{i+1})
     \\
     = \
     & \dx \, \left[
      b' \left(a(\check x_{i+1})\right) - b'\left(a (\check x_i)\right)
     \right] (\partial_1 \nabla \eta * \rho^n) (\check x_{i+1})
     \\
     & + \dx \, \, b'\!\left(a (\check x_i)\right) \left[
        (\partial_1 \nabla \eta * \rho^n) (\check x_{i+1})
        - (\partial_1 \nabla \eta * \rho^n) (\check x_{i})
     \right]
     \\
     = \
     & \dx \, b''\!\left(a(\overline x_{i+1/2})\right) \, a' (\overline x_{i+1/2}) \,
     (\check x_{i+1} - \check x_{i}) \,
     (\partial_1 \nabla \eta * \rho^n) (\check x_{i+1})
     \\
     & + \dx \, \,  b'\!\left(a (\check x_i)\right) \, \dx \, \dy
     \sum_{k, \ell \in \interi} \rho^n_{k, \ell} \left(
     \partial_1 \nabla \eta(\check x_{i+1-k}, y_{j-\ell})
     -  \partial_1 \nabla \eta(\check x_{i-k}, y_{j-\ell})
     \right)
     \\
     =\
     & \dx \, b''\!\left(a(\overline x_{i+1/2})\right) \, a' (\overline x_{i+1/2}) \,
     (\check x_{i+1} - \check x_{i}) \,
     (\partial_1 \nabla \eta * \rho^n) (\check x_{i+1})
     \\
     &  + \dx \, \,  b'\!\left(a (\check x_i)\right) \, \dx \, \dy
     \sum_{k, \ell \in \interi} \rho^n_{k, \ell} \int_{\check x_{i-k}}^{\check x_{i+1-k}}
     \partial^2_{11}\nabla \eta (x, y_{j-\ell}) \d{x}
   \end{align*}
   where $\check x_i \in \, ]x_{i-1/2}, x_{i+1/2}[$ and $\overline
   x_{i+1/2} \in \, ]\check x_i, \check x_{i+1}[$. Now insert this
   estimate back into~\eqref{eq:6} and~\eqref{eq:4a}: since $\modulo{b''(z)}\leq 1$,
   \begin{equation}
     \label{eq:4a2ok}
     \modulo{ \frac{D_{-1} - 2 \, D_1 + D_3}{D_1 \, D_{-1}}
      (\partial_1\eta * \rho^n) (x_{i+3/2,j})}
    \leq
    2 \, (\dx)^2 \left[
      \norma{\partial_1 \nabla \eta}_{\L\infty}^2 \norma{\rho^n}_{\L1}^2
    + \norma{\partial_{11}^2 \nabla \eta}_{\L\infty} \norma{\rho^n}_{\L1}\right].
   \end{equation}
   Collecting together~\eqref{eq:7}, \eqref{eq:4bok}, \eqref{eq:4a1ok}
   and~\eqref{eq:4a2ok} yields
   \begin{align*}
     & \modulo{J_1^n (x_{i+3/2},y_j) - 2 \, J_1^n (x_{i+1/2},y_j) + J_1^n (x_{i-1/2},y_j)}
     \\
     \leq \
     & 2 \, \epsilon \,  (\dx)^2 \left(
       2 \, \norma{\nabla^3 \eta}_{\L\infty} \norma{\rho^n}_{\L1}
       + 3 \, \norma{\nabla^2 \eta}^2_{\L\infty} \norma{\rho^n}^2_{\L1}
     \right).
   \end{align*}

 \end{proof}
\end{appendices}

\small{ \bibliography{matflow}

  \bibliographystyle{abbrv} }

 \end{document}

%% file: Roe.tex
\subsection{Positivity}
\label{sec:pos}

In the case of positive initial datum, we prove that under a suitable
CFL condition the approximate solution to~\eqref{eq:1} constructed via
the Algorithm~\ref{alg:2} preserves the positivity.

\begin{lemma}{\bf (Positivity)} Let
  $\rho_o \in \L\infty (\reali^2; \reali^+)$. Let~\ref{vs}, \ref{H},
  and~\ref{eta} hold. Assume that
\begin{align}
   \label{eq:CFLroe}
   \lambda_x \leq \
   & \frac{1}{2(\epsilon + \norma{\vs_1}_{\L\infty})},
   &
   \lambda_y \leq \
   & \frac{1}{2(\epsilon + \norma{\vs_2}_{\L\infty})}.
\end{align}
Then, for all $t>0$ and $(x,y) \in \reali^2$, the piece-wise constant
approximate solution $\rho_\Delta$~\eqref{eq:rhodelta} constructed
through Algorithm~\ref{alg:2} is such that
$\rho_\Delta (t,x,y) \geq 0$.
\end{lemma}
\begin{proof}
  Fix $n$ between $0$ and $N_T-1$ and assume that $\rh{i,j} \geq 0$
  for all $i, \, j \in \interi$. Consider~\eqref{eq:scheme1roe}, with
  the notation~\eqref{eq:fx1} and~\eqref{eq:fx2}, and observe that:
  \begin{align*}
    & V_1 (x_{i+1/2,j}, \rh{i,j}, \rh{i+1,j})
      + F (\rh{i,j}, \rh{i+1,j}, J^n_1(\x{i+1/2,j}))
    \\
    \leq \
    & \vs_1(x_{i+1/2,j}) \, \rh{i,j}
      + J^n_1(x_{i+1/2,j}) \, f(\rh{i,j})
      \leq
      \left(\norma{\vs_1}_{\L\infty} + \epsilon \right) \rh{i,j},
    \\
    & V_1 (x_{i-1/2,j}, \rh{i-1,j}, \rh{i,j})
      +F (\rh{i-1,j}, \rh{i,j}, J^n_1 (\x{i-1/2,j}))
    \\
    \geq \
    &\vs_1(x_{i-1/2,j}) \, \rh{i,j}
      + J^n_1(x_{i-1/2,j}) \, f(\rh{i,j})
      \geq
      - \left(\norma{\vs_1}_{\L\infty} + \epsilon \right) \rh{i,j}.
  \end{align*}
  Therefore, by~\eqref{eq:scheme1roe},
  \begin{displaymath}
    \rho_{i,j}^{n+1/2}
    \geq \rh{i,j} - 2\, \lambda_x \left(\norma{\vs_1}_{\L\infty} + \epsilon\right) \rh{i,j}
    \geq 0,
  \end{displaymath}
  thanks to the CFL condition~\eqref{eq:CFLroe}. Starting
  from~\eqref{eq:scheme2roe}, an analogous argument shows that
  $\rho_{i,j}^{n+1}\geq 0$, concluding the proof.
\end{proof}

\subsection{\texorpdfstring{$\L1$}{L 1} bound}
\label{sec:l1}

The following result on the $\L1$ bound follows from the conservation
property of the Roe scheme.
\begin{lemma}{\bf ($\L1$ bound)}\label{lem:L1roe} Let
  $\rho_o \in \L\infty (\reali^2; \reali^+)$. Let~\ref{vs}, \ref{H},
  \ref{eta} and~\eqref{eq:CFLroe} hold. Then, for all $t>0$ and
  $(x,y) \in \reali^2$, $\rho_\Delta$~\eqref{eq:rhodelta} constructed
  through Algorithm~\ref{alg:2} satisfies
  \begin{equation}
    \label{eq:l1}
    \norma{\rho_\Delta(t,\cdot,\cdot)}_{\L1(\reali^2)} = \norma{\rho_o}_{\L1(\reali^2)}.
  \end{equation}
\end{lemma}

\subsection{\texorpdfstring{$\L\infty$}{L infinity} bound}
\label{sec:linf}

\begin{lemma}{\bf ($\L\infty$~bound)}\label{lem:Linfroe} Let
  $\rho_o \in \L\infty (\reali^2; \reali^+)$. Let~\ref{vs}, \ref{H},
  \ref{eta} and~\eqref{eq:CFLroe} hold. Then, for all $t>0$ and
  $(x,y) \in \reali^2$, $\rho_\Delta$~\eqref{eq:rhodelta} constructed
  through Algorithm~\ref{alg:2} satisfies
  \begin{equation}
    \label{eq:linf}
    \norma{\rho_\Delta (t,\cdot,\cdot)}_{\L\infty (\reali^2)}
    \leq  \norma{\rho_o}_{\L\infty} \, e^{\mathcal{C}_\infty \, t},
  \end{equation}
  where
  \begin{equation}
    \label{eq:Cinf}
    \mathcal{C}_\infty =
      \norma{\partial_x \vs_1}_{\L\infty}
      + \norma{\partial_y \vs_2}_{\L\infty}
      + 4 \, \epsilon \,  \norma{\nabla^2 \eta}_{\L\infty} \norma{\rho_o}_{\L1}.
  \end{equation}
\end{lemma}
\begin{proof}
  Omitting the dependencies on
  $j$ and exploiting the notation introduced in~\eqref{eq:notvJ}, we

  observe that \eqref{eq:scheme1roe} attains its maximum for $v_{i+1/2} < 0$, $v_{i-1/2}\geq 0$, $ J_1^n(x_{i+1/2})< 0$ and $ J_1^n(x_{i-1/2})\geq 0$. In this case
  \begin{align*}
    \rho_{i,j}^{n+1/2}
    \leq \
      & \rh{i} - \lambda_x \left(
        \rh{i+1} \, v_{i+1/2} + J_1^n(x_{i+1/2}) \, f(\rh{i+1})
        \right)
        + \lambda_x \left(
        \rh{i-1} \, v_{i-1/2} + J_1^n(x_{i-1/2}) \, f(\rh{i-1})
        \right),
  \end{align*}
  where we use the positivity of each $\rh{i}$ and of the function $f$
  and discard all the terms giving a negative contribution.  Moreover, since $v_{i+1/2} < 0$ and $v_{i-1/2}\geq 0$,
  \begin{align*}
    \lambda_x \left(
    - \rh{i+1} \, v_{i+1/2} + \rh{i-1} \, v_{i-1/2}
    \right)
    \leq \
    & \lambda_x \, \norma{\rho^n}_{\L\infty} \left(
      -v_{i+1/2} + v_{i-1/2}
      \right)
    \\
    = \
    &
      \lambda_x \, \norma{\rho^n}_{\L\infty} \, (- \dx) \, \partial_x \vs_1(\hat x_i)
  \end{align*}
  with $\hat x_i \in \, ]x_{i-1/2}, x_{i+1/2}[$. In a similar way,
  since $ J_1^n(x_{i+1/2})< 0$ and $ J_1^n(x_{i-1/2})\geq 0$,
  exploiting also the fact that $f(r) \leq r$ for all $r\geq 0$, we
  get
  \begin{align*}
    \lambda_x \left(
    -  J_1^n(x_{i+1/2}) \, f(\rh{i+1}) +  J_1^n(x_{i-1/2}) \, f(\rh{i-1})
    \right)
    \leq \
    & \lambda_x \, \norma{\rho^n}_{\L\infty} \left(
      -  J_1^n(x_{i+1/2}) +  J_1^n(x_{i-1/2})
      \right)
    \\
    \leq \
    &\lambda_x \, \norma{\rho^n}_{\L\infty} \, 2 \, \epsilon \, \dx \,
      \norma{\nabla^2 \eta}_{\L\infty} \norma{\rho^n}_{\L1},
  \end{align*}
  thanks to~\eqref{eq:Jx}.  Therefore,
  \begin{align*}
    \rho_{i,j}^{n+1/2} \leq \
    \norma{\rho^n}_{\L\infty} \left[
    1 + \dt \left(
    \norma{\partial_x \vs_1}_{\L\infty} + 2 \, \epsilon
    \norma{\nabla^2 \eta}_{\L\infty} \norma{\rho_o}_{\L1}
    \right)
    \right].
  \end{align*}
  In a similar way we get
  \begin{align*}
    \rho_{i,j}^{n+1} \leq \
    \norma{\rho^{n+1/2}}_{\L\infty} \left[
    1 + \dt \left(
    \norma{\partial_y \vs_2}_{\L\infty} + 2 \, \epsilon
    \norma{\nabla^2 \eta}_{\L\infty} \norma{\rho_o}_{\L1}
    \right)
    \right],
  \end{align*}
  concluding the proof.
\end{proof}

\subsection{\texorpdfstring{$\BV$}{BV} bound}
\label{sec:bv}

\begin{proposition} {\bf ($\BV$ estimate in space)}\label{prop:bvroe}
  Let $\rho_o \in (\L\infty \cap \BV) (\reali^2;
  \reali^+)$. Let~\ref{vs}, \ref{H}, and~\ref{eta} hold. Assume that
  \begin{align}
   \label{eq:CFLroe-v2}
   \lambda_x \leq \
    & \frac{1}{3(\epsilon \, L_f + \norma{\vs_1}_{\L\infty})},
    &
      \lambda_y \leq \
    &  \frac{1}{3(\epsilon \, L_f + \norma{\vs_2}_{\L\infty})}.
  \end{align}
  Then, for all $t>0$, $\rho_\Delta$ in~\eqref{eq:rhodelta}
  constructed through Algorithm~\ref{alg:2} satisfies the following
  estimate: for all $n=0, \ldots, N_T$,
  \begin{equation}
    \label{eq:bvspaceroe}
    \sum_{i,j \in \interi} \left(
      \dy \, \modulo{\rh{i+1,j} - \rh{i,j}}
      + \dx \, \modulo{\rh{i,j+1} - \rh{i,j}}\right)
    \leq \mathcal{C}_x(t^n),
  \end{equation}
  where
  \begin{equation}
    \label{eq:Cx}
    \mathcal{C}_x(t) = e^{2 \,t \, \mathcal{K}_1}  \sum_{i,j \in \interi} \left(
      \dx \, \modulo{\rho^0_{i,j+1} - \rho^0_{i,j}}
      +  \dy \, \modulo{\rho^0_{i+1,j} -\rho^0_{i,j}}
    \right)
    +  \frac{2 \, \mathcal{K}_2}{\mathcal{K}_1}\left(e^{2 \, t \, \mathcal{K}_1} -1
    \right),
  \end{equation}
  with
  \begin{align}
    \label{eq:K1defroe}
    \mathcal{K}_1 = \
    &
    6  \left( \norma{\nabla \boldsymbol{\vs}}_{\L\infty} + 2 \, \epsilon \, L_f \,
    \norma{\nabla^2 \eta}_{\L\infty} \norma{\rho_o}_{\L1}\right),
    \\
    \label{eq:K2defroe}
    \mathcal{K}_2 = \
    &  \left(
      4 \, \epsilon \left(
        c_1 \norma{\rho_o}_{\L1} + c_2 \norma{\rho_o}_{\L1}^2
      \right)
      + 3 \, \norma{\nabla^2 \boldsymbol{\vs}}_{\L\infty}
      \right) \norma{\rho_o}_{\L1},
  \end{align}
  and $c_1, \, c_2 $ are defined in~\eqref{eq:c12}.
\end{proposition}

\begin{remark}
  Observe that the CFL conditions~\eqref{eq:CFLroe-v2} are stricter than~\eqref{eq:CFLroe}.
\end{remark}

\begin{proof}
  We follow the idea of~\cite[Lemma~2.7]{ACG2015}. First consider the
  term
  \begin{displaymath}
    \sum_{i,j \in \interi} \dy\, \modulo{\rho^{n+1/2}_{i+1,j} - \rho^{n+1/2}_{i,j}}.
  \end{displaymath}
  In particular, fixing $i, j \in \interi$ and omitting the
  dependencies on $y_j$ for the sake of simplicity,
  by~\eqref{eq:scheme1roe} we get
  \begin{align*}
    \rho_{i+1}^{n+1/2} - \rho_{i}^{n+1/2}
    = \
    & \rh{i+1} - \rh{i}
      - \lambda_x \left[ V_1 (\x{i+3/2}, \rh{i+1}, \rh{i+2})
      +F (\rh{i+1}, \rh{i+2}, J^n_1 (x_{i+3/2}))
      \right.
    \\
    & \qquad\qquad\qquad\quad
      -V_1 (\x{i+1/2}, \rh{i}, \rh{i+1})
      -F (\rh{i}, \rh{i+1}, J^n_1 (x_{i+1/2}))
    \\
    & \qquad\qquad\qquad\quad
      - V_1 (\x{i+1/2}, \rh{i}, \rh{i+1})
      - F (\rh{i}, \rh{i+1}, J^n_1 (x_{i+1/2}))
    \\
    & \qquad\qquad\qquad\quad
      \left.
      +  V_1 (\x{i-1/2}, \rh{i-1}, \rh{i})
      +F (\rh{i-1}, \rh{i}, J^n_1 (x_{i-1/2}))
      \right]
    \\
    & \pm \lambda_x \left[
      V_1 (\x{i+3/2}, \rh{i}, \rh{i+1})
      +F (\rh{i}, \rh{i+1}, J^n_1 (x_{i+3/2}))
      \right.
    \\
    &  \qquad\left.
      -V_1 (\x{i+1/2}, \rh{i-1}, \rh{i})
      -F (\rh{i-1}, \rh{i}, J^n_1 (x_{i+1/2}))\right]
    \\
    = \
    & \mathcal{A}_{i,j}^n - \lambda_x \, \mathcal{B}_{i,j}^n,
  \end{align*}
  where we set
  \begin{align*}
    \mathcal{A}_{i,j}^n = \
    & \rh{i+1} - \rh{i} - \lambda_x \left[
      V_1 (\x{i+3/2}, \rh{i+1}, \rh{i+2})
      +F (\rh{i+1}, \rh{i+2}, J^n_1 (x_{i+3/2})) \right.
    \\
    & \qquad\qquad\qquad\quad
      -V_1 (\x{i+1/2}, \rh{i}, \rh{i+1})
      -F (\rh{i}, \rh{i+1}, J^n_1 (x_{i+1/2}))
    \\
    & \qquad\qquad\qquad\quad
      +V_1 (\x{i+1/2}, \rh{i-1}, \rh{i})
      +F (\rh{i-1}, \rh{i}, J^n_1 (x_{i+1/2}))
    \\
    & \qquad\qquad\qquad\quad
      \left.
      -V_1 (\x{i+3/2}, \rh{i}, \rh{i+1})
      -F (\rh{i}, \rh{i+1}, J^n_1 (x_{i+3/2}))\right],
    \\
    \mathcal{B}_{i,j}^n = \
    &  V_1 (\x{i+3/2}, \rh{i}, \rh{i+1})
      +F (\rh{i}, \rh{i+1}, J^n_1 (x_{i+3/2}))
    \\
    &  -V_1 (\x{i+1/2}, \rh{i-1}, \rh{i})
      -F (\rh{i-1}, \rh{i}, J^n_1 (x_{i+1/2}))
    \\
    &+  V_1 (\x{i-1/2}, \rh{i-1}, \rh{i})
      +F (\rh{i-1}, \rh{i}, J^n_1 (x_{i-1/2}))
    \\
    &    -V_1 (\x{i+1/2}, \rh{i}, \rh{i+1})
      -F (\rh{i}, \rh{i+1}, J^n_1 (x_{i+1/2})).
  \end{align*}
  For the sake of shortness, introduce the following notation
  \begin{equation}
    \label{eq:10}
    H^n_{k,\ell}(u,w) = V_1(x_{k,\ell},u,w)+F(u,w,J^n_1(x_{k,\ell})),
  \end{equation}
  so that, dropping the $j$ dependencies, $\mathcal{A}^n_{i,j}$ reads
  \begin{align*}
    \mathcal{A}_{i,j}^n= & \ \rh{i+1} - \rh{i} - \lambda_x
     \left[
      H^n_{i+3/2} (\rh{i+1},\rh{i+2})
      -H^n_{i+1/2} (\rh{i},\rh{i+1})\right.
    \\
    & \left.
      +H^n_{i+1/2,j} (\rh{i-1,j},\rh{i,j})
      -H^n_{i+3/2,j} (\rh{i,j},\rh{i+1,j})\right]
    \\
   = \
    & \rh{i+1} - \rh{i}
      - \lambda_x \, \frac{
      H^n_{i+3/2} ( \rh{i+1},\rh{i+2})
      - H^n_{i+3/2} ( \rh{i+1},\rh{i+i}) }{\rh{i+2} - \rh{i+1}}
      \left(\rh{i+2}-\rh{i+1}\right)
    \\
    &
      -\lambda_x \, \frac{H^n_{i+3/2}( \rh{i+1},\rh{i+i})
      - H^n_{i+3/2} (\rh{i},\rh{i+i})}{\rh{i+1}-\rh{i}}
      \left(\rh{i+1}-\rh{i}\right)
    \\
    & + \lambda_x \, \frac{ H^n_{i+1/2}( \rh{i}, \rh{i+1})
      -  H^n_{i+1/2} (\rh{i}, \rh{i})}{\rh{i+1} - \rh{i}}
      \left(\rh{i+1}-\rh{i}\right)
    \\
    & + \lambda_x \, \frac{ H^n_{i+1/2}(\rh{i}, \rh{i})
      -  H^n_{i+1/2}( \rh{i-1}, \rh{i}) }{\rh{i} - \rh{i-1}}
      \left(\rh{i} - \rh{i-1} \right)
    \\
    = \
    & \delta_{i+1}^n  \left(\rh{i+2}-\rh{i+1}\right)
      + \theta_i^n  \left(\rh{i} - \rh{i-1} \right)
      + (1 - \delta_i^n - \theta_{i+1}^n) \left(\rh{i+1}-\rh{i}\right),
  \end{align*}
  where
  \begin{align}
    \label{eq:deltainroe}
    \delta_i^n = \ &
                     \begin{cases}
                       - \lambda_x \, \dfrac{H^n_{i+1/2} ( \rh{i},
                         \rh{i+1})- H^n_{i+1/2} (\rh{i},
                         \rh{i})}{\rh{i+1}-\rh{i}} & \mbox{if } \rh{i}
                       \neq \rh{i+1},
                       \\
                       0 & \mbox{if } \rh{i} = \rh{i+1} ,
                     \end{cases}
    \\
    \label{eq:thetainroe}
    \theta_i^n = \ &
                     \begin{cases}
                       \lambda_x \, \dfrac{H^n_{i+1/2} (\rh{i},
                         \rh{i}) - H^n_{i+1/2}(\rh{i-1},
                         \rh{i})}{\rh{i}-\rh{i-1}} & \mbox{if } \rh{i}
                       \neq \rh{i-1},
                       \\
                       0 & \mbox{if } \rh{i} = \rh{i-1}.
                     \end{cases}
  \end{align}
  Exploiting~\eqref{eq:10}, observe that, whenever
  $\rh{i}\neq \rh{i+1}$,
  \begin{align*}
    \delta_i^n = \
    & -\frac{\lambda_x}{\rh{i+1}-\rh{i}}
      \left[
      V_1(x_{i+1/2},\rh{i}, \rh{i+1})
      + F (\rh{i}, \rh{i+1}, J^n_1 (x_{i+1/2}))\right.
    \\
    & \qquad\qquad\qquad\left.
      -V_1(x_{i+1/2},\rh{i}, \rh{i})
      - F (\rh{i}, \rh{i}, J^n_1 (x_{i+1/2}))
      \right]
    \\
    = \
    & \! -\frac{\lambda_x}{\rh{i+1}-\rh{i}}\!\!
      \left[ \min\left\{0, \vs_1(x_{i+1/2})\right\} (\rh{i+1} - \rh{i})
      +  \min\left\{0, J^n_1(x_{i+1/2})\right\} \!\!\left(f(\rh{i+1}) -f( \rh{i})\right)\!
      \right]
    \\
    = \
    & -\lambda_x \left(  \min\left\{0, \vs_1(x_{i+1/2})\right\}
      + \min\left\{0, J^n_1(x_{i+1/2})\right\} f' (r^n_{i+1/2}) \right),
  \end{align*}
  with $r^n_{i+1/2} \in \mathcal{I}\left(\rh{i},
    \rh{i+1}\right)$. Since $f' (r) \geq 0$ and
  by~\eqref{eq:CFLroe-v2} we get
  \begin{displaymath}
    \delta_i^n \in \left[0, \frac13\right].
  \end{displaymath}
  In a similar way one can prove that $\theta_i^n \in [0, 1/3]$. Thus,
  \begin{equation}
    \label{eq:AijOKroe}
    \sum_{i,j\in \interi} \modulo{\mathcal{A}_{i,j}^n}
    \leq
    \sum_{i,j \in \interi} \modulo{\rh{i+1,j} - \rh{i,j}}.
  \end{equation}

  We pass now to $\mathcal{B}_{i,j}^n$. Consider separately the terms
  involving $V_1$ and those involving $F$. Observe that the maps
  \begin{align*}
    x \mapsto &\min\left\{0, \vs_1 (x)\right\},
    &
      x \mapsto &\min\left\{0, J_1^n (x)\right\}
  \end{align*}
  are Lipschitz continuous, with constant respectively
  $\norma{\partial_x \vs_1}_{\L\infty}$ and
  $2 \, \epsilon \norma{\nabla^2 \eta}_{\L\infty}
  \norma{\rho_o}_{\L1}$.  Exploiting~\eqref{eq:fx1} we get:
  \begin{align}
    \nonumber
    & V_1 (\x{i+3/2}, \rh{i}, \rh{i+1})
      -V_1 (\x{i+1/2}, \rh{i-1}, \rh{i})
      +  V_1 (\x{i-1/2}, \rh{i-1}, \rh{i})
      -V_1 (\x{i+1/2}, \rh{i}, \rh{i+1})
    \\  \nonumber
    = \
    & \vs_1 (x_{i+3/2}) \rh{i} + \min\left\{0, \vs_1 (x_{i+3/2})\right\} (\rh{i+1} - \rh{i})
    \\  \nonumber
    &  - \vs_1 (x_{i+1/2}) \rh{i} - \min\left\{0, \vs_1 (x_{i+1/2})\right\} (\rh{i+1} - \rh{i})
    \\  \nonumber
    &   + \vs_1 (x_{i-1/2}) \rh{i-1} + \min\left\{0, \vs_1 (x_{i-1/2})\right\} (\rh{i} - \rh{i-1})
    \\ \nonumber
    & - \vs_1 (x_{i+1/2}) \rh{i-1} - \min\left\{0, \vs_1 (x_{i+1/2})\right\} (\rh{i} - \rh{i-1})
    \\ \nonumber
    & \pm \left(\vs_1 (x_{i-1/2}) - \vs_1 (x_{i+1/2})\right) \rh{i}
    \\ \nonumber
    = \
    &  \left(\vs_1 (x_{i+3/2})  -2 \, \vs_1 (x_{i+1/2}) + \vs_1 (x_{i-1/2})\right) \rh{i}
    \\ \nonumber
    & +  \left(\vs_1 (x_{i-1/2}) - \vs_1 (x_{i+1/2})\right)(\rh{i-1} - \rh{i})
    \\ \nonumber
    &+ \left(
      \min\left\{0, \vs_1 (x_{i+3/2})\right\} - \min\left\{0, \vs_1 (x_{i+1/2})\right\}
      \right)(\rh{i+1} - \rh{i})
    \\ \nonumber
    & +\left(
      \min\left\{0, \vs_1 (x_{i-1/2})\right\} -   \min\left\{0, \vs_1 (x_{i+1/2})\right\}
      \right)(\rh{i} - \rh{i-1})
    \\ \label{eq:B-V}
    \leq \
    & 2 \, (\dx)^2 \, \norma{\partial_{xx}^2 \vs_1}_{\L\infty} \modulo{\rh{i}}
      + \dx \norma{\partial_x \vs_1}_{\L\infty} \left(
      \modulo{\rh{i+1} - \rh{i}} + 2 \, \modulo{\rh{i} - \rh{i-1}}
      \right),
  \end{align}
  since
  \begin{align*}
    \vs_1 (x_{i+3/2})  -2 \, \vs_1 (x_{i+1/2}) + \vs_1 (x_{i-1/2}) = \
    &
      \dx \, \partial_x \vs_1 (\xi_{i+1}) - \dx \, \partial_x \vs_1 (\xi_{i})
    \\
    = \
    &
      \dx \, (\xi_{i+1} - \xi_i) \, \partial_{xx}^2\vs_1 (\zeta_{i+1/2}),
  \end{align*}
  with $\xi_i \in\, ]x_{i-1/2}, x_{i+1/2}[ $ and
  $\zeta_{i+1/2} \in \,]\xi_i, \xi_{i+1}[$.

  Similarly, exploiting~\eqref{eq:fx2} we obtain
  \begin{align}
    \nonumber
    & F (\rh{i}, \rh{i+1}, J^n_1 (x_{i+3/2}))
      -F (\rh{i-1}, \rh{i}, J^n_1 (x_{i+1/2}))
    \\
    \nonumber
    & +F (\rh{i-1}, \rh{i}, J^n_1 (x_{i-1/2}))
      -F (\rh{i}, \rh{i+1}, J^n_1 (x_{i+1/2}))
    \\ \nonumber
    = \
    & J^n_1 (x_{i+3/2}) f( \rh{i}) + \min\left\{0, J^n_1 (x_{i+3/2})\right\} (f(\rh{i+1}) - f(\rh{i}))
    \\  \nonumber
    &  - J^n_1 (x_{i+1/2}) f(\rh{i}) - \min\left\{0, J^n_1 (x_{i+1/2})\right\} (f(\rh{i+1}) - f(\rh{i}))
    \\  \nonumber
    &   + J^n_1 (x_{i-1/2}) f(\rh{i-1}) + \min\left\{0, J^n_1 (x_{i-1/2})\right\} (f(\rh{i}) - f(\rh{i-1}))
    \\ \nonumber
    & - J^n_1 (x_{i+1/2}) f(\rh{i-1}) - \min\left\{0, J^n_1 (x_{i+1/2})\right\} (f(\rh{i}) - f(\rh{i-1}))
    \\ \nonumber
    & \pm \left(J^n_1 (x_{i-1/2}) - J^n_1 (x_{i+1/2})\right) f(\rh{i})
    \\ \nonumber
    = \
    &  \left(J^n_1 (x_{i+3/2})  -2 \, J^n_1 (x_{i+1/2}) + J^n_1 (x_{i-1/2})\right) f(\rh{i})
    \\ \nonumber
    & +  \left(J^n_1 (x_{i-1/2}) - J^n_1 (x_{i+1/2})\right)(f(\rh{i-1}) - f(\rh{i}))
    \\ \nonumber
    &+ \left(
      \min\left\{0, J^n_1 (x_{i+3/2})\right\} - \min\left\{0, J^n_1 (x_{i+1/2})\right\}
      \right)(f(\rh{i+1}) - f(\rh{i}))
    \\ \nonumber
    & +\left(
      \min\left\{0, J^n_1 (x_{i-1/2})\right\} -   \min\left\{0, J^n_1 (x_{i+1/2})\right\}
      \right)(f(\rh{i}) - f(\rh{i-1}))
    \\ \label{eq:B-F}
    \leq \
    & 2 \, \epsilon \, (\dx)^2 \, C \, \modulo{\rh{i}}
      + 2 \, \epsilon \,  L_f \, \dx \norma{\nabla^2 \eta}_{\L\infty} \norma{\rho_o}_{\L1} \left(
      \modulo{\rh{i+1} - \rh{i}} + 2 \, \modulo{\rh{i} - \rh{i-1}}
      \right),
  \end{align}
  where we used the fact that $f(r) \leq r$, \eqref{eq:Jx} and~\eqref{eq:Jtripla}, with the
  notation~\eqref{eq:c12}.  Collecting together~\eqref{eq:B-V}
  and~\eqref{eq:B-F} we therefore obtain
  \begin{align*}
    \modulo{\mathcal{B}_{i,j}^n}\leq \
    &  2 \, (\dx)^2 \left(
      \norma{\partial_{xx}^2 \vs_1}_{\L\infty}
      + \epsilon\, C
      \right)\modulo{\rh{i}}
    \\
    &   + \dx  \left(
      \norma{\partial_x \vs_1}_{\L\infty}
      + 2 \, \epsilon \, L_f \norma{\nabla^2 \eta}_{\L\infty} \norma{\rho_o}_{\L1}
      \right)
      \left(
      \modulo{\rh{i+1} - \rh{i}} + 2 \, \modulo{\rh{i} - \rh{i-1}}
      \right),
  \end{align*}
  so that
  \begin{align}
    \nonumber
    \sum_{i,j \in \interi} \lambda_x   \, \modulo{\mathcal{B}_{i,j}^n}
    \leq \
    & 3 \, \dt  \left(
      \norma{\partial_x \vs_1}_{\L\infty}
      + 2 \, \epsilon \, L_f \norma{\nabla^2 \eta}_{\L\infty} \norma{\rho_o}_{\L1}
      \right) \sum_{i,j \in \interi} \modulo{\rh{i+1} -\rh{i}}
    \\ \label{eq:BijOKroe}
    &
      + 2 \, \dt  \left(
      \norma{\partial_{xx}^2 \vs_1}_{\L\infty}
      + \epsilon\, C
      \right) \dx \sum_{i,j \in \interi} \modulo{\rh{i}}.
  \end{align}
  Therefore, by~\eqref{eq:AijOKroe} and~\eqref{eq:BijOKroe}, using
  also Lemma~\ref{lem:L1roe}
  \begin{align}
    \nonumber
    & \sum_{i,j \in \interi} \dy \, \modulo{\rho_{i+1,j}^{n+1/2} - \rho_{i,j}^{n+1/2}}
    \\ \nonumber
    \leq \
    & \sum_{i,j \in \interi} \dy \left(\modulo{\mathcal{A}_{i,j}^n}
      + \lambda_x\, \modulo{\mathcal{B}_{i,j}^n}\right)
    \\
    \label{eq:BVdyroe}
    \leq \
    & \left[ 1+ 3 \, \dt  \left(
      \norma{\partial_x \vs_1}_{\L\infty}
      + 2 \, \epsilon \, L_f  \norma{\nabla^2 \eta}_{\L\infty} \norma{\rho_o}_{\L1}
      \right) \right]\sum_{i,j \in \interi} \dy \, \modulo{\rh{i+1} -\rh{i}}
    \\ \nonumber
    &
      + 2 \, \dt  \left[
      \norma{\partial_{xx}^2 \vs_1}_{\L\infty}
      + \epsilon \, C
      \right] \norma{\rho_o}_{\L1}.
  \end{align}

  Now pass to the term
  \begin{displaymath}
    \sum_{i,j \in \interi} \dx \, \modulo{\rho_{i,j+1}^{n+1/2} - \rho_{i,j}^{n+1/2}}.
  \end{displaymath}
  Fix $i,j\in\interi$ and exploit~\eqref{eq:scheme1roe} again
  to get
  \begin{align*}
    &\rho^{n+1/2}_{i,j+1} - \rho^{n+1/2}_{i,j}
    \\
    = \
    & \rh{i,j+1} - \rh{i,j}
      - \lambda_x \left[
      V_1 (\x{i+1/2,j+1}, \rh{i,j+1}, \rh{i+1,j+1})
      +F (\rh{i,j+1}, \rh{i+1,j+1}, J^n_1 (x_{i+1/2,j+1}))
      \right.
    \\
    & \qquad\qquad\qquad\quad
      -V_1 (\x{i-1/2,j+1}, \rh{i-1,j+1}, \rh{i,j+1})
      -F (\rh{i-1,j+1}, \rh{i,j+1}, J^n_1 (\x{i-1/2,j+1}))
    \\
    & \qquad\qquad\qquad\quad
      - V_1 (\x{i+1/2,j}, \rh{i,j}, \rh{i+1,j})
      - F (\rh{i,j}, \rh{i+1,j}, J^n_1 (x_{i+1/2,j}))
    \\
    & \qquad\qquad\qquad\quad
      \left.
      +  V_1 (\x{i-1/2,j}, \rh{i-1,j}, \rh{i,j})
      +F (\rh{i-1,j}, \rh{i,j}, J^n_1 (x_{i-1/2,j}))
      \right]
    \\
    & \pm \lambda_x \left[
      V_1 (\x{i+1/2,j+1}, \rh{i,j}, \rh{i+1,j})
      +F (\rh{i,j}, \rh{i+1,j}, J^n_1 (x_{i+1/2,j+1}))
      \right.
    \\
    &  \qquad\left.
      -V_1 (\x{i-1/2,j+1}, \rh{i-1,j}, \rh{i,j})
      -F (\rh{i-1,j}, \rh{i,j}, J^n_1 (x_{i-1/2,j+1}))\right]
    \\
    = \
    & \mathcal{D}_{i,j}^n + \lambda_x \, \mathcal{E}_{i,j}^n,
  \end{align*}
  where we set
  \begin{align*}
    & \mathcal{D}_{i,j}^n
    \\ = \
    & \rh{i,j+1} - \rh{i,j} -\lambda_x \left[
      V_1 (\x{i+1/2,j+1}, \rh{i,j+1}, \rh{i+1,j+1})
      + F (\rh{i,j+1}, \rh{i+1,j+1}, J_1^n (\x{i+1/2,j+1}))
      \right.
    \\ &  \qquad\qquad\qquad\qquad
         - V_1 (\x{i+1/2,j+1}, \rh{i,j}, \rh{i+1,j})
         - F ( \rh{i,j}, \rh{i+1,j}, J_1^n (\x{i+1/2,j+1}))
    \\
    &  \qquad\qquad\qquad\qquad
      +  V_1 (\x{i-1/2,j+1}, \rh{i-1,j}, \rh{i,j})
      + F( \rh{i-1,j}, \rh{i,j}, J_1^n (\x{i-1/2,j+1}))
    \\
    & \left. \qquad\qquad\qquad\qquad
      - V_1 (\x{i-1/2,j+1}, \rh{i-1,j+1}, \rh{i,j+1})
      - F(\rh{i-1,j+1}, \rh{i,j+1}, J_1^n(\x{i-1/2,j+1}))
      \right],
  \end{align*}
  \begin{align*}
    \mathcal{E}_{i,j}^n = \
    & V_1 (\x{i+1/2,j}, \rh{i,j}, \rh{i+1,j})
      + F (\rh{i,j}, \rh{i+1,j}, J_1^n (\x{i+1/2,j}))
    \\
    & - V_1 (\x{i+1/2,j+1}, \rh{i,j}, \rh{i+1,j})
      - F (\rh{i,j}, \rh{i+1,j}, J_1^n (\x{i+1/2,j+1}))
    \\
    &
      + V_1 (\x{i-1/2,j+1}, \rh{i-1,j}, \rh{i,j})
      + F (\rh{i-1,j}, \rh{i,j}, J_1^n (\x{i-1/2,j+1}))
    \\
    & - V_1 (\x{i-1/2,j}, \rh{i-1,j}, \rh{i,j})
      - F (\rh{i-1,j}, \rh{i,j}, J_1^n (\x{i-1/2,j})).
  \end{align*}
  Similarly as before, rearrange $\mathcal{D}_{i,j}^n$, exploiting the
  notation~\eqref{eq:10}:
  \begin{align*}
    \mathcal{D}_{i,j}^n
    = \
    & \rh{i,j+1} - \rh{i,j}
      - \lambda_x\left[
      H^n_{i+1/2,j+1} (\rh{i,j+1}, \rh{i+1,j+1})
      - H^n_{i+1/2,j+1} (\rh{i,j}, \rh{i+1,j})
      \right.
    \\
    & \left. \qquad\qquad\qquad\qquad + H^n_{i-1/2,j+1} (\rh{i-1,j}, \rh{i,j})
      - H^n_{i-1/2,j+1} (\rh{i-1,j+1}, \rh{i,j+1})
      \right]
    \\
    & \pm \, \lambda_x H^n_{i+1/2, j+1} (\rh{i,j}, \rh{i+1,j+1})
      \pm \, \lambda_x H^n_{i-1/2,j+1} (\rh{i-1,j}, \rh{i,j+1})
    \\
    = \
    & \rh{i,j+1} - \rh{i,j}
    \\
    & - \lambda_x \frac{H^n_{i+1/2,j+1} (\rh{i,j+1}, \rh{i+1,j+1})
      - H^n_{i+1/2,j+1} (\rh{i,j}, \rh{i+1,j+1}) }
      {\rh{i, j+1}-\rh{i,j}}(\rh{i, j+1}-\rh{i,j})
    \\
    & - \lambda_x \frac{H^n_{i+1/2,j+1} (\rh{i,j}, \rh{i+1,j+1})
      - H^n_{i+1/2,j+1} (\rh{i,j}, \rh{i+1,j}) }
      {\rh{i+1, j+1}-\rh{i+1,j}}(\rh{i+1, j+1}-\rh{i+1,j})
    \\
    & + \lambda_x \frac{H^n_{i-1/2,j+1} (\rh{i-1,j}, \rh{i,j+1})
      - H^n_{i-1/2,j+1} (\rh{i-1,j}, \rh{i,j}) }
      {\rh{i, j+1}-\rh{i,j}}(\rh{i, j+1}-\rh{i,j})
    \\
    & + \lambda_x \frac{H^n_{i-1/2,j+1} ( \rh{i-1,j+1}, \rh{i,j+1})
      - H^n_{i-1/2,j+1} ( \rh{i-1,j}, \rh{i,j+1}) }
      {\rh{i-1, j+1}-\rh{i-1,j}}(\rh{i-1, j+1}-\rh{i-1,j})
    \\
    = \
    & (1 - \kappa_{i,j}^n - \nu_{i,j}^n) (\rh{i, j+1}-\rh{i,j})
      + \nu_{i+1,j}^n (\rh{i+1, j+1}-\rh{i+1,j})
      + \kappa_{i-1,j}^n (\rh{i-1, j+1}-\rh{i-1,j}),
  \end{align*}
  where
  \begin{align*}
    \kappa_{i,j}^n = \ &
                         \begin{cases}
                           \lambda_x \, \dfrac{H^n_{i+1/2,j+1}
                             (\rh{i,j+1}, \rh{i+1,j+1}) -
                             H^n_{i+1/2,j+1}(\rh{i,j}, \rh{i+1,j+1}) }
                           {\rh{i, j+1}-\rh{i,j}} & \mbox{ if } \rh{i,
                             j+1} \neq \rh{i,j},
                           \\
                           0 & \mbox{ if } \rh{i, j+1} = \rh{i,j},
                         \end{cases}
    \\
    \nu_{i,j}^n = \ &
                      \begin{cases}
                        - \lambda_x \, \dfrac{H^n_{i-1/2,j+1}
                          (\rh{i-1,j}, \rh{i,j+1}) - H^n_{i-1/2,j+1}
                          (\rh{i-1,j}, \rh{i,j}) } {\rh{i,
                            j+1}-\rh{i,j}} & \mbox{ if } \rh{i, j+1}
                        \neq \rh{i,j},
                        \\
                        0 & \mbox{ if } \rh{i, j+1} = \rh{i,j}.
                      \end{cases}
  \end{align*}
  As for $\delta_i^n$~\eqref{eq:deltainroe} and
  $\theta_i^n$~\eqref{eq:thetainroe}, it is immediate to prove that
  $\kappa_{i,j}^n, \, \nu_{i,j}^n \in \left[0, \dfrac13\right]$ for
  all $i, \, j \in \interi$. Hence,
  \begin{equation}
    \label{eq:DijOKroe}
    \sum_{i,j \in \interi} \modulo{\mathcal{D}_{i,j}^n} \leq
    \sum_{i,j \in \interi} \modulo{\rh{i,j+1} - \rh{i,j}}.
  \end{equation}

  Pass now to $\mathcal{E}_{i,j}^n$: we can proceed analogously to
  $\mathcal{B}_{i,j}^n$, treating separately the terms involving $V_1$
  and those involving $F$. First, by~\eqref{eq:fx1},
  \begin{align}
    \nonumber
    & V_1 (\x{i+1/2,j}, \rh{i,j}, \rh{i+1,j})
      - V_1 (\x{i+1/2,j+1}, \rh{i,j}, \rh{i+1,j})
     \\ \nonumber
    & + V_1 (\x{i-1/2,j+1}, \rh{i-1,j}, \rh{i,j})
      - V_1 (\x{i-1/2,j}, \rh{i-1,j}, \rh{i,j})
    \\ \nonumber
    = \
    &
      \vs_1 (x_{i+1/2,j}) \rh{i,j} + \min\left\{0, \vs_1 (x_{i+1/2,j})\right\} (\rh{i+1,j} - \rh{i,j})
    \\ \nonumber
    &- \vs_1 (x_{i+1/2,j+1}) \rh{i,j} - \min\left\{0, \vs_1 (x_{i+1/2,j+1})\right\} (\rh{i+1,j} - \rh{i,j})
    \\ \nonumber
    & + \vs_1 (x_{i-1/2,j+1}) \rh{i-1,j} + \min\left\{0, \vs_1 (x_{i-1/2,j+1})\right\} (\rh{i,j} - \rh{i-1,j})
    \\ \nonumber
    & -\vs_1 (x_{i-1/2,j}) \rh{i-1,j} - \min\left\{0, \vs_1 (x_{i-1/2,j})\right\} (\rh{i,j} - \rh{i-1,j})
    \\ \nonumber
    & \pm\left( \vs_1 (x_{i-1/2,j+1}) - \vs_1 (x_{i-1/2,j})\right)\rh{i,j}
    \\ \nonumber
    = \
    & \left(  \vs_1 (x_{i+1/2,j})  -  \vs_1 (x_{i+1/2,j+1})  -  \vs_1 (x_{i-1/2,j}) + \vs_1 (x_{i-1/2,j+1})
      \right) \rh{i,j}
    \\ \nonumber
    & + \left( \vs_1 (x_{i-1/2,j+1}) - \vs_1 (x_{i-1/2,j})\right) (\rh{i-1,j} - \rh{i,j})
    \\ \nonumber
    & + \left(
       \min\left\{0, \vs_1 (x_{i+1/2,j})\right\}-\min\left\{0, \vs_1 (x_{i+1/2,j+1})\right\}
      \right) (\rh{i+1,j} - \rh{i,j})
    \\ \nonumber
    &+  \left(
      \min\left\{0, \vs_1 (x_{i-1/2,j+1})\right\} -  \min\left\{0, \vs_1 (x_{i-1/2,j})\right\}
      \right)(\rh{i,j} - \rh{i-1,j})
    \\ \label{eq:12}
    \leq \
    & \dx \, \dy \, \norma{\partial_{xy}^2 \vs_1}_{\L\infty} \modulo{\rh{i,j}}
      + \dy \, \norma{\partial_y \vs_1}_{\L\infty}
      \left( \modulo{\rh{i+1,j} - \rh{i,j}}+  2 \, \modulo{\rh{i,j} - \rh{i-1,j}}
      \right),
  \end{align}
  since
  \begin{align*}
    & \vs_1 (x_{i+1/2,j})  -  \vs_1 (x_{i+1/2,j+1})  -  \vs_1 (x_{i-1/2,j}) + \vs_1 (x_{i-1/2,j+1})
     \\
    = \
    & \dx \, \partial_x \vs_1 (\xi_i,y_j) - \dx \, \partial_x \vs_1 (\xi_i, y_{j+1})
    \\
    = \
    & - \dx \, \dy \, \partial^2_{xy} \vs_1 (\xi_i, \zeta_{j+1/2}),
  \end{align*}
  with $\xi_i \in \ ]x_{i-1/2}, x_{i+1/2}[$ and $\zeta_{j+1/2} \in \ ] y_j, y_{j+1}[$.
  In a similar way, by~\eqref{eq:fx2},
  \begin{align}
    \nonumber
    & F (\rh{i,j}, \rh{i+1,j}, J_1^n (\x{i+1/2,j}))
      - F (\rh{i,j}, \rh{i+1,j}, J_1^n (\x{i+1/2,j+1}))
    \\ \nonumber
    & + F (\rh{i-1,j}, \rh{i,j}, J_1^n (\x{i-1/2,j+1}))
      - F (\rh{i-1,j}, \rh{i,j}, J_1^n (\x{i-1/2,j}))
    \\ \nonumber
    =\
    & J_1^n (\x{i+1/2,j}) \, f (\rh{i,j})
      + \min\left\{0, J_1^n (\x{i+1/2,j})\right\} \left(f (\rh{i+1,j}) - f (\rh{i,j})\right)
    \\ \nonumber
    & - J_1^n (\x{i+1/2,j+1}) \, f (\rh{i,j})
      - \min\left\{0, J_1^n (\x{i+1/2,j+1})\right\} \left(f (\rh{i+1,j}) - f (\rh{i,j})\right)
    \\ \nonumber
    &+  J_1^n (\x{i-1/2,j+1}) \, f (\rh{i-1,j})
      + \min\left\{0, J_1^n (\x{i-1/2,j+1})\right\} \left(f (\rh{i,j}) - f (\rh{i-1,j})\right)
    \\ \nonumber
    & - J_1^n (\x{i-1/2,j}) \, f (\rh{i-1,j})
      - \min\left\{0, J_1^n (\x{i-1/2,j})\right\} \left(f (\rh{i,j}) - f (\rh{i-1,j})\right)
    \\ \nonumber
    & \pm \left( J_1^n (\x{i-1/2,j+1})  - J_1^n (\x{i-1/2,j}) \right) f (\rh{i,j})
    \\ \nonumber
    = \
    & \left(
      J_1^n (\x{i+1/2,j})   - J_1^n (\x{i+1/2,j+1})  - J_1^n (\x{i-1/2,j}) + J_1^n (\x{i-1/2,j+1})
      \right)f (\rh{i,j})
    \\ \nonumber
    & + \left( J_1^n (\x{i-1/2,j+1})  - J_1^n (\x{i-1/2,j}) \right) \left(f (\rh{i-1,j}) - f (\rh{i,j})\right)
    \\ \nonumber
    & + \left(
       \min\left\{0, J_1^n (\x{i+1/2,j})\right\}- \min\left\{0, J_1^n (\x{i+1/2,j+1})\right\}
      \right)\left(f (\rh{i+1,j}) - f (\rh{i,j})\right)
    \\ \nonumber
    & + \left(
      \min\left\{0, J_1^n (\x{i-1/2,j+1})\right\} - \min\left\{0, J_1^n (\x{i-1/2,j})\right\}
      \right)\left(f (\rh{i,j}) - f (\rh{i-1,j})\right)
    \\ \label{eq:13}
    \leq \
    & 2 \, \epsilon \, \dx \, \dy \, C \modulo{\rh{i}}
      + 2 \, \epsilon \, L_f \, \dy \, \norma{\nabla^2 \eta}_{\L\infty} \norma{\rho_o}_{\L1}
      \left(
      \modulo{ \rh{i+1,j} - \rh{i,j}}
      + 2 \modulo{\rh{i,j} + \rh{i-1,j}}
      \right),
  \end{align}
  where we used the fact that $f(r) \leq r$, \eqref{eq:J1y}
  and~\eqref{eq:Jtriplabis}, with the notation~\eqref{eq:c12}.
  Therefore, collecting together~\eqref{eq:12} and~\eqref{eq:13}, we
  get
  \begin{align*}
    \modulo{\mathcal{E}_{i,j}^n} \leq \
    & \dx \, \dy \left( \norma{\partial_{xy}^2 \vs_1}_{\L\infty}
      +  2 \, \epsilon \,C \right) \modulo{\rh{i}}
    \\
    & +\dy \left(
      \norma{\partial_y \vs_1}_{\L\infty}
      + 2 \, \epsilon \, L_f \norma{\nabla^2 \eta}_{\L\infty}\norma{\rho_o}_{\L1}
      \right)
      \left(
       \modulo{ \rh{i+1,j} - \rh{i,j}}
      + 2 \modulo{\rh{i,j} + \rh{i-1,j}}
      \right),
  \end{align*}
  so that
  \begin{align}
    \nonumber
    \sum_{i,j \in \interi} \lambda_x  \modulo{\mathcal{E}_{i,j}^n}
    \leq \
    &  3 \, \lambda_x \, \dy \left(
      \norma{\partial_y \vs_1}_{\L\infty}
      + 2 \, \epsilon \, L_f \norma{\nabla^2 \eta}_{\L\infty}\norma{\rho_o}_{\L1}
      \right)
    \sum_{i,j \in \interi} \modulo{\rh{i+1,j} - \rh{i,j}}
    \\
    \label{eq:EijOKroe}
    & + \dt   \left( \norma{\partial_{xy}^2 \vs_1}_{\L\infty}
      +  2 \, \epsilon \,C \right)  \, \dy  \sum_{i,j \in \interi} \rh{i,j}.
  \end{align}
  Hence, by~\eqref{eq:DijOKroe} and~\eqref{eq:EijOKroe}, using also
  Lemma~\ref{lem:L1roe}, we obtain
  \begin{align}
    \nonumber
    & \sum_{i,j \in \interi} \dx \, \modulo{\rho^{n+1/2}_{i,j+1} - \rho^{n+1/2}_{i,j}}
    \\\nonumber
    \leq \
    & \sum_{i,j \in \interi} \dx \left( \modulo{\mathcal{D}_{i,j}^n}
      + \lambda_x \modulo{\mathcal{E}_{i,j}^n}\right)
    \\
    \nonumber
    \leq \
    & \sum_{i,j \in \interi} \dx \, \modulo{\rh{i,j+1} - \rh{i,j}}
    \\
    \label{eq:BVdxroe}
    & + 3 \, \dt  \left(
      \norma{\partial_y \vs_1}_{\L\infty}
      +  2\, \epsilon  \, L_f \, \norma{\nabla^2 \eta}_{\L\infty} \norma{\rho_o}_{\L1}
      \right)
      \sum_{i,j \in \interi} \dy \, \modulo{\rh{i+1,j} - \rh{i,j}}
    \\
    \nonumber
    & + \dt \left(
      \norma{\partial_{xy}^2 \vs_1}_{\L\infty} + 2 \,  \epsilon  \, C  \right) \norma{\rho_o}_{\L1}.
  \end{align}
  Setting
  \begin{align}
    \label{eq:K1roe}
    K_1 = \
    & 3 \left(\norma{\partial_x \vs_1}_{\L\infty}
      + \norma{\partial_y \vs_1}_{\L\infty}
      + 4 \, \epsilon \, L_f \,
      \norma{\nabla^2 \eta}_{\L\infty} \norma{\rho_o}_{\L1}\right) ,
    \\
    \label{eq:K2roe}
    K_2 = \
    &  \left( 4 \, \epsilon   \, C
      + 2 \, \norma{\partial_{xx}^2 \vs_1}_{\L\infty} + \norma{\partial_{xy}^2 \vs_1}_{\L\infty}
      \right)
      \norma{\rho_o}_{\L1},
  \end{align}
  by~\eqref{eq:BVdyroe} and~\eqref{eq:BVdxroe} we conclude
  \begin{align*}
    & \sum_{i,j \in \interi} \left(
      \dy \, \modulo{\rho^{n+1/2}_{i+1,j} - \rho^{n+1/2}_{i,j}}
      +
      \dx \, \modulo{\rho^{n+1/2}_{i,j+1} - \rho^{n+1/2}_{i,j}}
      \right)
    \\
    \leq \
    & \left( 1 + \dt  K_1 \right)
      \sum_{i,j \in \interi} \left(
      \dx \, \modulo{\rh{i,j+1} - \rh{i,j}} +  \dy \, \modulo{\rh{i+1,j} -\rh{i,j}}
      \right)
      +  \dt \, K_2.
  \end{align*}
  Analogous computations yield
  \begin{align*}
    & \sum_{i,j \in \interi} \left(
      \dy \, \modulo{\rho^{n+1}_{i+1,j} - \rho^{n+1}_{i,j}}
      +
      \dx \, \modulo{\rho^{n+1}_{i,j+1} - \rho^{n+1}_{i,j}}
      \right)
    \\
    \leq \
    & \left( 1 + \dt \, K_3 \right)
      \sum_{i,j \in \interi} \left(
      \dx \, \modulo{\rho^{n+1/2}_{i,j+1} - \rho^{n+1/2}_{i,j}}
      +  \dy \, \modulo{\rho^{n+1/2}_{i+1,j} -\rho^{n+1/2}_{i,j}}
      \right)
      +  \dt \, K_4,
  \end{align*}
  where
  \begin{align}
    \label{eq:K3roe }
    K_3 = \
    & 3 \left( \norma{\partial_x \vs_2}_{\L\infty}
      +  \norma{\partial_y \vs_2}_{\L\infty}
      + 4 \, \epsilon \, L_f \,
      \norma{\nabla^2 \eta}_{\L\infty} \norma{\rho_o}_{\L1}\right),
    \\
    \label{eq:K4roe}
    K_4 = \
    &  \left(
      4 \, \epsilon  \, C
      + 2 \, \norma{\partial_{yy}^2 \vs_2}_{\L\infty} + \norma{\partial_{xy}^2 \vs_2}_{\L\infty}
      \right) \norma{\rho_o}_{\L1}.
  \end{align}
  Observe that, using the notation~\eqref{eq:K1defroe}
  and~\eqref{eq:K2defroe},
  \begin{align*}
    K_1, \, K_3 \leq \mathcal{K}_1,
    \qquad
    K_2, \, K_4 \leq \mathcal{K}_2.
  \end{align*}
  A recursive argument yields the desired result:
  \begin{align*}
    & \sum_{i,j \in \interi} \left(
      \dy \, \modulo{\rho^{n}_{i+1,j} - \rho^{n}_{i,j}}
      +
      \dx \, \modulo{\rho^{n}_{i,j+1} - \rho^{n}_{i,j}}
      \right)
    \\
    \leq \
    & e^{2 \, n \, \dt \, \mathcal{K}_1}  \sum_{i,j \in \interi} \left(
      \dx \, \modulo{\rho^0_{i,j+1} - \rho^0_{i,j}}
      +  \dy \, \modulo{\rho^0_{i+1,j} -\rho^0_{i,j}}
      \right)
      +  \frac{2 \, \mathcal{K}_2}{\mathcal{K}_1}\left(e^{2 \, n \, \dt \, \mathcal{K}_1} -1
      \right).
  \end{align*}
\end{proof}

\begin{corollary} {\bf ($\BV$ estimate in space and time)}
  \label{cor:bvxtroe}
  Let $\rho_o \in (\L\infty \cap \BV) (\reali^2;
  \reali^+)$. Let~\ref{vs}, \ref{H}, \ref{eta}, \eqref{eq:CFLroe-v2}
  hold. Then, for all $t>0$, $\rho_\Delta$ in~\eqref{eq:rhodelta}
  constructed through Algorithm~\ref{alg:2} satisfies the following
  estimate: for all $n=1, \ldots, N_T$,
  \begin{equation}
    \label{eq:bvxt}
    \sum_{m=0}^{n-1} \sum_{i,j \in \interi}
    \dt \left(
      \dy \, \modulo{\rho^m_{i+1,j} - \rho^m_{i,j}}
      + \dx \, \modulo{\rho^m_{i,j+1} - \rho^m_{i,j}}\right)
    +
    \sum_{m=0}^{n-1} \sum_{i,j \in \interi}
    \dx \, \dy\,  \modulo{\rho^{m+1}_{i,j} - \rho^m_{i,j}}
    \leq \mathcal{C}_{xt}(t^n),
  \end{equation}
  where
  \begin{equation}
    \label{eq:Cxt}
    \mathcal{C}_{xt} (t) = t \left(\mathcal{C}_x (t) + 2 \, \mathcal{C}_t (t)\right),
  \end{equation}
  with $\mathcal{C}_x$ as in~\eqref{eq:Cx} and $\mathcal{C}_t$ as in~\eqref{eq:Ct}.
\end{corollary}

\begin{proof}
  By Proposition~\ref{prop:bvroe} we have
  \begin{equation}
    \label{eq:9r}
    \sum_{m=0}^{n-1} \sum_{i,j \in \interi}
    \dt \left(
      \dy \, \modulo{\rho^m_{i+1,j} - \rho^m_{i,j}}
      + \dx \, \modulo{\rho^m_{i,j+1} - \rho^m_{i,j}}\right)
    \leq
    n \, \dt \, \mathcal{C}_x (n \, \dt).
  \end{equation}
  Since
  \begin{displaymath}
    \modulo{\rho^{m+1}_{i,j} - \rho^m_{i,j}}
    \leq
     \modulo{\rho^{m+1}_{i,j} - \rho^{m+1/2}{i,j}}
     +
     \modulo{\rho^{m+1/2}_{i,j} - \rho^m_{i,j}},
  \end{displaymath}
  we focus first on
  \begin{displaymath}
    \sum_{i,j \in \interi}
    \dx \, \dy \,\modulo{\rho^{m+1/2}_{i,j} - \rho^m_{i,j}}.
  \end{displaymath}
  By the scheme~\eqref{eq:scheme1roe}, we have, using the notation~\eqref{eq:notvJ},
  \begin{align*}
    \rho^{m+1/2}_{i,j} - \rho^m_{i,j} \leq \
    & \lambda_x \left[
      V_1 (x_{i+1/2,j}, \rho^m_{i,j}, \rho^m_{i+1,j}) + F (\rho^m_{i,j}, \rho^m_{i+1,j}, J^m_1 (x_{i+1/2,j}))\right.
    \\
    & \qquad \left.
      - V_1 (x_{i-1/2,j}, \rho^m_{i-1,j}, \rho^m_{i,j}) - F (\rho^m_{i-1,j}, \rho^m_{i,j}, J^m_1 (x_{i-1/2,j}))
      \right]
    \\
    = \
    & \lambda_x \left[
      v_{i+1/2,j} \, \rho^m_{i,j} + \min\left\{0, v_{i+1/2,j}\right\} (\rho^m_{i+1,j} - \rho^m_{i,j})\right.
    \\
    & -  v_{i-1/2,j} \, \rho^m_{i-1,j} - \min\left\{0, v_{i-1/2,j}\right\} (\rho^m_{i,j} - \rho^m_{i-1,j})
    \\
    & + J_1^m(x_{i+1/2,j}) \,f( \rho^m_{i,j} )+ \min\left\{0, J^m_1(x_{i+1/2,j})\right\}
      (f(\rho^m_{i+1,j}) - f(\rho^m_{i,j}))
    \\
    & -  J^m_1(x_{i-1/2,j}) \, f(\rho^m_{i-1,j}) - \min\left\{0, J^m_1(x_{i-1/2,j})\right\}
      (f(\rho^m_{i,j}) - f(\rho^m_{i-1,j}))
    \\
    & \left.
      \pm v_{i-1/2,j} \, \rho^m_{i,j} \pm J^m_1(x_{i-1/2,j})\, f (\rho^m_{i,j})
      \right]
    \\
    = \
    & \lambda_x \left[
      \dx \, \partial_x \vs_1 (\xi_i, y_j) \, \rho^m_{i,j}
      + \left(v_{i-1/2,j} -  \min\left\{0, v_{i-1/2,j}\right\} \right)(\rho^m_{i,j} - \rho^m_{i-1,j})
      \right.
    \\
    & + \min\left\{0, v_{i+1/2,j}\right\} (\rho^m_{i+1,j} - \rho^m_{i,j})
      + \left( J_1^m(x_{i+1/2,j}) - J^m_1(x_{i-1/2,j})\right)\, f (\rho^m_{i,j})
    \\
    & + \left(
      J^m_1(x_{i-1/2,j}) - \min\left\{0, J^m_1(x_{i-1/2,j})\right\}
      \right) (f(\rho^m_{i,j}) - f(\rho^m_{i-1,j}))
      \\
    & \left.+ \min\left\{0, J^m_1(x_{i+1/2,j})\right\}
      (f(\rho^m_{i+1,j}) - f(\rho^m_{i,j}))\right]
    \\
    \leq \
    & \dt \left(
      \norma{\partial_x \vs_1}_{\L\infty}
      + 2 \, \epsilon \norma{\nabla^2 \eta}_{\L\infty} \norma{\rho_o}_{\L1}
      \right) \rho^m_{i,j}
    \\
    & + \lambda_x \left( \norma{\vs_1}_{\L\infty} + \epsilon \, L_f \right)
      \left( \modulo{\rho^m_{i,j} - \rho^m_{i-1,j}} + \modulo{\rho^m_{i+1,j} -\rho^m_{i,j}}\right),
  \end{align*}
  where $\xi_i \in \ ]x_{i-1/2}, x_{i+1/2}[$ and we used $f(r) \leq r$,
  \eqref{eq:Jinf} and~\eqref{eq:Jx}.  Therefore,
  \begin{align*}
     \sum_{i,j \in \interi}
    \dx \, \dy \, \modulo{\rho^{m+1/2}_{i,j} - \rho^m_{i,j}} \leq \
    & 2 \, \dt \,  \left(\norma{\vs_1}_{\L\infty} + \epsilon \, L_f  \right)
    \sum_{i,j \in \interi} \dy \, \modulo{\rho^m_{i+1,j} - \rho^m_{i,j}}
    \\
    & + \dt \left(
      \norma{\partial_x \vs_1}_{\L\infty}
      + 2 \, \epsilon  \norma{\nabla^2 \eta}_{\L\infty} \norma{\rho_o}_{\L1}
    \right) \norma{\rho_o}_{\L1}
    \\
    \leq \
    &   2 \, \dt \,  \left( \norma{\vs_1}_{\L\infty} + \epsilon \, L_f  \right)
    \mathcal{C}_x (m \, \dt)
    \\
     & + \dt \left(
      \norma{\partial_x \vs_1}_{\L\infty}
      + 2\, \epsilon  \norma{\nabla^2 \eta}_{\L\infty} \norma{\rho_o}_{\L1}
    \right) \norma{\rho_o}_{\L1}
    \\
    \leq \
    & \dt \, \mathcal{C}_t (m \, \dt),
  \end{align*}
  where we set
  \begin{equation}
    \label{eq:Ct}
    \mathcal{C}_t (s) =
    2 \,\left(\norma{\boldsymbol{\vs}}_{\L\infty}
      + \epsilon \, L_f  \right)\mathcal{C}_x (s)
    +  \left(
      \norma{\nabla \boldsymbol{\vs}}_{\L\infty}\!\!
      + 2 \,\epsilon  \norma{\nabla^2 \eta}_{\L\infty} \norma{\rho_o}_{\L1}
    \right) \norma{\rho_o}_{\L1}.
  \end{equation}
  Analogously, we get
  \begin{align*}
    \sum_{i,j \in \interi}
    \dx \, \dy \, \modulo{\rho^{m+1}_{i,j} - \rho^{m+1/2}_{i,j}} \leq \
    & 2\,  \dt \,  \left( \norma{\vs_2}_{\L\infty} + \epsilon \, L_f  \right)
    \sum_{i,j \in \interi} \dx \, \modulo{\rho^{m+1/2}_{i,j+1} - \rho^{m+1/2}_{i,j}}
    \\
    & + \dt \left(
      \norma{\partial_y \vs_2}_{\L\infty}
      + 2\,\epsilon \norma{\nabla^2 \eta}_{\L\infty} \norma{\rho_o}_{\L1}
    \right) \norma{\rho_o}_{\L1}
    \\
    \leq \
    &2\,   \dt \,  \left(\norma{\vs_2}_{\L\infty} + \epsilon \, L_f  \right)
    \mathcal{C}_x (m \, \dt)
    \\
     & +  \dt \left(
      \norma{\partial_y \vs_2}_{\L\infty}
      + 2\, \epsilon  \norma{\nabla^2 \eta}_{\L\infty} \norma{\rho_o}_{\L1}
    \right) \norma{\rho_o}_{\L1}
    \\
    \leq \
    & \dt \, \mathcal{C}_t (m \, \dt).
  \end{align*}
  Hence
  \begin{equation}
    \label{eq:8r}
    \sum_{m=0}^{n-1} \sum_{i,j \in \interi}
    \dx \, \dy \, \modulo{\rho^{m+1}_{i,j} - \rho^m_{i,j}}
    \leq
     2 \, n \, \dt \, \mathcal{C}_t (n \, \dt),
  \end{equation}
  which together with~\eqref{eq:9r} completes the proof.
\end{proof}


%% file: LF.tex
\section{Lax--Friedrichs scheme}
\label{sec:lf}

It is also possible to consider a piece-wise constant solution
$\rho_{\Delta}$ to~\eqref{eq:1} as in~\eqref{eq:rhodelta} defined
through a Lax--Friedrichs type finite volume scheme with dimensional
splitting. The algorithm reads as follows
\begin{lgrthm}
  \label{alg:1}
  \begin{align}
    &\texttt{for } n=0,\ldots N_T-1 \nonumber
    \\
    \label{eq:numFx}
    &
    \quad F^n (x,y,u,w) = \frac{1}{2} \left[ \vs_1(x,y) (u+w) +
     J^n_1 (x,y) \left( f(u) + f(w) \right)\right] -
    \frac{\alpha}{2} \left(w - u \right)
    \\ \label{eq:numGy}
    & \quad G^n (x,y,u,w) = \frac{1}{2} \left[ \vs_2 (x,y)(u+w) +
      J^n_2(x,y) \left( f(u) + f(w) \right)\right] -
    \frac{\beta}{2} \left(w - u \right)
    \\ \label{eq:scheme1}
    & \quad
    \rho_{i,j}^{n+1/2} = \rh{i,j} - \lambda_x \left[F^n (
      \x{i+\sfrac12,j}, \rh{i,j}, \rh{i+1,j}) - F^n (
      \x{i-\sfrac12,j}, \rh{i-1,j}, \rh{i,j})\right]
    \\ \label{eq:scheme2}
    & \quad \rho_{i,j}^{n+1} =
    \rho_{i,j}^{n+1/2} - \lambda_y \left[G^n ( \x{i,j+\sfrac12},
      \rh{i,j}, \rh{i,j+1}) - G^n(\x{i,j-\sfrac12}, \rh{i,j-1},
      \rh{i,j})\right]
    \\
    & \texttt{end} \nonumber
  \end{align}
\end{lgrthm}
The algorithm is close to that studied in~\cite{ACG2015}, except that
in the present case to compute $\rho^{n+1}$ the flux is evaluated at
$\rho^{n}$, instead of $\rho^{n+1/2}$.

Following closely the proofs presented in~\cite{ACG2015}, it is
possible to recover also for Algorithm~\ref{alg:1} the bounds on the
approximate solution necessary to prove the convergence.
Below, we state only the final results, omitting the computations.

\begin{lemma}{\bf (Positivity)}
  \label{lem:pos}
  Let $\rho_o \in \L\infty (\reali^2; \reali^+)$. Let assumptions~\ref{vs}, \ref{H} and~\ref{eta} hold.
  Assume that
  \begin{align}
    \label{eq:CFL1}
    \alpha \geq \
    & \norma{\vs_1}_{\L\infty} + \epsilon \, L_f,
    &
    \lambda_x \leq \
    & \frac13 \, \min\left\{\frac1\alpha, \frac{1}{2 \, \epsilon \, L_f + \dx \,
        \norma{\vs_1}_{\L\infty}}\right\},
    \\
    \label{eq:CFL2}
    \beta \geq \
    & \norma{\vs_2}_{\L\infty} + \epsilon \, L_f,
    & \lambda_y \leq\
    &\frac13
    \, \min\left\{\frac1\beta, \frac{1}{2 \, \epsilon \, L_f + \dx \,
        \norma{\vs_2}_{\L\infty}}\right\}.
  \end{align}
  Then, for all $t>0$ and $(x,y) \in \reali^2$, the piece-wise constant
  approximate solution $\rho_\Delta$~\eqref{eq:rhodelta} constructed through Algorithm~\ref{alg:1} is such that
  $\rho_\Delta (t,x,y)\geq 0$.
\end{lemma}

\begin{lemma} {\bf ($\L1$ bound)}
  \label{lem:l1}
  Let $\rho_o \in \L\infty (\reali^2; \reali^+)$. Let~\ref{vs}, \ref{H}, \ref{eta}, \eqref{eq:CFL1}
  and~\eqref{eq:CFL2} hold. Then, for all $t>0$, $\rho_\Delta$
  in~\eqref{eq:rhodelta} constructed through Algorithm~\ref{alg:1}  satisfies~\eqref{eq:l1}, that is
  \begin{displaymath}
    \norma{\rho_\Delta (t)}_{\L1 (\reali^2)}
    = \norma{\rho_o}_{\L1(\reali^2)}.
  \end{displaymath}
\end{lemma}


\begin{lemma} {\bf ($\L\infty$ bound)}
  \label{lem:linf}
  Let $\rho_o \in \L\infty (\reali^2; \reali^+)$.Let~\ref{vs}, \ref{H}, \ref{eta}, \eqref{eq:CFL1}
  and~\eqref{eq:CFL2} hold. Then, for all $t>0$, $\rho_\Delta$
  in~\eqref{eq:rhodelta} constructed through Algorithm~\ref{alg:1}  satisfies
  \begin{displaymath}
    \norma{\rho_\Delta (t)}_{\L\infty (\reali^2)}
    \leq  \norma{\rho_o}_{\L\infty} \, e^{\mathcal{\tilde{C}}_\infty \, t},
  \end{displaymath}
  where
  \begin{displaymath}
    \mathcal{\tilde{C}}_\infty =
      \norma{\partial_x \vs_1}_{\L\infty}
      + \norma{\partial_y \vs_2}_{\L\infty}
      + 4 \, \epsilon \, L_f \, \norma{\nabla^2 \eta}_{\L\infty} \norma{\rho_o}_{\L1}.
  \end{displaymath}
\end{lemma}

\begin{remark}
    Compare the $\L\infty$ estimate obtained in Lemma~\ref{lem:linf} using the Lax--Friedrichs scheme with that in Lemma~\ref{lem:Linfroe}, given by the Roe scheme. Although they look very similar, the constants appearing in the exponent are actually different: when comparing $\mathcal{\tilde{C}}_\infty$ above to $\mathcal{C}_\infty$ as in~\eqref{eq:Cinf}, we see that in $\mathcal{\tilde{C}}_\infty$ the last addend is multiplied by $L_f$.
\end{remark}

\begin{proposition} {\bf ($\BV$ estimate in space)}
  \label{prop:bv}
  Let $\rho_o \in (\L\infty \cap \BV) (\reali^2; \reali^+)$. Let~\ref{vs}, \ref{H}, \ref{eta}, \eqref{eq:CFL1} and~\eqref{eq:CFL2} hold. Then, for all $t>0$, $\rho_\Delta$
  in~\eqref{eq:rhodelta} constructed through Algorithm~\ref{alg:1} satisfies the following estimate: for all
 $n=0, \ldots, N_T$,
  \begin{displaymath}
    \sum_{i,j \in \interi} \left(
      \dy \, \modulo{\rh{i+1,j} - \rh{i,j}}
      + \dx \, \modulo{\rh{i,j+1} - \rh{i,j}}\right)
    \leq \tilde{\mathcal{C}}_x(t^n),
  \end{displaymath}
  where
  \begin{equation}
    \label{eq:Cxlf}
    \tilde{\mathcal{C}}_x(t) = e^{2 \,t \, \mathcal{\tilde{K}}_1}  \sum_{i,j \in \interi} \left(
      \dx \, \modulo{\rho^0_{i,j+1} - \rho^0_{i,j}}
      +  \dy \, \modulo{\rho^0_{i+1,j} -\rho^0_{i,j}}
    \right)
    +  \frac{2 \, \mathcal{K}_2}{\mathcal{\tilde{K}}_1}\left(e^{2 \, t \, \mathcal{\tilde{K}}_1} -1
    \right),
  \end{equation}
  with
  \begin{align*}
    \mathcal{\tilde{K}}_1 = \
    &
    2 \, \norma{\nabla \boldsymbol{\vs}}_{\L\infty} + 4 \, \epsilon \, L_f \,
    \norma{\nabla^2 \eta}_{\L\infty} \norma{\rho_o}_{\L1},
  \end{align*}
  $\mathcal{K}_2$ as in~\eqref{eq:K2defroe}
  and $c_1, \, c_2 $ are defined in~\eqref{eq:c12}.
\end{proposition}

\begin{remark}
Observe that $\tilde{\mathcal{K}_1} < \mathcal{K}_1$ in~\eqref{eq:K1defroe}.
\end{remark}

\begin{corollary} {\bf ($\BV$ estimate in space and time)}
  \label{cor:bvxt}
  Let $\rho_o \in (\L\infty \cap \BV) (\reali^2; \reali^+)$. Let~\ref{vs}, \ref{H}, \ref{eta}, \eqref{eq:CFL1}
  and~\eqref{eq:CFL2} hold. Then, for all $t>0$, $\rho_\Delta$
  in~\eqref{eq:rhodelta} constructed through Algorithm~\ref{alg:1} satisfies the following estimate: for all
  $n=1, \ldots, N_T$,
  \begin{displaymath}
    \sum_{m=0}^{n-1} \sum_{i,j \in \interi}
    \dt \left(
      \dy \, \modulo{\rho^m_{i+1,j} - \rho^m_{i,j}}
      + \dx \, \modulo{\rho^m_{i,j+1} - \rho^m_{i,j}}\right)
    +
    \sum_{m=0}^{n-1} \sum_{i,j \in \interi}
    \dx \, \dy\,  \modulo{\rho^{m+1}_{i,j} - \rho^m_{i,j}}
    \leq \tilde{\mathcal{C}}_{xt}(t^n),
  \end{displaymath}
  where
  \begin{displaymath}
    \tilde{\mathcal{C}}_{xt} (t) = t \left(\tilde{\mathcal{C}}_x (t) + 2 \, \tilde{\mathcal{C}}_t (t)\right),
  \end{displaymath}
  with $\tilde{\mathcal{C}}_x$ as in~\eqref{eq:Cxlf} and
  \begin{displaymath}
    \tilde{\mathcal{C}}_t (s) =
    2 \, \left(\norma{\boldsymbol{\vs}}_{\L\infty}
      + \epsilon \, L_f  \right)\tilde{\mathcal{C}_x} (s)
    +  \left(
      \norma{\nabla \boldsymbol{\vs}}_{\L\infty}\!\!
      + \epsilon \, \norma{\nabla^2 \eta}_{\L\infty} \norma{\rho_o}_{\L1}
    \right) \norma{\rho_o}_{\L1}.
  \end{displaymath}
\end{corollary}

%% file: numericalresults.tex
\section{Numerical results}
\label{sec:results}

We consider the test setting given in \cite{original} to compare the results of the Roe scheme, cf.~Algorithm~\ref{alg:2}, to the results of the Lax-Friedrichs type scheme, cf.~Algorithm~\ref{alg:1}.


\subsection{Test setting}

A total number of $N=192$ parts in the shape of metal cylinders are transported on a conveyor belt moving with speed $v_T=$ 0.42 m/s
and are redirected by a diverter. The diverter is positioned at an angle of $\theta = 45$ degree with respect to the border of the conveyor belt. Figure~\ref{img:staticvelocityfield} illustrates the static velocity field of the conveyor belt. Parts are transported with velocity $\boldsymbol{\vs}= (v_T, 0)$ in Region A and the diverter redirects the parts (Region C). The area behind the diverter (Region B) is modelled in a way that should prohibit parts from passing through the diverter, see~\cite{original}. The point $(x_d,y_d)$ marks the end of the diverter.

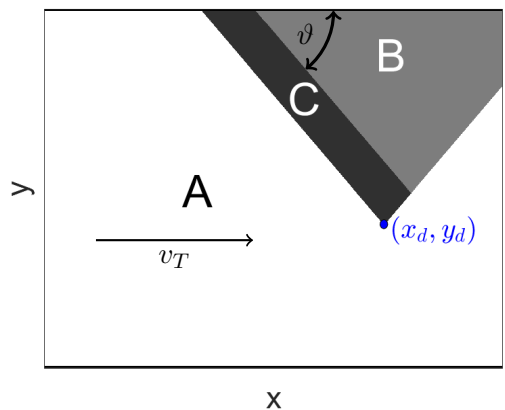
\begin{figure}[H]
\centering
\setlength{\abovecaptionskip}{2pt}
\begin{tikzpicture}
\node[anchor=south west,inner sep=0] (Bild) at (0,0)
{\includegraphics[width=0.5\textwidth]{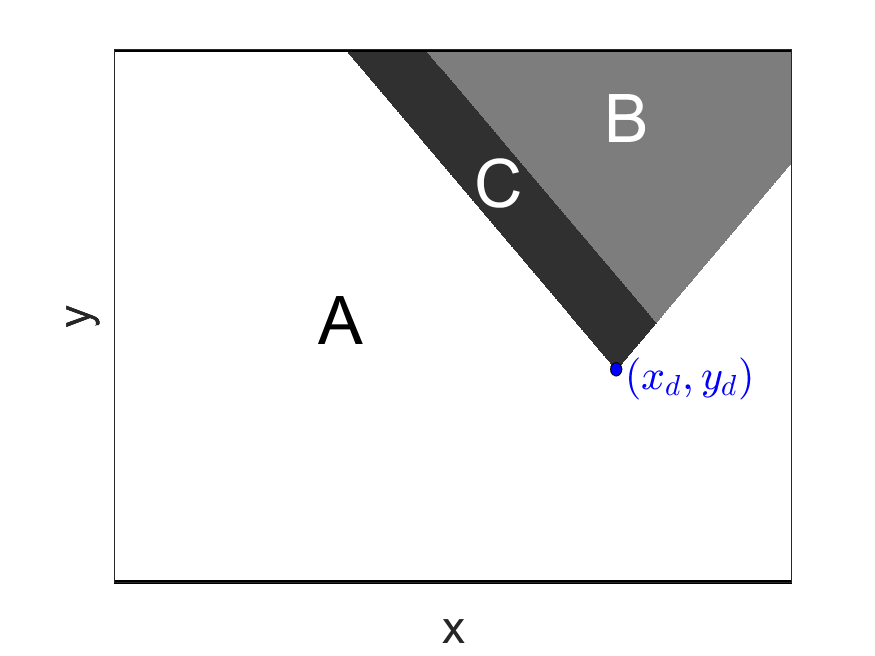}};
\begin{scope}[x=(Bild.south east),y=(Bild.north west)]
	\node (E5) at (0.2,0.4) {};
\node (E6) at (0.5,0.4) {};
\draw[->,thick] (E5) to node [below,text width = 2cm, align=center] {$v_T$}(E6);

  \draw
(0.7,0.58) coordinate (a) node[right] {}
(0.49,0.92) coordinate (b) node[left] {}
(0.7,0.92) coordinate (c) node[above right] {}
pic["$\theta$", very thick, draw=black, <->, angle eccentricity=0.7, angle radius=1cm]
{angle=a--b--c};

\end{scope}
\end{tikzpicture}
\caption{Schematic view of the static field of the conveyor belt.}
\label{img:staticvelocityfield}
\setlength{\abovecaptionskip}{10pt}
\end{figure}

\subsection{Discretisation and solution properties}

To numerically model the setting, we introduce a uniform grid $\Delta x = \Delta y $ on the selected area of the conveyor belt. Initial conditions for the density at time $t=0$ are given by the experimental data and normalised so that $\rho_{max}=1$. The mollifier $\eta$ which is used in the operator $\boldsymbol{I}(\rho)$~\eqref{eq:vdyn} is chosen as follows
\begin{align*}
\eta(x) = \frac{\sigma}{2 \pi} e^{-\frac{1}{2} \sigma \vert \vert x \vert \vert_2^2},
\end{align*}
with $\sigma = 10 000$.

In the original model formulation~\cite{original}, the Heaviside function was introduced to avoid densities larger than $\rho_{max}$. In this work, we investigate the performances of two numerical schemes with two types of smooth approximations of the Heaviside function, one sensibly closer than the other. The former is the approximation $H_t$ (atan), obtained using the inverse tangent
\begin{equation}
\label{eq:Ht}
    H_t(u) = \frac{\arctan(50 (u-1))}{\pi} + 0.5,
\end{equation}
while the latter is denoted $H_p$ (polynomial) and it is obtained by cubic spline interpolation with the following conditions
\begin{align*}
H_p(u) = \
& 0
& \forall u \leq  \
 d_l < 1&,
& H_p(u) = \
& 1
& \forall u  \geq \
d_r > 1, & \\
H_p(d_l) = \
& 0,
& H_p(1) = \
& \frac{1}{2},
& H_p(d_r) = \
& 1,
& H_p'(d_l) =\
& 0,
& H_p'(d_r) = \
& 0.
\end{align*}
The approximation $H_p$ for $d_l = \frac{1}{2}$ and $d_r = \frac{8}{5}$ together with the inverse tangent approximation are depicted in Figure~\ref{img:Heaviside_approx}.

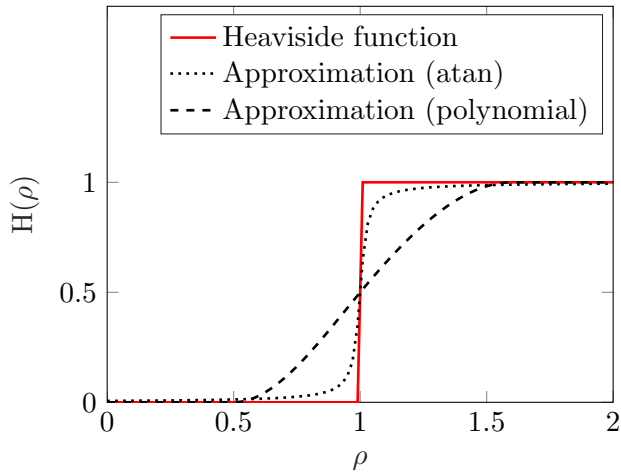
\begin{figure}[!h]
    \centering
    \setlength{\abovecaptionskip}{0pt}
    \input{./zzzHeavi3.tex}
    \caption{The Heaviside function and approximations $H_t$ (atan) and $H_p$ (polynomial).}
    \label{img:Heaviside_approx}
    \setlength{\abovecaptionskip}{10pt}
\end{figure}
Using the inverse tangent approximation corresponds to activating the collision operator $\boldsymbol{I}(\rho)$ very close to $\rho_{max}$. On the other hand, with the polynomial approximation, the collision operator starts activating at $\frac{1}{2} \rho_{max}$, which implies that clusters with densities values between $\frac{1}{2} \rho_{max}$ and $\rho_{max}$ are already dispersed to some extent. Numerically, this means that densities above the maximum one are more likely to appear when using the inverse tangent approximation, while exploiting the polynomial approximation prevents from reaching such high values of the density.

Clearly, different approximations of the Heaviside function lead to different Lipschitz constants $L_f$~\eqref{eq:L} and therefore influence the CFL time steps of the Roe scheme~\eqref{eq:CFLroe-v2} and of the Lax-Friedrichs scheme~\eqref{eq:CFL1}--\eqref{eq:CFL2}. Moreover, the constants $\alpha$ and $\beta$ given by the CFL condition for the Lax-Friedrichs scheme depend on the Lipschitz constant $L_f$. Larger Lipschitz constants, and therefore higher viscosity coefficients $\alpha$ and $\beta$, add additional viscosity to the Lax-Friedrichs scheme and therefore more diffusion, as shown in~\cite{leveque}. Note that in general, the Lax-Friedrichs scheme is more diffusive than the Roe scheme. To ensure conservation of mass within the given area of the conveyor belt, we impose zero--flux--conditions at the boundaries of the conveyor belt for the Lax-Friedrichs scheme. Therefore, at the boundary, the flux that would exit the domain is set to zero.

The Lipschitz constants of the approximations of the Heaviside functions depicted in Figure~\ref{img:Heaviside_approx}, as well as the corresponding CFL conditions, are displayed in Table~\ref{tab:Heaviside_approx_CFL}, for a fixed space step size $\Delta x$. The inverse tangent approximation has a greater Lipschitz constant, leading to small CFL time steps and thus to an increased computational effort.

\begin{table}[H]
    \centering
    \begin{tabular}{l |c | c | c | c}
         Approximation & $L_f$ & $\Delta x =  \Delta y$ [m] & CFL time step Roe [s] & CFL time step LxF [s]  \\
         \toprule
         $H_{t}$ (atan)  & $16.42$ & $1 \cdot 10^{-2}$ &$2.37 \cdot 10^{-4}$& $1.21\cdot 10^{-4}$ \\
         $H_p$ (polynomial) &  $2.09$ & $1 \cdot 10^{-2}$ & $1.63 \cdot 10^{-3}$ & $9.50 \cdot 10^{-4}$\\
        \bottomrule
    \end{tabular}
    \caption{CFL time steps for different Heaviside approximations and fixed  $\Delta x, \Delta y$.}
    \label{tab:Heaviside_approx_CFL}
\end{table}

We analyze the amount of parts that pass the obstacle, i.e. the outflow at the end of the obstacle $(x_d,y_d)$. The time-dependent mass function $U(t)$ counts the measured parts that are located in the region $\Omega_0$
\begin{align}
    U(t) &= \frac{1}{N}\sum_{i=1}^{N} \mathcal{X}_{\Omega_0}(x^{(i)}(t),y^{(i)}(t)) & \mathcal{X}_{\Omega_0}(x,y) = \begin{cases} 1, &(x,y) \in \Omega_0 \\ 0, & \text{otherwise}\end{cases}
\end{align}
where $\Omega_0 = \{(x,y) \in \mathbb{R}^2 \vert ~ x \leq x_d\}$ is the left sided region upstream the obstacle and $(x^{(i)}(t),y^{(i)}(t))$ is the position of part $i$, $i \in {1, \dots, N}$, at time $t$. The time-dependent mass function describing the outflow to the solution of the conservation law is given by
\begin{align}
    U_{\rho}(t) &= \frac{1}{\int_{\Omega_0} \rho(x,0) \, \mathrm{dx}} \, \int_{\Omega_0} \rho(x,t) \, \mathrm{dx}.
\end{align}
The outflow curves obtained using Roe scheme, Lax-Friedrichs scheme and the outflow measured experimentally are shown in Figure~\ref{img:LXFvsROEvsData_outflow}. The parameters chosen for each scheme are those given in Table~\ref{tab:Heaviside_approx_CFL}. Figure~\ref{img:Linf} displays the $\L\infty$ norms of the solution over time.
\begin{figure}[H]
    \setlength{\abovecaptionskip}{0pt}
	\begin{minipage}[t]{.45\linewidth}
		\centering
	    \input{./zzzOutflow}
	    \caption{Outflow.}
	    \label{img:LXFvsROEvsData_outflow}
	\end{minipage}
	\hfill
	\begin{minipage}[t]{.45\linewidth}
	    \centering
	    \input{./zzzLinf.tex}
	    \caption{Time evolution of the $\L\infty$ norm of the approximate solutions.}
	    \label{img:Linf}
	\end{minipage}
	\setlength{\abovecaptionskip}{10pt}
\end{figure}

The outflow curve given by Roe scheme for both approximations of the Heaviside function is closer to the experimental data, due to the fact that the scheme captures more congestion, as indicated also by the $\L\infty$ norm. With Roe scheme, as expected, the density piles up even more when using the inverse tangent approximation. We observe that a maximum principle is not verified. On the contrary, with the polynomial approximation, higher densities are avoided, since the collision operator is activated earlier. Due to the influence of the viscosity coefficients, an opposite behaviour is observable with the Lax-Friedrichs scheme. Results for the $\L\infty$ norm of the solution are quite promising using the polynomial approximation, whereas the viscosity of the scheme is too large in the case of the inverse tangent approximation. The $\L\infty$ norm of the solution is constantly decreasing over time because of diffusion.

Figure~\ref{img:LXFvsROEvsData_densityplot} displays the parts' positions in the experiment and the density distribution computed with Roe and Lax-Friedrichs scheme using the polynomial Heaviside approximation at time $t=1.5 \mbox{s}$. The density plot of the Roe scheme matches the experimental data quite well: regions with higher densities mostly coincide with regions in the experiments, where the parts are side by side. In contrast, the Lax-Friedrichs scheme produces a more widely spread density distribution. Even the parts on the upper section of the belt, which are transported to the right with the velocity of the conveyor belt, are not correctly portrayed.


\begin{figure}[H]
	\begin{minipage}[c]{.3\linewidth}
		\centering
	    \scalebox{-1}[1]{\includegraphics[width=\textwidth, angle=180]{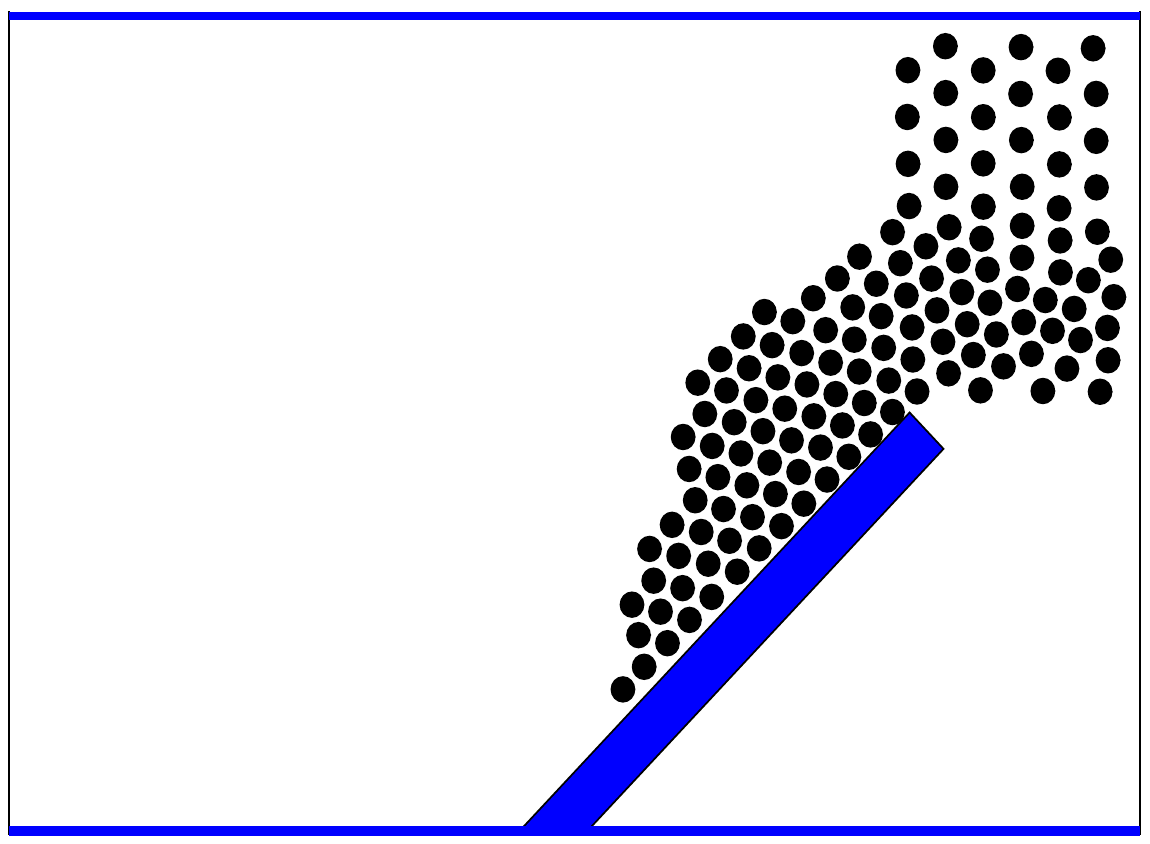}}
	    \label{img:Density_Data}
	\end{minipage}
	\hfill
	\begin{minipage}[c]{.3\linewidth}
	    \centering
	    \scalebox{-1}[1]{\includegraphics[width=\textwidth, angle=180]{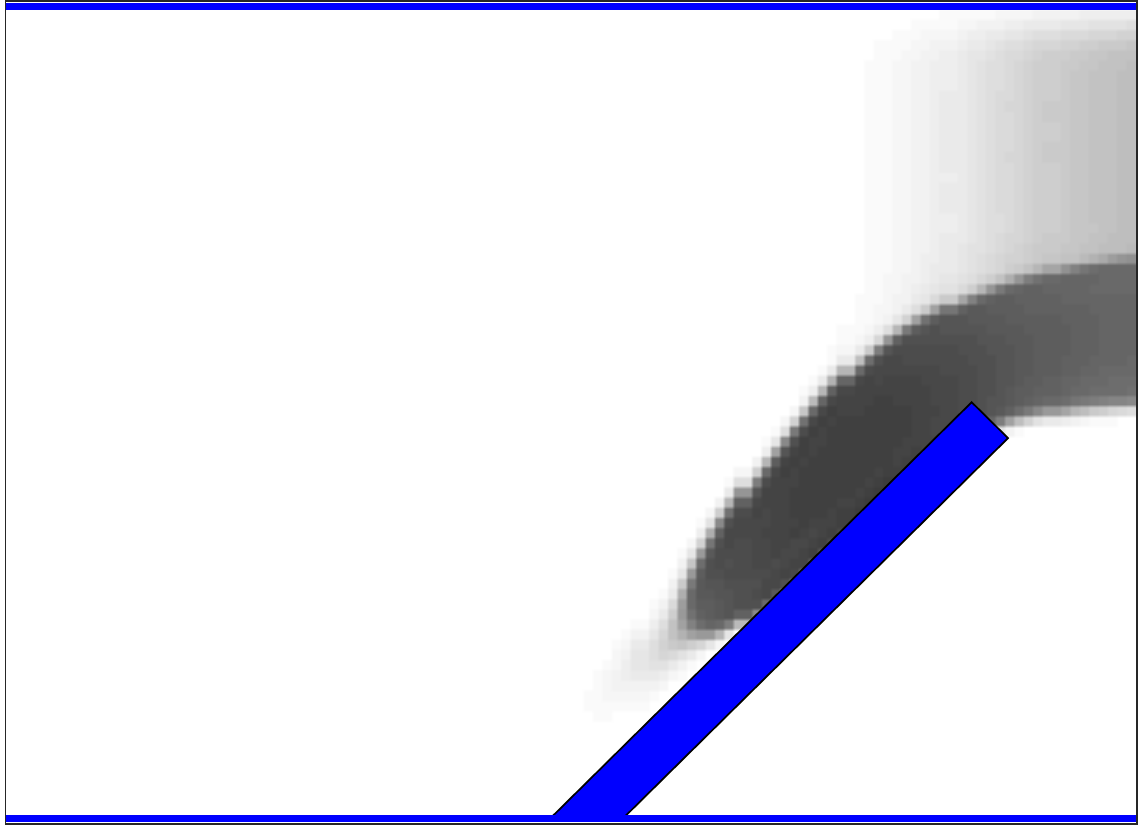}}
	    \label{img:Density_ROE}
	\end{minipage}
	\hfill
		\begin{minipage}[c]{.3\linewidth}
		\centering
	    \scalebox{-1}[1]{\includegraphics[width=\textwidth, angle = 180]{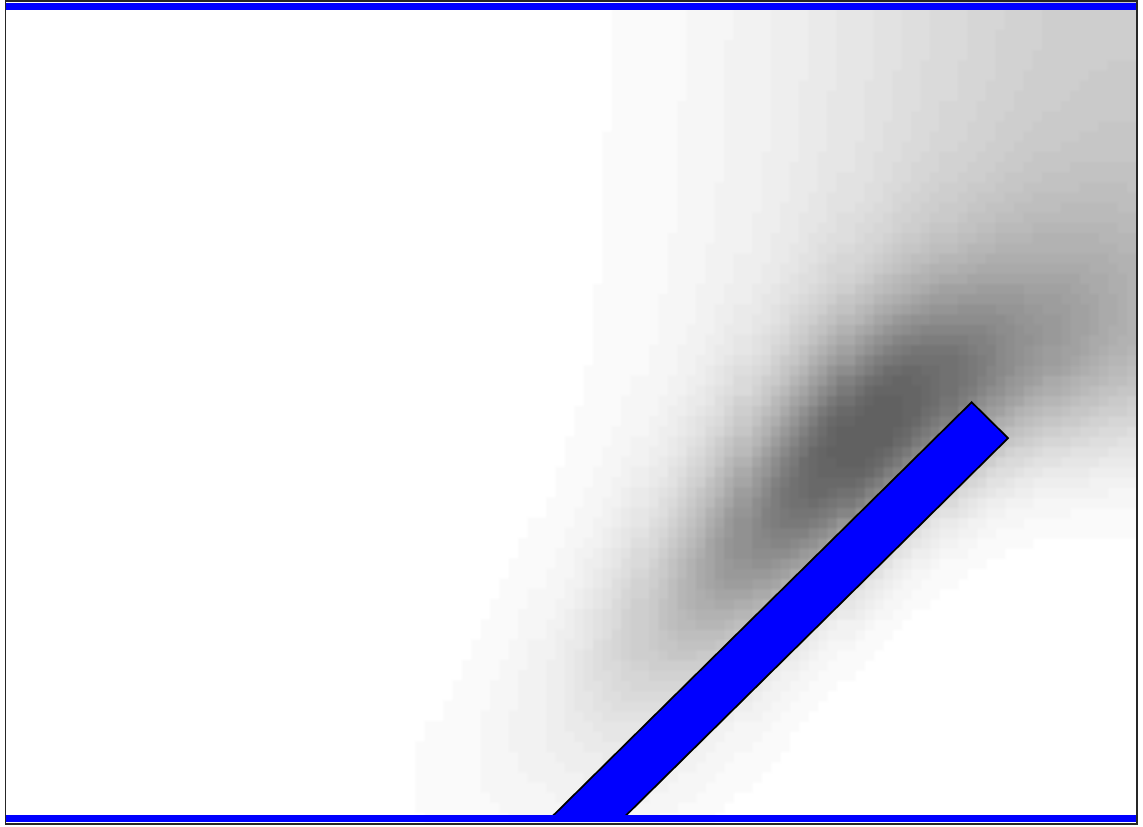}}
	    \label{img:Density_LXF}
	\end{minipage}
	\caption{Experimental data (left), results of the Roe scheme (middle) and the Lax-Friedrichs scheme (right)  at $t = 1.5$ s.}
	\label{img:LXFvsROEvsData_densityplot}
	\setlength{\abovecaptionskip}{10pt}
\end{figure}

Since the Roe scheme using the sharper approximation of the Heaviside function provides the best result in comparison to the experimental data, we analyse its behaviour for $\Delta x, \Delta y \rightarrow 0 $. Table~\ref{tab:convergence_ROE} shows the error norms of the outflow of the simulations with the Roe scheme and the inverse tangent approximation of the Heaviside function compared to to the outflow given by the experimental data.

\begin{table}[H]
    \centering
    \begin{tabular}{c |c | c | c | c}
         $\Delta x = \Delta y$ &  CFL time step & $\L1$-error & $\L2$-error &$\L\infty$-error \\
         \toprule
         $4 \cdot 10^{-2}$ & $9.47 \cdot 10^{-4}$ & $0.42$ & $0.26$ & $0.20$  \\
         $2 \cdot 10^{-2}$ & $4.73 \cdot 10^{-4}$ & $0.16$ & $0.10$ & $0.10$ \\
         $1 \cdot 10^{-2}$ & $2.37 \cdot 10^{-4}$ & $0.09$ & $0.07$ & $0.09$   \\
         $5 \cdot 10^{-3}$ & $1.18 \cdot 10^{-4}$ & $0.07$ & $0.05$ & $0.07$\\
        \bottomrule
    \end{tabular}
    \caption{Error norms of Roe scheme, with the inverse tangent approximation of the Heaviside function~\eqref{eq:Ht}, against experimental data. }
    \label{tab:convergence_ROE}
\end{table}

The scheme is evaluated for different space step sizes and their corresponding CFL time steps. We observe that the error of the Roe scheme decreases as the space step decreases, suggesting the convergence of the outflow to the experimental data, compare also Figure~\ref{img:ROE_convergence}. 

\begin{figure}[H]
        \setlength{\abovecaptionskip}{0pt}
		\centering
	    \input{./zzzRoe_convergence}
	    \caption{Outflow computed by Roe scheme with different space step sizes using the inverse tangent approximation of the Heaviside function~\eqref{eq:Ht}.}
	    \label{img:ROE_convergence}
	    \setlength{\abovecaptionskip}{10pt}
\end{figure}
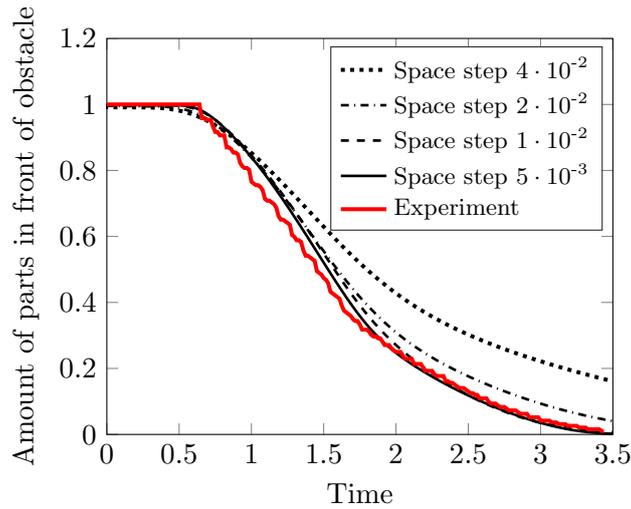

%% file: zzzHeavi3.tex
%
%
\begin{tikzpicture}
\setlength\fwidth{0.45\textwidth}
\begin{axis}[%
width=0.951\fwidth,
height=0.75\fwidth,
at={(0\fwidth,0\fwidth)},
scale only axis,
xmin=0,
xmax=2,
xtick={  0, 0.5,   1, 1.5,   2},
xlabel style={font=\color{white!15!black}},
xlabel={$\rho$},
ymin=0,
ymax=1.8,
ytick={  0, 0.5,   1},
ylabel style={font=\color{white!15!black}},
ylabel={$\text{H(}\rho\text{)}$},
axis background/.style={fill=white},
legend style={legend cell align=left, align=left, draw=white!15!black}
]
\addplot [color=red, line width=1.0pt]
  table[row sep=crcr]{%
0	0\\
0.98989898989899	0\\
1.01010101010101	1\\
2	1\\
};
\addlegendentry{Heaviside function}

\addplot [color=black, dotted, line width=1.0pt]
  table[row sep=crcr]{%
0	0.00636534910097275\\
0.242424242424243	0.00840142953776679\\
0.404040404040404	0.0106782564627546\\
0.525252525252526	0.0134017260812009\\
0.606060606060606	0.0161464849887922\\
0.666666666666667	0.0190757242358361\\
0.707070707070707	0.021699206071486\\
0.747474747474747	0.0251576291011752\\
0.787878787878788	0.0299236126127123\\
0.808080808080808	0.0330519391150959\\
0.828282828282828	0.0369074510534668\\
0.848484848484849	0.0417753906954972\\
0.868686868686869	0.0481112952731233\\
0.888888888888889	0.0566887428985092\\
0.909090909090909	0.068930102930004\\
0.92929292929293	0.0877440811163361\\
0.949494949494949	0.120019809612001\\
0.96969696969697	0.18569339545891\\
0.98989898989899	0.351132881700148\\
1.01010101010101	0.648867118299854\\
1.03030303030303	0.81430660454109\\
1.05050505050505	0.879980190387999\\
1.07070707070707	0.912255918883664\\
1.09090909090909	0.931069897069996\\
1.11111111111111	0.943311257101491\\
1.13131313131313	0.951888704726877\\
1.15151515151515	0.958224609304503\\
1.17171717171717	0.963092548946533\\
1.19191919191919	0.966948060884904\\
1.23232323232323	0.972665063811102\\
1.27272727272727	0.97669898450894\\
1.31313131313131	0.979696818503264\\
1.37373737373737	0.982982351675549\\
1.45454545454545	0.986003392826632\\
1.55555555555556	0.988545790606897\\
1.6969696969697	0.990868396116429\\
1.8989898989899	0.992919667114612\\
2	0.993634650899027\\
};
\addlegendentry{Approximation (atan)}

\addplot [color=black, dashed, line width=1.0pt]
  table[row sep=crcr]{%
0	0\\
0.505050505050505	8.2004477485853e-05\\
0.525252525252526	0.00201849152941502\\
0.545454545454545	0.0064374624342598\\
0.565656565656566	0.0132174948263373\\
0.585858585858586	0.0222371663399632\\
0.606060606060606	0.0333750546094556\\
0.626262626262626	0.0465097372691305\\
0.646464646464647	0.0615197919533053\\
0.666666666666667	0.0782837962962963\\
0.686868686868687	0.0966803279324209\\
0.707070707070707	0.116587964495996\\
0.727272727272728	0.137885283621337\\
0.747474747474747	0.160450862942763\\
0.767676767676768	0.18416328009459\\
0.787878787878788	0.208901112711133\\
0.828282828282828	0.260967334875641\\
0.868686868686869	0.315678150510822\\
0.92929292929293	0.400578095321648\\
1.01010101010101	0.513956865151876\\
1.05050505050505	0.568760577409644\\
1.09090909090909	0.621695191585274\\
1.13131313131313	0.672452573382019\\
1.17171717171717	0.720724588503132\\
1.21212121212121	0.766203102651863\\
1.23232323232323	0.787798504893853\\
1.25252525252525	0.808579981531466\\
1.27272727272727	0.828509015777611\\
1.29292929292929	0.847547090845193\\
1.31313131313131	0.86565568994712\\
1.33333333333333	0.882796296296296\\
1.35353535353535	0.89893039310563\\
1.37373737373737	0.914019463588028\\
1.39393939393939	0.928024990956396\\
1.41414141414141	0.940908458423641\\
1.43434343434343	0.952631349202668\\
1.45454545454545	0.963155146506386\\
1.47474747474747	0.9724413335477\\
1.4949494949495	0.980451393539517\\
1.51515151515152	0.987146809694744\\
1.53535353535354	0.992489065226286\\
1.55555555555556	0.996439643347051\\
1.57575757575758	0.998960027269945\\
1.5959595959596	1.00001170020787\\
2	1\\
};
\addlegendentry{Approximation (polynomial)}

\end{axis}
\end{tikzpicture}%

%% file: zzzOutflow.tex
%
%
\begin{tikzpicture}
\setlength\fwidth{0.9\textwidth}
\begin{axis}[%
width=0.951\fwidth,
height=0.75\fwidth,
at={(0\fwidth,0\fwidth)},
scale only axis,
xmin=0,
xmax=3.5,
xlabel style={font=\color{white!15!black}},
xlabel={Time},
ymin=0,
ymax=1.35,
ylabel style={font=\color{white!15!black}},
ylabel={Amount of parts in front of obstacle},
axis background/.style={fill=white},
legend style={font=\footnotesize,legend cell align=left, align=left, draw=white!15!black}
]
\addplot [color=black, dashed, line width=1.0pt]
  table[row sep=crcr]{%
0	1\\
0.000473283163039184	0.997846056835621\\
0.434947226832998	0.997461933879262\\
0.483222109462993	0.996570158958701\\
0.51919162985397	0.995160474316456\\
0.549718393869997	0.993207115802448\\
0.577168817326269	0.990688816541875\\
0.602962749711903	0.987551304939426\\
0.62757347418994	0.983787517143068\\
0.651710915504938	0.979322401840354\\
0.675848356819936	0.97407313457281\\
0.700222439716453	0.967981213267439\\
0.725069805776009	0.960977932584009\\
0.750863738161645	0.952906199628595\\
0.777840878454877	0.943655547395906\\
0.806474509818747	0.933017473187273\\
0.837237915416293	0.920753494448084\\
0.870367736829035	0.906703436356755\\
0.906337257220012	0.890598974497493\\
0.945383118170743	0.872262511293649\\
0.987741961262749	0.851508163505399\\
1.03270386175147	0.828618271401216\\
1.07955889489235	0.803902337101392\\
1.12712385277778	0.777946465261408\\
1.17492545224474	0.750989026530999\\
1.22343697645626	0.722746112006918\\
1.27431491648297	0.692229077395815\\
1.32637606441728	0.660103508509797\\
1.3822234776559	0.624717481116529\\
1.4475365541553	0.582391330266776\\
1.63069713825146	0.463091881637657\\
1.67613232190322	0.434807228750111\\
1.71636139076155	0.410648044517952\\
1.75398740222317	0.388938904491048\\
1.78995692261414	0.369067377860376\\
1.82497987667904	0.350590095244384\\
1.85976618916242	0.333103279633425\\
1.89478914322732	0.316360305354523\\
1.93028538045526	0.300245370034543\\
1.96672818400927	0.284548399013821\\
2.00435419547089	0.269179978651322\\
2.04363669800314	0.253966624955953\\
2.08504897476907	0.238758644694698\\
2.12882766735019	0.223508790086382\\
2.17497277574651	0.20825220690182\\
2.22372094153954	0.192947807709472\\
2.27483552314777	0.1777101558144\\
2.32784323740816	0.16271291128748\\
2.38203415957615	0.148176644928328\\
2.43717164807021	0.134179613290118\\
2.49254577814579	0.120910280495828\\
2.54815654980289	0.108371086132477\\
2.60376732146	0.0966173661181928\\
2.6593780931171	0.0856466747129199\\
2.7149888647742	0.0754562606804683\\
2.77083627801282	0.066000194529384\\
2.82692033283297	0.0572800483540377\\
2.88300438765311	0.0493341615501279\\
2.93932508405477	0.0421297384704724\\
2.99564578045643	0.0356964478081596\\
3.05267640160265	0.0299539510299396\\
3.11112687223799	0.0248433604977629\\
3.171707117107	0.0203195949780164\\
3.23583698569881	0.0163056451504295\\
3.30469968592101	0.0127710272694008\\
3.37995170884424	0.00968509126504857\\
3.46372282870217	0.00703037984941179\\
3.49992899067467	0.00609501108645105\\
};
\addlegendentry{Roe (atan)}

\addplot [color=black, dashdotted, line width=1.0pt]
  table[row sep=crcr]{%
0	1\\
0.00162650112530338	0.997846709944956\\
0.393613272323412	0.997392435197911\\
0.4359023015813	0.996340130161868\\
0.468432324087367	0.994629269628009\\
0.496082843217524	0.992256827046816\\
0.520480360097074	0.989281004866708\\
0.544877876976625	0.985354194985778\\
0.567648892730872	0.980762048849698\\
0.590419908485118	0.975245646641942\\
0.613190924239366	0.968810299774103\\
0.637588441118916	0.960934237439538\\
0.661985957998466	0.952109890588852\\
0.68800997600332	0.941744693684561\\
0.715660495133477	0.929761401223823\\
0.744937515388937	0.916078323859782\\
0.774214535644398	0.901411958553282\\
0.803491555899858	0.885746113421312\\
0.834395077280622	0.868145190962905\\
0.868551600911993	0.847615665173683\\
0.907587627919273	0.823032904973833\\
0.94987665717716	0.795288585564998\\
0.997045189810958	0.763203167936242\\
1.04746672469536	0.72775784286881\\
1.10276776295568	0.687739520933027\\
1.16620130684251	0.640668422221335\\
1.28005638561374	0.55473056144722\\
1.34674293175118	0.505080117708141\\
1.39391146438498	0.471103019768199\\
1.43294749139226	0.444082159573741\\
1.46873051614893	0.420412151220651\\
1.501260538655	0.399945568572744\\
1.53379056116107	0.38057718797283\\
1.56469408254183	0.36324505740832\\
1.59559760392259	0.346970072992991\\
1.62650112530336	0.331745388852635\\
1.65740464668412	0.317536133270489\\
1.68993466919019	0.30361008675739\\
1.72246469169625	0.290661026904334\\
1.75824771645293	0.277433466077186\\
1.79565724233491	0.264599041972706\\
1.83631977046749	0.251633960421486\\
1.88023530085068	0.238595976847831\\
1.92903033460978	0.225062180441332\\
1.9843313728701	0.210693664519792\\
2.04613841563162	0.195610293175774\\
2.11445146289436	0.179908219736322\\
2.18764401353301	0.164030210790425\\
2.26571606754758	0.148037041927016\\
2.34704112381274	0.13232053570874\\
2.42999268120322	0.117215739576824\\
2.51619724084429	0.102459915297402\\
2.60240180048537	0.088631575058804\\
2.69023286125175	0.0754767352272032\\
2.77806392201813	0.0632524290910692\\
2.86589498278451	0.0519562427528246\\
2.95209954242559	0.0417883083359887\\
3.03505109981606	0.0329044644709158\\
3.11637615608123	0.0251164732849141\\
3.19282170897049	0.0187076759866476\\
3.26438775848384	0.0136136256870363\\
3.33270080574658	0.00966128346495232\\
3.40101385300932	0.00661912319115654\\
3.47420640364797	0.00427877189060677\\
3.49860392052752	0.00367970204888612\\
};
\addlegendentry{Roe (polynomial)}

\addplot [color=gray, dashed, line width=1.0pt]
  table[row sep=crcr]{%
0	1\\
0.0771027807680795	0.99954447528139\\
0.104535901825061	0.998313649160207\\
0.130518813839382	0.996215836113058\\
0.156864278114369	0.993152978873684\\
0.184418249924905	0.989008393138935\\
0.213664132285211	0.983661302768039\\
0.244964477455951	0.976983854106339\\
0.278681837697792	0.968828869924315\\
0.315178765271397	0.959034503694169\\
0.355059513944541	0.947358696214674\\
0.39917003899211	0.933464093618299\\
0.448960549456763	0.916789245903527\\
0.506727209656046	0.896442097872659\\
0.577787452746376	0.870391745269504\\
0.679181234979007	0.832147276876555\\
0.91109383105102	0.744427391448602\\
1.01079570273388	0.707875384257327\\
1.10131291714656	0.675659670634294\\
1.18747950443127	0.645961584710764\\
1.27134992739843	0.618022646372547\\
1.35413269358359	0.591413096954199\\
1.43643205675453	0.565923781107178\\
1.51873141992548	0.541396377202302\\
1.60139333535709	0.517719955069027\\
1.68465950455647	0.494826021370296\\
1.76889247978429	0.47262047718248\\
1.85421311179411	0.451080568858952\\
1.94074225133948	0.430184815247647\\
2.02872159992751	0.409886111212394\\
2.11827200831175	0.390170193767125\\
2.20951432724576	0.371024534683978\\
2.30269025823665	0.352414799283548\\
2.39779980128443	0.334357757180457\\
2.49508465789619	0.316824553347169\\
2.59478652957905	0.299792243217558\\
2.69690541633301	0.283281335768593\\
2.80180387041873	0.2672550156859\\
2.90948189183623	0.251736155626304\\
3.02030203284614	0.236696063979537\\
3.13438514420205	0.22214381439588\\
3.25197292741105	0.208074040865512\\
3.37330708398025	0.194484359676172\\
3.49875016617032	0.181362646802284\\
3.49995867370587	0.181240619594812\\
};
\addlegendentry{LxF (atan)}

\addplot [color=gray, dashdotted, line width=1.0pt]
  table[row sep=crcr]{%
0	1\\
0.272631651063727	0.999544929654041\\
0.326778006849902	0.998425769592611\\
0.370475065905413	0.996632466457251\\
0.409422444628802	0.994145869275884\\
0.446469951219344	0.990871047702456\\
0.481617585677037	0.986864281647935\\
0.51676522013473	0.981932709693935\\
0.550962918525999	0.976222854312091\\
0.585160616917268	0.969607741492947\\
0.62030825137496	0.961873376878277\\
0.655455885832653	0.953209800314394\\
0.691553456356771	0.943373757512356\\
0.728600962947312	0.932327741086891\\
0.766598405604277	0.920042944203192\\
0.805545784327667	0.906498936499978\\
0.846393035183905	0.89131970187636\\
0.889140158172991	0.874437389797465\\
0.933787153294925	0.855799309255049\\
0.980334020549708	0.835371158206405\\
1.03068063207019	0.812251193959066\\
1.08482698785636	0.786356076701398\\
1.14562289610751	0.756234675467936\\
1.21781803715574	0.719403261781344\\
1.32991049299379	0.661040259133892\\
1.45530205376178	0.596032337083528\\
1.53034700300928	0.558161341634293\\
1.59589259159255	0.526091745676483\\
1.65668849984369	0.497362570042112\\
1.71463459989556	0.470997441050631\\
1.77068082781459	0.446506617161851\\
1.82577711966719	0.423436567650369\\
1.87992347545336	0.401753348084421\\
1.93406983123954	0.381053308664734\\
1.98821618702572	0.361327431604836\\
2.04331247887832	0.3422365938227\\
2.09840877073092	0.324108918264922\\
2.15445499864994	0.306623869427662\\
2.21145116263539	0.289791845668974\\
2.27034719875368	0.273359897218992\\
2.3301931709384	0.257614405714559\\
2.39193901525597	0.242320332702212\\
2.45558473170639	0.227508447711374\\
2.52113032028966	0.213203072973339\\
2.58857578100577	0.199422636370719\\
2.65887104992116	0.186005128135046\\
2.73106619096939	0.173158181462884\\
2.8061111402169	0.160733121119057\\
2.88400589766368	0.148762494829396\\
2.96570039937615	0.137143110858281\\
3.0502447092879	0.126046505212679\\
3.13858876346534	0.115377759254601\\
3.23073256190848	0.105175211507562\\
3.32762604068375	0.0953786028800754\\
3.42831926372471	0.0861212565809915\\
3.49956446870652	0.0800938447710373\\
};
\addlegendentry{LxF (polynomial)}

\addplot [color=red, line width=1.5pt]
  table[row sep=crcr]{%
0	1\\
0.640625	1\\
0.65625	0.958333333333333\\
0.671875	0.958333333333333\\
0.6875	0.953125\\
0.71875	0.953125\\
0.734375	0.932291666666667\\
0.75	0.916666666666667\\
0.765625	0.916666666666667\\
0.78125	0.90625\\
0.8125	0.90625\\
0.828125	0.869791666666667\\
0.84375	0.869791666666667\\
0.859375	0.864583333333333\\
0.875	0.854166666666667\\
0.890625	0.854166666666667\\
0.90625	0.848958333333333\\
0.921875	0.822916666666667\\
0.9375	0.8125\\
0.953125	0.807291666666667\\
0.96875	0.807291666666667\\
0.984375	0.802083333333333\\
1	0.765625\\
1.03125	0.755208333333333\\
1.046875	0.755208333333333\\
1.0625	0.75\\
1.078125	0.739583333333333\\
1.09375	0.713541666666667\\
1.109375	0.708333333333333\\
1.125	0.708333333333333\\
1.15625	0.697916666666667\\
1.171875	0.6875\\
1.1875	0.661458333333333\\
1.203125	0.651041666666667\\
1.21875	0.651041666666667\\
1.265625	0.635416666666667\\
1.28125	0.604166666666667\\
1.296875	0.598958333333333\\
1.3125	0.588541666666667\\
1.328125	0.588541666666667\\
1.34375	0.583333333333333\\
1.359375	0.557291666666667\\
1.375	0.541666666666667\\
1.390625	0.541666666666667\\
1.4375	0.526041666666667\\
1.453125	0.494791666666667\\
1.484375	0.484375\\
1.5	0.473958333333333\\
1.515625	0.46875\\
1.53125	0.458333333333333\\
1.546875	0.432291666666667\\
1.609375	0.411458333333333\\
1.625	0.390625\\
1.640625	0.375\\
1.6875	0.359375\\
1.703125	0.348958333333333\\
1.71875	0.34375\\
1.734375	0.34375\\
1.75	0.333333333333333\\
1.765625	0.317708333333333\\
1.796875	0.317708333333333\\
1.8125	0.3125\\
1.828125	0.302083333333333\\
1.84375	0.296875\\
1.859375	0.296875\\
1.875	0.291666666666667\\
1.890625	0.291666666666667\\
1.921875	0.270833333333333\\
1.953125	0.270833333333333\\
1.96875	0.260416666666667\\
2	0.25\\
2.03125	0.25\\
2.0625	0.229166666666667\\
2.109375	0.229166666666667\\
2.125	0.213541666666667\\
2.171875	0.213541666666667\\
2.203125	0.192708333333333\\
2.25	0.192708333333333\\
2.265625	0.182291666666667\\
2.28125	0.177083333333333\\
2.328125	0.177083333333333\\
2.34375	0.161458333333333\\
2.359375	0.161458333333333\\
2.375	0.15625\\
2.390625	0.15625\\
2.4375	0.140625\\
2.46875	0.140625\\
2.515625	0.125\\
2.53125	0.125\\
2.546875	0.119791666666667\\
2.5625	0.119791666666667\\
2.578125	0.109375\\
2.625	0.109375\\
2.640625	0.0989583333333335\\
2.65625	0.09375\\
2.6875	0.09375\\
2.71875	0.0833333333333335\\
2.75	0.0833333333333335\\
2.78125	0.0729166666666665\\
2.828125	0.0729166666666665\\
2.859375	0.0625\\
2.90625	0.0625\\
2.921875	0.0520833333333335\\
2.984375	0.0520833333333335\\
3	0.0416666666666665\\
3.0625	0.0416666666666665\\
3.078125	0.0364583333333335\\
3.109375	0.0364583333333335\\
3.125	0.03125\\
3.171875	0.03125\\
3.1875	0.0260416666666665\\
3.25	0.0260416666666665\\
3.265625	0.0208333333333335\\
3.3125	0.0208333333333335\\
3.328125	0.015625\\
3.40625	0.015625\\
3.421875	0.0104166666666665\\
3.4375	0.0104166666666665\\
};
\addlegendentry{Experiment}

\end{axis}
\end{tikzpicture}%

%% file: zzzLinf.tex
%
%
\begin{tikzpicture}
\setlength\fwidth{0.9\textwidth}
\begin{axis}[%
width=0.951\fwidth,
height=0.75\fwidth,
at={(0\fwidth,0\fwidth)},
scale only axis,
xmin=0,
xmax=3.5,
xlabel style={font=\color{white!15!black}},
xlabel={Time},
ymin=0,
ymax=4,
ylabel style={font=\color{white!15!black}},
ylabel={$\text{L}^\infty\text{ norm}$},
axis background/.style={fill=white},
legend style={at={(0.99,0.98)},anchor=north east, font=\footnotesize, legend cell align=left, align=left, draw=white!15!black, legend columns = 2}
]
\addplot [color=black, dashed, line width=1.0pt]
  table[row sep=crcr]{%
0	0.49215859092304\\
0.0149084196357339	0.50975692880737\\
0.0269771402932326	0.524089392288718\\
0.0478015994669563	0.553104535658943\\
0.0750153813417085	0.592359963912337\\
0.10246580479798	0.634241369046777\\
0.126603246112978	0.67434848386313\\
0.142221590493271	0.70066471565095\\
0.172038429764739	0.753606523607663\\
0.186710207818953	0.780351390409928\\
0.220549953976254	0.842472047092203\\
0.234985090448948	0.869858291835651\\
0.269771402932327	0.934790390683026\\
0.293435561084286	0.979148776142366\\
0.312603529187372	1.01277052090065\\
0.344076859529477	1.06670476508412\\
0.374603623545504	1.12075191455808\\
0.415305975566872	1.19187244178416\\
0.447962513816575	1.24803784951367\\
0.462634291870789	1.27409407106743\\
0.493161055886816	1.32661374769042\\
0.530313784185391	1.38991440471193\\
0.563680247179652	1.44532735909161\\
0.575512326255632	1.46609154740546\\
0.609825355575972	1.52476333160062\\
0.642245252244155	1.58006963711452\\
0.675375073656896	1.63565013497939\\
0.695016324923022	1.66924158566702\\
0.728146146335764	1.72522263766397\\
0.759856118259389	1.77869169550977\\
0.791329448601493	1.83057224340938\\
0.81594017307953	1.87135506831083\\
0.846940220258596	1.92158159099065\\
0.877230342693103	1.97050469883339\\
0.908230389872168	2.01928063458933\\
0.930711340116529	2.05500136700745\\
0.961238104132556	2.1025713364907\\
0.993421359219219	2.1523881083124\\
1.02252827374613	2.19594129061484\\
1.04382601608289	2.22810137423617\\
1.07198636428372	2.26968502373143\\
1.18817738080984	2.43134175122903\\
1.20048274304886	2.44799464731214\\
1.22485682594537	2.48009027710794\\
1.24781105935277	2.50703579222035\\
1.25727672261356	2.51891900054889\\
1.28023095602096	2.54649567702233\\
1.30957451212939	2.58098637211551\\
1.33418523660742	2.61810355248558\\
1.35737611159634	2.64945775466533\\
1.37914713709614	2.67554485166737\\
1.39949831310683	2.69684348605047\\
1.41890292279143	2.71423057870237\\
1.43736096614996	2.72803990212665\\
1.44942968680746	2.73732192914199\\
1.46788773016599	2.75099299397864\\
1.48563584877996	2.76160461003986\\
1.50291068423089	2.76953655670294\\
1.51971223651878	2.77497847620143\\
1.53627714722515	2.7781497465092\\
1.55284205793152	2.77915001731678\\
1.56917032705637	2.7780088251905\\
1.58573523776274	2.77471731946954\\
1.60277343163215	2.7691421225496\\
1.6202849086646	2.76116076237262\\
1.63850631044161	2.75052709989474\\
1.65791092012621	2.73673183731794\\
1.70026976321822	2.70256899401438\\
1.72204078871802	2.68285220666698\\
1.74499502212542	2.65920320706868\\
1.76865918027738	2.63182289568858\\
1.79326990475542	2.60017272048345\\
1.81930047872257	2.56329414397683\\
1.84675090217884	2.52076224925933\\
1.87609445828727	2.47144275731102\\
1.9078044302109	2.41405835969844\\
1.91963650928688	2.39396664611998\\
1.94968999013986	2.34936583333772\\
1.98187324522653	2.29787919959935\\
2.01689619929143	2.23790688220556\\
2.05688862656824	2.16528362182319\\
2.10137724389392	2.08023295774254\\
2.22821713158842	1.82965773142385\\
2.28785081013135	1.7108742620515\\
2.3609730588209	1.58120735933527\\
2.41208764042913	1.49341736225646\\
2.43385866592893	1.48753461572004\\
2.46864497841231	1.48013691043811\\
2.49656868503163	1.47343454738064\\
2.52520231639549	1.46441792514438\\
2.55478251408544	1.45292868072683\\
2.58578256126451	1.43865472041141\\
2.61843909951421	1.42131790127679\\
2.65298877041607	1.40061078253567\\
2.68966821555161	1.37619969050028\\
2.72871407650234	1.34772969118065\\
2.77083627801282	1.314444408823\\
2.81627146166458	1.27589403058331\\
2.8659661937837	1.23100219523853\\
2.9215769654408	1.17795291862829\\
2.9916228735706	1.10811865533911\\
3.09053905464578	1.0095874118594\\
3.1184627612651	0.98516678822582\\
3.16034832119406	0.948436966190983\\
3.19158500995465	0.918066987115121\\
3.22495147294891	0.88271004905776\\
3.26328740915508	0.839023949261954\\
3.30990580071444	0.782673709949703\\
3.37427231088777	0.701416234373668\\
3.49992899067467	0.541601630422329\\
};
\addlegendentry{Roe (atan)}

\addplot [color=black, dashdotted, line width=1.0pt]
  table[row sep=crcr]{%
0	0.49215859092304\\
0.013012009002427	0.509188425651198\\
0.0227710157542469	0.522018583319457\\
0.0341565236313706	0.536796925709953\\
0.0618070427615276	0.567038157440522\\
0.0666865461374377	0.572155640722769\\
0.078072054014561	0.585396569965464\\
0.0894575618916846	0.595270631459308\\
0.102469570894112	0.608380772260229\\
0.113855078771235	0.619950680590069\\
0.125240586648359	0.628266081407057\\
0.139879096776089	0.63591743249999\\
0.146385101277302	0.63941591092444\\
0.159397110279729	0.648953372954516\\
0.174035620407459	0.656989303459635\\
0.1935536339111	0.66507976971472\\
0.216324649665347	0.67262878011348\\
0.240722166544897	0.678566516428166\\
0.265119683424447	0.684320248895295\\
0.284637696928087	0.692065714844995\\
0.309035213807638	0.699358632154546\\
0.348071240814918	0.708447161678473\\
0.398492775699323	0.717479320201042\\
0.468432324087367	0.729147549284145\\
0.505841849969344	0.737290938799111\\
0.557889885979051	0.745438538890094\\
0.613190924239366	0.752089848023731\\
0.639214942244219	0.75546987010932\\
0.691262978253927	0.762851097977466\\
0.740058012013027	0.768417181925128\\
0.892949117791543	0.779802434169072\\
0.964515167304891	0.781369008022731\\
1.05397272919658	0.780421067892754\\
1.1157797719581	0.778885779200233\\
1.18734582147145	0.775011088787164\\
1.28168288673905	0.76816413411645\\
1.41180297676331	0.756641509949407\\
1.45734500827181	0.753266336088262\\
1.6476456399323	0.73999968827771\\
1.75174171195172	0.729991996350132\\
1.81029575246264	0.722250048788482\\
1.82981376596628	0.720215074401475\\
1.9404158424869	0.713633651227301\\
2.0347529077545	0.705966711139527\\
2.08842744488951	0.69955681504467\\
2.19089701578362	0.686032523150152\\
2.26896906979818	0.680036868483222\\
2.31776410355728	0.674348260862311\\
2.44300469020564	0.655711990411365\\
2.48041421608762	0.648314637648418\\
2.49342622509005	0.646407298544425\\
2.53571525434793	0.641644428846227\\
2.568245276854	0.636098411882641\\
2.60728130386128	0.629064088022655\\
2.63981132636735	0.622406432085338\\
2.77643742089283	0.590521568327507\\
2.80571444114829	0.585093808381574\\
2.83011195802784	0.578641653136763\\
2.85288297378209	0.570549571983603\\
2.86426848165921	0.566843846818177\\
2.89354550191467	0.559279090478917\\
2.98462956493166	0.530891574521339\\
3.0057740795606	0.522400610783737\\
3.03342459869076	0.510555405665362\\
3.04806310881849	0.506606559522075\\
3.08547263470047	0.50324569166479\\
3.11474965495593	0.498701132663589\\
3.1489061785873	0.492865046025487\\
3.22047222810065	0.489193055202112\\
3.26601425960914	0.484921054228147\\
3.30179728436581	0.479668074030049\\
3.33432730687188	0.472937917915606\\
3.36523082825265	0.464540241193661\\
3.39613434963341	0.454080361683302\\
3.42866437213948	0.440873258719268\\
3.46282089577085	0.424724245203693\\
3.49860392052752	0.405499522266342\\
};
\addlegendentry{Roe (polynomial)}

\addplot [color=gray, dashed, line width=1.0pt]
  table[row sep=crcr]{%
0	0.49215859092304\\
0.00326297034598477	0.443125917762758\\
0.00894295576306892	0.396587774730812\\
0.0155897472085931	0.363440057289184\\
0.022719941668337	0.340752404334328\\
0.0282790763318661	0.328679983931311\\
0.0350467185309453	0.317509426694266\\
0.0421769129906893	0.308812257127758\\
0.0529326300570827	0.299246935807202\\
0.0633257948628114	0.292438781087726\\
0.079761497346289	0.284054360927906\\
0.11202864854547	0.271072264864945\\
0.147317068583525	0.259308940522303\\
0.182001234853805	0.249773187403345\\
0.224298998598049	0.240360591824034\\
0.278319285437127	0.231397714222748\\
0.344666349138813	0.223481353809837\\
0.413551278665153	0.217364588093367\\
0.51264889658024	0.210525178924339\\
0.746495104709132	0.197260989976445\\
1.04137094338329	0.179286257592275\\
1.3809615608728	0.156396467230093\\
1.84128208116373	0.12619875411471\\
2.15174766704648	0.107937145429743\\
2.44831541627041	0.0925011784980274\\
2.75310101673608	0.0786614321342003\\
3.07105934933924	0.0662509728017664\\
3.40919975778608	0.0550847842179829\\
3.49995867370587	0.0524090394624266\\
};
\addlegendentry{LxF (atan)}

\addplot [color=gray, dashdotted, line width=1.0pt]
  table[row sep=crcr]{%
0	0.49215859092304\\
0.00759948853139303	0.479748482835948\\
0.013299104929938	0.473532802008417\\
0.0246983377270276	0.46600700001876\\
0.0341976983912691	0.460266759216479\\
0.0636457164504169	0.450276906371847\\
0.0940436705759895	0.445550413720783\\
0.131091177166531	0.446552380072168\\
0.177638044421313	0.452306278912931\\
0.230834464141065	0.462845840849618\\
0.287830628126513	0.477151787733379\\
0.347676600311233	0.494263076928526\\
0.417021933160195	0.51529103323562\\
0.457869184016433	0.528052559607247\\
0.593710041515085	0.56568300812775\\
0.707702369485981	0.59118039435132\\
0.806495720394091	0.609051883934189\\
0.859692140113843	0.617146015554648\\
0.919538112298563	0.624826195101051\\
0.98318382874898	0.631181940568732\\
1.05157922553152	0.636005533505604\\
1.12757411084545	0.638832030307168\\
1.18837001909659	0.638941769429974\\
1.27861394540689	0.636630023875994\\
1.37170767991645	0.6313526721754\\
1.44580269309753	0.624816651655275\\
1.5426961718728	0.613948749656481\\
1.61869105718673	0.603260171860359\\
1.71178479169629	0.587885670721984\\
1.78113012454525	0.574388312107775\\
1.84667571312852	0.55958053217776\\
1.90937149351251	0.543327228137419\\
1.97206727389651	0.524961761497269\\
2.04426241494474	0.501584846933737\\
2.18010327244339	0.457115575070982\\
2.35679138079828	0.400522449041986\\
2.48503274976554	0.36158714925662\\
2.58952571707219	0.331845117887771\\
2.70541791717594	0.300889513389824\\
2.81846030908041	0.27277269633257\\
2.93150270098488	0.24672374563006\\
3.0464449650222	0.222321269005424\\
3.16328710119237	0.199595649430882\\
3.28297904556181	0.178386817011525\\
3.40647073419695	0.158572672989367\\
3.49956446870652	0.144938199036587\\
};
\addlegendentry{LxF (polynomial)}

\end{axis}
\end{tikzpicture}%

%% file: zzzRoe_convergence.tex
%
%
\begin{tikzpicture}
\setlength\fwidth{0.45\textwidth}
\begin{axis}[%
width=0.951\fwidth,
height=0.75\fwidth,
at={(0\fwidth,0\fwidth)},
scale only axis,
xmin=0,
xmax=3.5,
xlabel={Time},
ymin=0,
ymax=1.2,
ylabel={Amount of parts in front of obstacle},
axis background/.style={fill=white},
legend style={font=\footnotesize,legend cell align=left, align=left, draw=white!15!black}
]
\addplot [color=black, dotted,line width=1.5pt]
  table[row sep=crcr]{%
0	1\\
0.000947190945598475	0.991397750937443\\
0.279421328951557	0.990801439682802\\
0.341935931361058	0.989836592004926\\
0.39118986053218	0.988331017081748\\
0.433813453084112	0.986272079912441\\
0.472648281853651	0.983635708480426\\
0.508641537786394	0.980443468373953\\
0.543687602773538	0.976571069748008\\
0.577786476815084	0.972030195912565\\
0.610938159911032	0.966853661724119\\
0.644089843006979	0.960910098291821\\
0.677241526102926	0.954196240290897\\
0.710393209198874	0.946719922175019\\
0.74449208324042	0.938252840419121\\
0.779538148227564	0.928757946875237\\
0.815531404160307	0.918208781794285\\
0.852471851038648	0.906589549238728\\
0.891306679808187	0.893567925541101\\
0.932035890468923	0.879093844166605\\
0.974659483020855	0.863133432587592\\
1.02012464840958	0.84528998387915\\
1.0693785775807	0.825122698114437\\
1.12242127053422	0.802569439795782\\
1.18114710916133	0.776761195381204\\
1.24839766629882	0.746355094877408\\
1.33269766045709	0.707353634844003\\
1.56949539685671	0.597273313440694\\
1.63390438115741	0.568495091367856\\
1.69168302883892	0.543510666763832\\
1.74472572179243	0.521392714241492\\
1.79587403285475	0.500887570805176\\
1.84512796202587	0.481955343878372\\
1.8934347002514	0.464190596942214\\
1.94174143847692	0.447233179737934\\
1.99004817670244	0.431081192451533\\
2.03930210587357	0.415427822294674\\
2.08855603504469	0.400573464428134\\
2.13875715516141	0.386222126323345\\
2.19085265716933	0.372127862828346\\
2.24484254106844	0.358326496979864\\
2.30072680685875	0.344842390461329\\
2.35850545454026	0.331688896263182\\
2.41912567505856	0.318671707993794\\
2.48353465935926	0.305632371220274\\
2.55173240744235	0.292618401019174\\
2.62371891930784	0.279662524044359\\
2.70138857684692	0.266472098211427\\
2.78379418911399	0.253256558648936\\
2.87283013800025	0.239760110697286\\
2.96849642350569	0.226045026474467\\
3.06984585468473	0.212291169461159\\
3.17782562248296	0.198413184940999\\
3.29148853595478	0.184580423114014\\
3.40988740415459	0.170944410678448\\
3.49987054398645	0.161074939325254\\
};
\addlegendentry{$\text{Space step 4}\cdot\text{10}^{\text{-2}}$}

\addplot [color=black, dashdotted, line width=1.0pt]
  table[row sep=crcr]{%
0	1\\
0.000473283163039184	0.995692803610164\\
0.357802071257613	0.993690933728466\\
0.413176201333196	0.992546580385688\\
0.454825119680643	0.990947449586455\\
0.490321356908581	0.988837393881124\\
0.522504611995244	0.986166279639685\\
0.552321451266712	0.982937225063858\\
0.581191724212101	0.979042284425777\\
0.609115430831413	0.974506115974164\\
0.636565854287685	0.969282167168977\\
0.664016277743956	0.963290877404191\\
0.691939984363267	0.956416158228669\\
0.720336974145618	0.948640362495462\\
0.749680530254047	0.939811141918472\\
0.779970652688553	0.929900453190875\\
0.811680624612178	0.918720444057775\\
0.84481044602492	0.906232329998121\\
0.879833400089818	0.892212042053687\\
0.916749486806874	0.876603684015288\\
0.955558706176086	0.859352202175151\\
0.995787775034415	0.840621777355074\\
1.03696341021882	0.820604477041393\\
1.07908561172931	0.799275646059582\\
1.12215437956587	0.776609286583191\\
1.16664299689156	0.752331626437806\\
1.21397131319547	0.725626668056145\\
1.26650574429282	0.695087918412436\\
1.33039897130311	0.657016768200926\\
1.5537886242576	0.52308898758816\\
1.60584977219191	0.493261267769161\\
1.65270480533278	0.467284912960142\\
1.69624685633239	0.444007206234505\\
1.73789577467983	0.422602988712991\\
1.77812484353816	0.402784723969828\\
1.81740734607042	0.384280065939401\\
1.85621656543963	0.366838736410406\\
1.89502578480884	0.350236582804464\\
1.93383500417805	0.334465088703342\\
1.9731175067103	0.319327444649349\\
2.01287329240559	0.304823639480092\\
2.05357564442696	0.290785154739952\\
2.09569784593745	0.277069900643538\\
2.13923989693705	0.263700380093816\\
2.18467508058881	0.250553772555802\\
2.23247668005577	0.23752657078453\\
2.28264469533792	0.224649235700627\\
2.33612569276135	0.211715966693232\\
2.39339295548909	0.198664984898249\\
2.45444648352114	0.185543718446441\\
2.51975956002054	0.172295601093766\\
2.5893321849873	0.158968641320857\\
2.66269107525837	0.145695167186393\\
2.73936294767072	0.132597256324943\\
2.81840123589826	0.119867356213933\\
2.89885937361492	0.107678719254631\\
2.98026407765766	0.0961194259886686\\
3.06166878170039	0.0853317839802363\\
3.14260020258009	0.0753729820408289\\
3.22353162345979	0.0661810650785841\\
3.30446304433949	0.0577567812067863\\
3.38586774838223	0.0500530626667874\\
3.467745735588	0.0430712587879709\\
3.49992899067467	0.0405331981643093\\
};
\addlegendentry{$\text{Space step 2}\cdot\text{10}^{\text{-2}}$}

\addplot [color=black, dashed, line width=1.0pt]
  table[row sep=crcr]{%
0	1\\
0.000473283163039184	0.997846056835621\\
0.434947226832998	0.997461933879262\\
0.483222109462993	0.996570158958701\\
0.51919162985397	0.995160474316456\\
0.549718393869997	0.993207115802448\\
0.577168817326269	0.990688816541875\\
0.602962749711903	0.987551304939426\\
0.62757347418994	0.983787517143068\\
0.651710915504938	0.979322401840354\\
0.675848356819936	0.97407313457281\\
0.700222439716453	0.967981213267439\\
0.725069805776009	0.960977932584009\\
0.750863738161645	0.952906199628595\\
0.777840878454877	0.943655547395906\\
0.806474509818747	0.933017473187273\\
0.837237915416293	0.920753494448084\\
0.870367736829035	0.906703436356755\\
0.906337257220012	0.890598974497493\\
0.945383118170743	0.872262511293649\\
0.987741961262749	0.851508163505399\\
1.03270386175147	0.828618271401216\\
1.07955889489235	0.803902337101392\\
1.12712385277778	0.777946465261408\\
1.17492545224474	0.750989026530999\\
1.22343697645626	0.722746112006918\\
1.27431491648297	0.692229077395815\\
1.32637606441728	0.660103508509797\\
1.3822234776559	0.624717481116529\\
1.4475365541553	0.582391330266776\\
1.63069713825146	0.463091881637657\\
1.67613232190322	0.434807228750111\\
1.71636139076155	0.410648044517952\\
1.75398740222317	0.388938904491048\\
1.78995692261414	0.369067377860376\\
1.82497987667904	0.350590095244384\\
1.85976618916242	0.333103279633425\\
1.89478914322732	0.316360305354523\\
1.93028538045526	0.300245370034543\\
1.96672818400927	0.284548399013821\\
2.00435419547089	0.269179978651322\\
2.04363669800314	0.253966624955953\\
2.08504897476907	0.238758644694698\\
2.12882766735019	0.223508790086382\\
2.17497277574651	0.20825220690182\\
2.22372094153954	0.192947807709472\\
2.27483552314777	0.1777101558144\\
2.32784323740816	0.16271291128748\\
2.38203415957615	0.148176644928328\\
2.43717164807021	0.134179613290118\\
2.49254577814579	0.120910280495828\\
2.54815654980289	0.108371086132477\\
2.60376732146	0.0966173661181928\\
2.6593780931171	0.0856466747129199\\
2.7149888647742	0.0754562606804683\\
2.77083627801282	0.066000194529384\\
2.82692033283297	0.0572800483540377\\
2.88300438765311	0.0493341615501279\\
2.93932508405477	0.0421297384704724\\
2.99564578045643	0.0356964478081596\\
3.05267640160265	0.0299539510299396\\
3.11112687223799	0.0248433604977629\\
3.171707117107	0.0203195949780164\\
3.23583698569881	0.0163056451504295\\
3.30469968592101	0.0127710272694008\\
3.37995170884424	0.00968509126504857\\
3.46372282870217	0.00703037984941179\\
3.49992899067467	0.00609501108645105\\
};
\addlegendentry{$\text{Space step 1}\cdot\text{10}^{\text{-2}}$}

\addplot [color=black, line width=1.0pt]
  table[row sep=crcr]{%
0	1\\
0.000473283163039184	0.998921761079101\\
0.478489277832602	0.997722781688398\\
0.520374837761568	0.996812751501996\\
0.550664960196075	0.995405055760068\\
0.57622225100019	0.99345612317738\\
0.59929480519835	0.990929582239746\\
0.621184151488912	0.987755352750346\\
0.642481893825674	0.983885431290938\\
0.663897956953197	0.979203086202006\\
0.685905624034518	0.973590087825194\\
0.708978178232678	0.966896496566091\\
0.733707223501475	0.958903997145609\\
0.760566043003947	0.949398246279733\\
0.790027919903136	0.93813777240154\\
0.82197453340828	0.925093061246114\\
0.85628756272862	0.910245289919868\\
0.892612045491876	0.89368601405898\\
0.93035637774425	0.875631972150829\\
0.968455672368903	0.856555935033268\\
1.00620000462128	0.836801064027499\\
1.04429929924593	0.815990951934635\\
1.0838184433597	0.793523183794808\\
1.12546736170715	0.768952673084816\\
1.16936437507903	0.742157257184438\\
1.21586444584763	0.712864098335805\\
1.26555917796674	0.680638673083193\\
1.31986842092549	0.644489974045\\
1.38317004398198	0.60140083975551\\
1.4813763003126	0.533476236962207\\
1.56715887361345	0.474496388144012\\
1.6185100968032	0.440107389097755\\
1.66051397752293	0.412875950844277\\
1.69743006423999	0.389840181105745\\
1.73103316881577	0.369763324838091\\
1.76250649915787	0.351844214720327\\
1.79255998001086	0.335608940485891\\
1.82190353611929	0.320626065136616\\
1.8510104506462	0.306626725068312\\
1.88023568596387	0.293422573021531\\
1.91017084602609	0.280744699597488\\
1.94117089320516	0.268455595925813\\
1.97382743145486	0.256343978481933\\
2.00885038551976	0.244187582972722\\
2.04683135935365	0.231831224886924\\
2.08883524007338	0.218989921768566\\
2.13569027321426	0.205483562524984\\
2.18822470431161	0.191153093701359\\
2.24632021257466	0.176112112386883\\
2.30891191088659	0.160705833479044\\
2.37457994975828	0.145334539879467\\
2.44225944207288	0.130281689222687\\
2.51100382150432	0.115777850939092\\
2.5802214840988	0.101956318274794\\
2.64955746748404	0.0888917003606116\\
2.718538488497	0.0766717570880342\\
2.78704622634692	0.0653122626612874\\
2.85472571866152	0.0548637768292273\\
2.92134032385928	0.0453517694972314\\
2.98618011719565	0.0368618779595593\\
3.04889013629834	0.0294173629378061\\
3.10935206037659	0.023005024648441\\
3.16768421022117	0.0175830133260093\\
3.22400490662283	0.0131117325134107\\
3.27902407432613	0.00951182573608911\\
3.33427988361096	0.00667014440826552\\
3.39249371266478	0.00445724326775121\\
3.45768846837342	0.00276733304071186\\
3.49992899067467	0.00201221683676289\\
};
\addlegendentry{$\text{Space step 5}\cdot\text{10}^{\text{-3}}$}

\addplot [color=red, line width=1.5pt]
  table[row sep=crcr]{%
0	1\\
0.640625	1\\
0.65625	0.958333333333333\\
0.671875	0.958333333333333\\
0.6875	0.953125\\
0.71875	0.953125\\
0.734375	0.932291666666667\\
0.75	0.916666666666667\\
0.765625	0.916666666666667\\
0.78125	0.90625\\
0.8125	0.90625\\
0.828125	0.869791666666667\\
0.84375	0.869791666666667\\
0.859375	0.864583333333333\\
0.875	0.854166666666667\\
0.890625	0.854166666666667\\
0.90625	0.848958333333333\\
0.921875	0.822916666666667\\
0.9375	0.8125\\
0.953125	0.807291666666667\\
0.96875	0.807291666666667\\
0.984375	0.802083333333333\\
1	0.765625\\
1.03125	0.755208333333333\\
1.046875	0.755208333333333\\
1.0625	0.75\\
1.078125	0.739583333333333\\
1.09375	0.713541666666667\\
1.109375	0.708333333333333\\
1.125	0.708333333333333\\
1.15625	0.697916666666667\\
1.171875	0.6875\\
1.1875	0.661458333333333\\
1.203125	0.651041666666667\\
1.21875	0.651041666666667\\
1.265625	0.635416666666667\\
1.28125	0.604166666666667\\
1.296875	0.598958333333333\\
1.3125	0.588541666666667\\
1.328125	0.588541666666667\\
1.34375	0.583333333333333\\
1.359375	0.557291666666667\\
1.375	0.541666666666667\\
1.390625	0.541666666666667\\
1.4375	0.526041666666667\\
1.453125	0.494791666666667\\
1.484375	0.484375\\
1.5	0.473958333333333\\
1.515625	0.46875\\
1.53125	0.458333333333333\\
1.546875	0.432291666666667\\
1.609375	0.411458333333333\\
1.625	0.390625\\
1.640625	0.375\\
1.6875	0.359375\\
1.703125	0.348958333333333\\
1.71875	0.34375\\
1.734375	0.34375\\
1.75	0.333333333333333\\
1.765625	0.317708333333333\\
1.796875	0.317708333333333\\
1.8125	0.3125\\
1.828125	0.302083333333333\\
1.84375	0.296875\\
1.859375	0.296875\\
1.875	0.291666666666667\\
1.890625	0.291666666666667\\
1.921875	0.270833333333333\\
1.953125	0.270833333333333\\
1.96875	0.260416666666667\\
2	0.25\\
2.03125	0.25\\
2.0625	0.229166666666667\\
2.109375	0.229166666666667\\
2.125	0.213541666666667\\
2.171875	0.213541666666667\\
2.203125	0.192708333333333\\
2.25	0.192708333333333\\
2.265625	0.182291666666667\\
2.28125	0.177083333333333\\
2.328125	0.177083333333333\\
2.34375	0.161458333333333\\
2.359375	0.161458333333333\\
2.375	0.15625\\
2.390625	0.15625\\
2.4375	0.140625\\
2.46875	0.140625\\
2.515625	0.125\\
2.53125	0.125\\
2.546875	0.119791666666667\\
2.5625	0.119791666666667\\
2.578125	0.109375\\
2.625	0.109375\\
2.640625	0.0989583333333335\\
2.65625	0.09375\\
2.6875	0.09375\\
2.71875	0.0833333333333335\\
2.75	0.0833333333333335\\
2.78125	0.0729166666666665\\
2.828125	0.0729166666666665\\
2.859375	0.0625\\
2.90625	0.0625\\
2.921875	0.0520833333333335\\
2.984375	0.0520833333333335\\
3	0.0416666666666665\\
3.0625	0.0416666666666665\\
3.078125	0.0364583333333335\\
3.109375	0.0364583333333335\\
3.125	0.03125\\
3.171875	0.03125\\
3.1875	0.0260416666666665\\
3.25	0.0260416666666665\\
3.265625	0.0208333333333335\\
3.3125	0.0208333333333335\\
3.328125	0.015625\\
3.40625	0.015625\\
3.421875	0.0104166666666665\\
3.4375	0.0104166666666665\\
};
\addlegendentry{Experiment}

\end{axis}
\end{tikzpicture}%